\documentclass[11pt,reqno,oneside]{amsart}
\usepackage{amscd}
\usepackage{amsmath}
\usepackage{latexsym}
\usepackage{amsfonts}
\usepackage{amssymb}
\usepackage{amsthm}
\usepackage{graphicx}
\usepackage{verbatim}
\usepackage{mathrsfs}
\usepackage{enumerate}
\usepackage{hyperref}
\usepackage{fullpage}
\usepackage{color}
\theoremstyle{plain}\newtheorem{definition}{Definition}[section]
\theoremstyle{definition}\newtheorem{theorem}{Theorem}[section]
\theoremstyle{plain}\newtheorem{lemma}[theorem]{Lemma}
\theoremstyle{plain}\newtheorem{corollary}[theorem]{Corollary}
\theoremstyle{plain}\newtheorem{proposition}[theorem]{Proposition}
\theoremstyle{remark}\newtheorem{remark}{Remark}[section]

\newcommand{\norm}[1]{\left\|#1\right\|}
\DeclareMathOperator{\dtd}{\frac{d}{d{\emph t}}}
\newcommand{\intd }{\,{\rm d}}

\newcommand{\NN}{\mathbb{N}}

\newcommand{\RR}{\mathbb{R}}

\DeclareMathOperator{\Div}{div}
\DeclareMathOperator{\Avg}{Avg}
\allowdisplaybreaks
\begin{document}
\title[The two-dimensional Euler equations]{The two-dimensional Euler equation in Yudovich and $\mathrm{\textbf{bmo}}$-type spaces }
\author[Q. Chen, X. Miao and X. Zheng]{Qionglei Chen$^{1}$, Changxing Miao$^{2}$ and  Xiaoxin Zheng$^{3}$}
\address{$^1$ Institute of Applied Physics and Computational Mathematics,
        P.O. Box 8009, Beijing 100088, P.R. China.}

\email{chen\_{}qionglei@iapcm.ac.cn}

\address{$^2$ Institute of Applied Physics and Computational Mathematics,
        P.O. Box 8009, Beijing 100088, P.R. China.}

\email{miao\_{}changxing@iapcm.ac.cn}

\address{$^3$ School of Mathematics and Systems Science, Beihang University, Beijing 100191, P.R. China}

\email{xiaoxinzheng@buaa.edu.cn}
\date{\today}

\begin{abstract}

We construct global-in-time, unique solutions of the two-dimensional Euler equations in a Yudovich type space
 and a $\rm bmo$-type space.  First, we  show the  regularity of solutions for the
two-dimensional Euler equations in the Spanne space involving an unbounded
and non-decaying vorticity. Next, we establish an estimate with a
logarithmic loss of regularity for the transport equation in a
bmo-type space by developing classical analysis tool such as the
John-Nirenberg inequality.  We also  optimize estimates  of solutions to the vorticity-stream formulation of the
two-dimensional Euler equations with a bi-Lipschitz vector field in
bmo-type space by combining an observation introduced in \cite{Y1} by Yodovich with
the so-called ``quasi-conformal property" of the incompressible

\end{abstract}

\maketitle
\section{Introduction}\label{INTR}
\setcounter{section}{1}\setcounter{equation}{0}

The two-dimensional incompressible Euler equations have the following form
\begin{equation*}\label{eq.Euler}
\left\{\begin{array}{ll}
\partial_tu+(u\cdot\nabla) u+\nabla\Pi=0,\quad (t,x)\in\RR^+\times\RR^2,\\
\Div u=0,\\
u|_{t=0}=u_0.
\end{array}\right.
\tag{E}
\end{equation*}
Here, $u=(u_1,u_2)(t,x_1,x_2)$ denotes the velocity vector-field.  The scalar function $\Pi$ stands for the pressure which can be recovered  at least formally from $u$  via Calder\'on-Zygmund operators, namely,
\begin{equation}\label{eq-pressure}
\Pi=\sum^{2}_{i,\,j=1}\frac{\partial_{x_i}\partial_{x_j}}{-\Delta}(u_i\,u_j).
\end{equation}
Due to its physical importance, there are numerous studies on the two-dimensional incompressible Euler equations by many physicists and mathematicians.
In particular, local-in-time  smooth solutions with the large initial data were obtained in different type function spaces such as $H^s(\RR^2)$, $B^s_{p,q}(\RR^2)$ and $F^s_{p,q}(\RR^2)$, etc. Since the vorticity $\omega$ satisfies the free transport equation
\begin{equation}\label{eq.vor}
\partial_t\omega+(u\cdot\nabla)\omega=0,\quad \Div u=0,\quad \omega|_{t=0}=\omega_0,
\end{equation}
the incompressible condition guarantees that the quantity $\|\omega(t)\|_{L^\infty(\RR^2)}$ is conserved. This observation together with the Beale-Kato-Majda (abbr. B-K-M) criterion established in \cite{BKM}, and the logarithmic Sobolev inequality entail the global existence and uniqueness of smooth solutions for the general initial data in subcritical functional spaces such as $B^s_{p,q}(\RR^2)$ with $s>\frac2p+1$, see for example \cite{KP86,MAj-book,MWZ2012,Wol}. However, the B-K-M criterion does not work in the critical framework because the logarithmic Sobolev inequality is not available. In order to overcome this difficulty, Vishik~\cite{Vi98} established the following logarithmic estimate for the vorticity equation
\begin{equation}\label{eq.Vishik-log}
\|\omega(t)\|_{B^0_{\infty,1}(\RR^2)}\leq C\Big(1+\int_0^t\|\nabla u(\tau)\|_{L^\infty(\RR^2)}\intd\tau\Big)\|\omega_0\|_{B^0_{\infty,1}(\RR^2)},
\end{equation}
 (also was proved in \cite{HK08}) which allowed him to construct  global-in-time solutions in the critical Besov space $B^{\frac2p+1}_{p,1}(\RR^2)$.
 In addition, by making a proper use of the structure of the equations,
  Chemin \cite{Chemin} obtained  a notable existence and uniqueness result without any integrable conditions.
  One easily shows  that the above theory for the two-dimensional Euler equations can be obtained in the context of Lipschitz vector fields.
  It follows from the structure of the system that three fundamental properties of solutions: the global-in-time existence, the uniqueness,
   and the regularity persistence, hold true as far as we can guarantee $u$ to be Lipschitz continuous in the space variable $x$.
     However, such results are  much more complex for weak solution.
     Here, we recall  an existence result on a weak solution  from \cite{Dip87} corresponding to
      $\omega_0\in L^1(\RR^2)\cap L^p(\RR^2)$ with $p\in]2,\infty[$, where a  weak solution is defined in the following way:

 \begin{definition}\label{def-weak-solution}
A  function $u\in L^2_{\rm loc}(\RR^+\times\RR^2)$ is a weak solution of Problem \eqref{eq.Euler}  if the following two conditions holds true:

{\rm(1)}\; \label{chae-weak-1} for every
  $\phi=(\phi_1,\phi_2)\in C_0^\infty(\RR^+\times\RR^2)$ such that $\Div\phi=0$,
\begin{equation*}
\int_0^\infty\int_{\RR^2}\big(\phi_t\cdot u+\nabla\phi:(u\otimes u)\big)\intd x\intd t=0;
\end{equation*}

{\rm(2)}\;\label{chae-weak-2} $\int_0^\infty\int_{\RR^2}\nabla\psi\cdot u\intd x\intd t=0,\quad \text{for every }\,\psi\in C^\infty_0(\RR^+\times\RR^2),$

\noindent where $u\otimes u=(u_iu_j),\,\nabla\phi=(\partial_{x_{j}}\phi_i)$ and $A:B=\sum_{i,j=1}^2A_{ij}B_{ij}$.
\end{definition}
Giga,  Miyakawa and  Osada \cite{GMO88} established the global existence of the weak solution to Problem \eqref{eq.vor}
with $\omega_0\in L^p(\mathbb R^2)$, $p\in]2,\infty[$. Next,  Chae \cite{Ch} showed a global existence result for \eqref{eq.vor}  with
$\omega_0\in {L\log^+L}$ of compact support which can be viewed as a variation of a function from $L^1(\RR^2)$. The papers \cite{Del91} and \cite{FLX} were concerned with measure-valued solutions to the two-dimensional Euler equations.  Taniuchi \cite{Tan04} proved a global existence result for $(u_0,\omega_0)\in L^\infty\times{\rm bmo}$ by establishing a local uniformly $L^p(\RR^2)$ estimate for the vorticity and using the continuity argument, see also \cite{Tan10} for the case of  spatially almost periodic initial data.

Regularity persistence and uniqueness are also hot topics in the study of the two-dimensional Euler equations with non-Lipschitzian vector field.
 Yudovich \cite{Y-1995} proved a uniqueness of solution under the assumption that $\omega(t,x)\in Y^\Theta$ with $\Theta\in\mathcal{A}_2$ (see Definition \ref{def-yu} below).
Other interesting results on uniqueness can be found in \cite{Vi00,Ser-1}. As for the problem of propagation of regularities,  Vishik \cite{Vi99}  showed that $\omega(t,x)$ belongs to the class $B_{\Gamma_2}$ with $\Gamma_2(n)=n\Gamma_1(n)$ under the assumption  $\omega_0\in B_{\Gamma_1}$, where $B_{\Gamma_i}$ ($i=1,\,2$) are defined by
the relation
\begin{equation*}
\|f\|_{B_{\Gamma_i}}:=\sup_{n\geq2}\frac{1}{\Gamma_i(n)}\sum_{k=-1}^n\|\Delta_{k}f\|_{L^\infty}<\infty,
\end{equation*}
see \cite[Def.1.3]{Vi99} for more details. This implies  that the loss of regularity of $\omega(t,x)$ occurs, as time develops, in the borderline space $B_\Gamma$.
Therefore, the question whether the regularity of $\omega(t,x)$ is preserved or not in a borderline space which does not belong to the
Lipschitz class is a challenging issue.

  A study of the global well-posedness to the two-dimensional incompressible Euler equations with non-Lipschitzian vector field was initiated by V. Yudovich.  In his pioneering work  \cite{Y-1995}, a result of global well-posedness for  essentially  bounded vorticity was obtained.  A result very similar to \cite{Y-1995} was obtained independently in Philippe Serfati \cite{Ser-0},  Pertes de  r\'{e}gularit\'{e} le laplacien et l'\'{e}quation d'Euler sui $\mathbb R^n$, priprint.15,pp., 1994."    Subsequently, many works have been dedicated to the extension of this result to  more general spaces, see e.g. the work of Serfati \cite{Ser-2} on the global well-posedness of problem \eqref{eq.Euler} for  initial data $(u_0,\omega_0)\in L^\infty(\RR^2)\times L^\infty(\RR^2)$. More recently,  Bernicot and  Keraani \cite{B-K-2} considered the equivalent form of Problem \eqref{eq.Euler}, that is, the vorticity-stream equations:
  \begin{equation}\label{eq.vorticity-stream}
\left\{\begin{array}{ll}
\partial_t\omega+(u \cdot \nabla)\omega=0, \quad (t,x)\in\RR^+\times\RR^2,\\
u=K\ast\omega,\quad \hbox{with}\quad K(x)=\frac{x^\perp}{2\pi|x|^2},\\
\omega|_{t=0}=\omega_0.
\end{array}\right.
\end{equation}
They investigated  the global well-posedness of \eqref{eq.vorticity-stream} for an unbounded vorticity $\omega_0\in L^p\cap{\rm L^0mo_F}$  with $p\in[1,2[$. Their proof strongly relies on the preserving measure property of the flow and the Whitney covering theorem. Based on this result,  Bernicot and  Hmidi \cite{B-H13} further generalized this result and established the global well-posedness of Problem \eqref{eq.vorticity-stream} for $\omega_0$ belonging to $ L^p\cap{\rm L^\alpha mo_F} $ with $p\in[1,2[$ and $\alpha\in[0,1]$. We refer to \cite[Def.1]{B-H13} for more details on spaces $\rm L^\alpha mo_F$.

In this paper  we study  the global existence and uniqueness of  weak solutions of Problem  \eqref{eq.Euler}
in $\rm bmo$-type spaces, by establishing new \textit{a priori} estimate for the smooth solution in such spaces.
We begin by proving a  uniformly local $L^p$ estimate for
$\omega$ as in \cite{Tan04}. Our strategy is to obtain
a logarithmic estimate for $u$ as well as a global-in-time estimate
of $\|u(t)\|_{L^\infty}$. Unfortunately, it seems to be impossible to get a
similar result to that from \cite{Tan04} if we apply  the algorithm
used in \cite{Ser-1,Ser-2} directly. In order to overcome this
difficulty, we exploit a new estimate for the convective term  which allows us to get  for $\alpha\in]0,1[$
\begin{equation}\label{eq.log-u}
\|u(t)\|^{1+\alpha}_{L^\infty(\RR^2)}\leq C(t)\Phi\Big(1+\int_0^t\|u(\tau)\|^{1+\alpha}_{L^\infty(\RR^2)}\intd\tau\Big)\cdot\Big(1+\int_0^t\|u(\tau)\|^{1+\alpha}_{L^\infty(\RR^2)}\intd\tau\Big),
\end{equation}
where  $\Phi(\cdot)=(T\Theta)(\cdot)$ (see Definition \ref{def-take-log}, below ).
The above estimate yields the global bound for the quantities $\|u(t)\|_{L^\infty}$ and $\|\omega(t)\|_{Y^\Theta_{\rm ul}}$ with $\Theta\in\mathcal{A}_1$. This enables us to show the global existence and uniqueness of weak solutions to Problem \eqref{eq.Euler} in a large class involving unbounded and non-decaying vorticities.
We further obtain a global well-posedness  result of problem \eqref{eq.Euler} in the Spanne space. Next, we investigate the preservation of  the regularity of $\omega$ in the borderline space $\rm L^{\alpha}bmo$ with $\alpha\in[0,1]$.
 It is known that Bernicot and  Keraani \cite {B-K-1} obtained recently an optimal estimate
\begin{equation}\label{eq.optimal-BK}
\|\omega(t)\|_{\rm bmo}\leq\Big(1+\int_0^t\|\nabla u(\tau)\|_{L^\infty}\intd\tau\Big)\|\omega_0\|_{\rm bmo},
\end{equation}
which makes  impossible to preserve the regularity of $\omega$ in $\rm bmo$.
Inspired by \cite{Vi00},  we proceed with a study of  the regularity of $\omega(t,x)$ in  space $\rm L^\alpha bmo$ with $\alpha\in[0,1]$.
 By developing some tools from a   classical harmonic analysis such as the John-Nirenberg inequality, we obtain the following estimates with a logarithmic loss of regularity
  \begin{equation*}
  \|\omega(t)\|_{\rm L^{\alpha-1}bmo}\leq C(t)\|\omega_0\|_{\rm L^{\alpha}bmo},\quad \text{for}\quad \alpha\in[0,1[,
  \end{equation*}
  and
 \begin{equation*}
  \|\omega(t)\|_{\rm L_{\log}bmo}\leq C(t)\|\omega_0\|_{\rm Lbmo}.
  \end{equation*}
Based on this, we also establish the following estimates
\begin{align*}
 & \|\omega(t)\|_{\rm L^{\alpha}bmo}&\\
 \leq& C(t)\begin{cases} \Big(1+\int_0^t\|\nabla u\|_{L^\infty}\intd\tau\Big)\|\omega_0\|_{\rm L^{\alpha}bmo},\,&\text{for}\,\,\alpha\in[0,1[;\\
 \Big(1+\int_0^t\|\nabla u\|_{L^\infty}\intd\tau\Big)\log\Big(1+\int_0^t\|\nabla u\|_{L^\infty}\intd\tau\Big)\|\omega_0\|_{\rm L^{\alpha}bmo},\,&\text{for}\,\,\alpha=1;\\
   \Big(1+\int_0^t\|\nabla u\|_{L^\infty}\intd\tau\Big)^\alpha\|\omega_0\|_{\rm L^{\alpha}bmo},\,&\text{for}\,\,\alpha>1.
  \end{cases}
\end{align*}
These inequalities generalize the sharp estimate \eqref{eq.optimal-BK} and we
give a simple proof of \eqref{eq.optimal-BK}.
As a corollary, we obtain the global well-posedness  of Problem \eqref{eq.Euler} in the space $\rm L^\alpha bmo$ with $\alpha>1$.

\emph{Notation}: We will use the following notations. Let
$m\left(\Omega\right)$ denote  the Lebesgue measure of the set
$\Omega$. We put $B_r(x_0):=\{x\in\RR^d|\,|x-x_0|<r\}$ and $\lambda
B_r(x_0):=\{x\in\RR^d|\,|x-x_0|<\lambda r\}$ for any positive number
$\lambda$. $\Avg_{\Omega}(f)$ stands for
$\frac{1}{m\left(\Omega\right)}\int_{\Omega}f(y)\intd y$. We
define
\begin{align*}
&\|f\|_{p,\,\lambda}:=\sup_{x\in\mathbb{R}^d}\left(\|f\|_{L^p(B_\lambda(x))}\right)=\sup_{x\in\mathbb{R}^d}
\Big(\int_{|x-y|<\lambda}|f(y)|^{p}\intd y\Big)^{\frac{1}{p}},\\
&L^p_{\rm ul}(\RR^d):=\big\{f\in L^1_{\rm loc}(\RR^d);\,\|f\|_{p,\,1}<\infty\big\},\quad
\|f\|_{L^p_{\rm ul}(\RR^d)}:=\|f\|_{p,\,1}.
\end{align*}
\begin{definition}
Let $\alpha \in[0,1]$. The spaces  ${\rm LLog}^\alpha $ and ${\rm LogLog}^\alpha $ consist of bounded functions $f$ satisfying
\begin{equation*}
\|f\|_{{\rm LLog}^\alpha(\RR^d)}:=\sup_{\substack{x,\,x'\in\RR^d,\\0<|x-x'|<1}}\frac{|f(x)-f(x')|}{|x-x'|(e-\log|x-x'|)^\alpha}<\infty
\end{equation*}
and
\[\Big(\|f\|_{{\rm LogLog}^\alpha(\RR^d)}:=\sup_{\substack{x,\,x'\in\RR^d,\\0<|x-x'|<1}}\frac{|f(x)-f(x')|}{|x-x'|(e-\log|x-x'|)\log^\alpha(e-\log|x-x'|)}<\infty\Big)
\]
\end{definition}
respectively. Let us remark that the spaces ${\rm LLog}^0\,$ (abbr. $\rm Lip$) and ${\rm LLog}^1$ (abbr. $\rm LLog$) correspond to the Lipschitz and the
logLipschitz space, respectively.
In addition, the space  ${\rm LogLog}^0 $ corresponds to the LogLipschitz space,
  and we denote by ${\rm LogLog} $ the space ${\rm LogLog}^1$ for the sake of simplicity.

\begin{definition}[\cite{Taylor}]\label{BCD-Modulus}
A modulus of continuity is any nondecreasing nonzero continuous function $\phi$ on $[0,\infty[$ such that $\phi(0)=0$,
 and satisfying the so-called $\Delta_2$-condition: $\phi(2h)\leq C_\phi\cdot\phi(h)$ for every $ h\in]0,\infty[$.
 \end{definition}

 \begin{definition}[\cite{Y-1995}]\label{def-take-log}
Let $\phi(p)$ be the modulus of continuity such that $\phi(p)\geq1$ on $[p_0,\infty[$  with $p_0\geq1$.
The function $\Phi(a):=(T\phi)(a)$ on the positive axis $[0,\infty]$ is defined by
\begin{equation*}
\Phi(a)=\begin{cases}
\inf_{p_0\leq p<\infty}\big\{a^{\frac{1}{p}}\phi(p)\big\},&\quad\text{for}\quad a\geq1;\\
\inf_{p_0\leq p<\infty}\big\{\phi(p)\big\},&\quad\text{for}\quad a<1.
\end{cases}
\end{equation*}
 \end{definition}

 \begin{remark}
The definition of the function $\Phi(\cdot)$ does not depend on the choice of the index $p_0$ in
the sense of the germs at infinity, see \cite{Y-1995} for more explanations.
\end{remark}
 \begin{definition}
 Let $\phi$ be a modulus of continuity. We say that

{\rm(i)}\; $\phi$ belongs to the class $\mathcal{A}_1$ if the function $\phi (p)$
satisfies the following admissible condition
\begin{equation}\label{eq.converse-os}
\int^\infty_a\frac{1}{p\phi(\log p)}\intd p=\infty\quad \text{for some positive}\quad a\in]0,\infty[.
\end{equation}

{\rm(ii)}\; $\phi$ belongs to the class $\mathcal{A}_2$ if the function $\phi$ satisfies $\int_a^\infty\frac{1}{p\phi(p)}\intd p=\infty$ for some positive $a\in]0,\infty[$.
\end{definition}

\noindent\textbf{Examples:} The following functions $\phi: q\mapsto \phi(q)$ belong to the class $\mathcal{A}_1$:
\begin{equation*}
q\mapsto q^{\beta},\quad q\mapsto q\big(\log(1+q)\big)^{\beta},\quad \text{if}\quad \beta\in[0,1] \end{equation*}
and
\begin{equation*} q\mapsto q\log(1+q)  \Big(\log\big(1+\log(1+q)\big)\Big)^{\beta}\quad\text{if}\quad \beta\in[0,1].
\end{equation*}

\begin{definition}[Yudovich]\label{def-yu}
Let $\Theta (p)\geq1$ be a non-decaying positive function on $[1,\infty[$. We define
\vskip0.15cm

{\rm(1)}\; $Y^{\Theta}(\RR^d):=\{f\in\cap_{1\leq p<\infty}L^p(\RR^d);\,\|f\|_{Y^{\Theta}(\RR^d)}<\infty\}$, where $
\|f\|_{Y^{\Theta}(\RR^d)}:=\sup_{p\geq1}\frac{\|f\|_{L^p(\RR^d)}}{\Theta(p)}$.

\vskip0.15cm

{\rm(2)}\; $Y^\Theta_{\rm ul}(\mathbb{R}^d):=\big\{f\in\cap_{1\leq p<\infty}L^p_{\rm ul}(\mathbb{R}^d);\,\|f\|_{Y^{\Theta}_{\rm ul}(\RR^d)}<\infty\big\}$, where
$\|f\|_{Y^{\Theta}_{\rm ul}(\mathbb{R}^d)}:=\sup_{p\geq1}\frac{\|f\|_{L^p_{\rm ul}(\mathbb{R}^d)}}{\Theta(p)}$.

\vskip0.15cm

{\rm(3)}\;  $Y^\Theta_{\rm Lip}(\mathbb{R}^d):=\big\{f\in C(\RR^d)\cap L^\infty(\mathbb{R}^d);\,\|f\|_{Y^{\Theta}_{\rm Lip}(\RR^d)}<\infty\big\}$, where
$$\|f\|_{Y^{\Theta}_{\rm Lip}(\mathbb{R}^d)}:=\sup_{p\geq1}\frac{\|\nabla S_{p+1}f\|_{L^\infty(\mathbb{R}^d)}}{\Theta(p)},$$
and where the operator $S_{p+1}$ will be defined in Section \ref{Sec-PRE}.
\end{definition}
From the above definition, it is easy to check that
$\|f\|_{Y^{\Theta}_{\rm Lip}(\mathbb{R}^d)}\leq C\|\nabla f\|_{Y^{\Theta}_{\rm ul}(\mathbb{R}^d)}$ for any  non-decaying function $\Theta$.

First of all, let us recall the results about the global well-posedness of problem \eqref{eq.Euler} for Yudovich type initial data,
which were established in \cite{Tan04} by the continuity argument. For the sake of completeness,
we give another proof of
 this result in virtue of the localization technique.
\begin{theorem}[Taniuchi, \cite{Tan04}]\label{Theorem-Exi}
Let $u_0\in L^\infty(\RR^2)$ and $\omega_0\in Y^\Theta_{\rm ul}(\RR^2)$ with $\Theta\in\mathcal{A}_1$. Then  system \eqref{eq.Euler} admits at least a global solution $u$ such that
\begin{equation*}
u(t,x)\in C\big(\RR^+;L^\infty(\RR^2)\big)\quad\text{and}\quad\omega(t,x)\in L^\infty_{\rm loc}\big(\RR^+;Y^\Theta_{\rm ul}(\RR^2)\big).
\end{equation*}
If, moreover, $\Theta\in\mathcal{A}_2$, then the solution $u$ is unique.
\end{theorem}
\begin{remark}
Let $\alpha\in[0,1]$ and
\begin{equation}\label{eq.exam-1}
g^\alpha(x)=\begin{cases}
\log^\alpha\big(e-\log |x|\big),\quad&\text{for}\quad |x|\leq1;\\
0,\quad&\text{for}\quad |x|>1.
\end{cases}
\end{equation}
A direct calculation allows us to conclude that $g^1(x)$ belongs to $Y^\Theta_{\rm ul}(\RR^2)$ with $\Theta(p)=\log p\in\mathcal{A}_2$,
also see \cite{Y-1995} for the proof.

Next, we take the initial velocity $u_0=(u_{1,0}, u_{2,0})$ with
\begin{equation}\label{eq.exam-vel}
u_{1,0}=\frac{\partial_2}{\Delta}g^1(x)-\sin(x_1)\cos(x_2)\quad\text{and}\quad u_{2,0}=-\frac{\partial_1}{\Delta}g^1(x)+\cos(x_1)\sin(x_2).
\end{equation}
From this, it follows that $\Div u_0=0$ and $u_0$ is bounded. In fact,  we see from the paper \cite{Ser-1} that
\begin{equation*}
\begin{split}
\left|\frac{\partial_2}{\Delta}g^1(x)\right|+\left|\frac{\partial_1}{\Delta}g^1(x)\right|&\leq C\int_{\RR^2}\frac{1}{|x-y|}|g^1(y)|\intd y\\
&= C\int_{|x-y|\leq1}\frac{1}{|x-y|}|g^1(y)|\intd y+ C\int_{|x-y|>1}\frac{1}{|x-y|}|g^1(y)|\intd y\\
&\leq C\|g^1\|_{ L^3(\RR^2)}+C\|g^1\|_{L^1(\RR^2)}\\
&\leq C\|g^1\|_{1,\,1}+C\|g^1\|_{3,\,1}\leq C\|g^1\|_{3,\,1}<\infty.
\end{split}
\end{equation*}
Thanks to the Biot-Savart law, one infers that
\begin{equation}\label{eq.exam-vor}
\omega_0=\partial_2u_{1,0}-\partial_1u_{2,0}=g^1(x)+2\sin(x_1)\sin(x_2).
\end{equation}
We easily find that $\omega_0(x)$ is an unbounded
and non-decaying  function belonging to $Y^\Theta_{\rm ul}(\RR^2)$ with $\Theta(p)=\log p\in\mathcal{A}_2$.
This implies that we can obtain the global existence and uniqueness of weak solution to \eqref{eq.Euler} in $Y^\Theta_{\rm ul}(\RR^2)$
 involving unbounded and non-decaying vorticity.
\end{remark}
\begin{theorem}\label{Coro-Exi-bmo}
Let $\alpha\in]0,1/2]$ and $u_0\in L^\infty(\RR^2)$ and $\omega_0\in Y^\Theta_{\rm ul}(\RR^2)\cap \mathcal{M}_{\varphi}(\RR^2)$ with $\Theta\in \mathcal{A}_1$ and $\varphi(r)=\log^\alpha(e-\log r)$ with $r\in]0,1/2[$. Then Problem \eqref{eq.Euler} admits a unique global solution $u$ such that
\[u(t,x)\in C\big(\RR^+;L^\infty(\RR^2)\big)\quad\text{and}\quad\omega(t,x)\in L^\infty_{\rm loc}\big(\RR^+;Y^\Theta_{\rm ul}(\RR^2)\cap \mathcal{M}_{\varphi}(\RR^2)\big).\]
Here, the Spanne space  $\mathcal{M}_{\varphi}(\RR^2)$ will be recalled in Definition \ref{def-spanne}.
\end{theorem}
\begin{remark}
A simple calculation yields that $g^\alpha(x)$ defined in \eqref{eq.exam-1} lies in Spanne space  $\mathcal{M}_{\varphi}(\RR^2)$ with $\varphi(r)=\log^\alpha(e-\log r)$ for $\alpha\in[0,1]$.
\end{remark}
\begin{remark}
We can generalize Theorem \ref{Coro-Exi-bmo} to the more general Spanne space $\mathcal{M}_{\varphi}(\RR^2)$ with
\begin{equation*}
\varphi(r)=\log^\alpha(e-\log r)\,\log_2(e-\log r)\cdots\log_m(e-\log r)\,\,\text{for any natural}\,\,m,
\end{equation*}
where $\log_m$ stands for the $m$-th iteration of logarithm.
\end{remark}

Next, we  state  results concerning a loss of regularity of solution to the two-dimensional Euler equations
in $\mathrm{\textbf{bmo}}$-type space. Specifically:
\begin{theorem}\label{Theorem-Exi-UNI}
Let $u_0\in L^\infty(\RR^2)$ and $\omega_0\in{\rm L^\alpha bmo}(\RR^2)$ with $\alpha\in[0,\infty[$ (see Definition \ref{def2.24}). Then Problem \eqref{eq.Euler} admits  at least a global solution $u$ such that
\begin{enumerate}
\item[\rm (1)] \; \label{THM-3-item-1} when $0\leq\alpha\leq1$,
$
u(t,x)\in C\big(\RR^+;L^\infty(\RR^2)\big)$ and $\omega(t,x)\in L^\infty_{\rm loc}\big(\RR^+;{\rm L^{\alpha-1} bmo}(\RR^2)\big).
$

  \item[\rm (2)]\; \label{THM-3-item-2} when $\alpha>1$,
 $
u(t,x)\in C\big(\RR^+;L^\infty(\RR^2)\big)$ and $\omega(t,x)\in L^\infty_{\rm loc}\big(\RR^+;{\rm L^{\alpha} bmo}(\RR^2)\big).
$
 \end{enumerate}
\noindent In particular, the solution $u$ is unique for  $\alpha\geq1$.

\end{theorem}

\begin{remark}
For the case $0\leq\alpha\leq1$, we see that the vorticity $\omega(t,x)$ has a
logarithmic loss of regularity for all $t>0$. Similar results concerning  loss of regularity on vorticity were shown in the works \cite{B-H13} and \cite{Vi00},
see also Theorem  \ref{Theorem-Exi-UNI} for $\alpha\in]0,1]$. From the optimal estimate \eqref{eq.optimal-BK},  it seem to be impossible to get the regularity persistence of vorticity in $\rm bmo$. Thus,  the loss of regularity on vorticity  in Theorem \ref{Theorem-Exi-UNI} seems to be inevitable.
It is interesting to  show  whether this loss is optimal or not, and we plan to study it in our future work.
\end{remark}

Next, we focus on  the control estimates in  $\rm L^\alpha bmo$ (
$\alpha\geq0$) of the flow mapping determined by the the
vorticity-stream equation \eqref{eq.vorticity-stream} with a
bi-Lipschitz vector field.  The interesting point  is how to
optimize the control estimates  by using the measure preserving
property,  the generalized John-Nirenberg
inequality in a new way.
\begin{theorem}\label{prop.reg-pser}
Let $u_0\in L^\infty(\RR^2)$ and $\omega_0\in {\rm
L^{\alpha}bmo}(\RR^2)$ with $\alpha\in[0,\infty[$. Assume that $u\in
L^1_{\rm loc}(\RR^+;{\rm Lip}(\RR^2))$ is a smooth solution of
\eqref{eq.Euler}. Then there exists a positive constant $C$,
dependent of the initial data and $\alpha$, such that
\begin{equation}\label{eq.log-crease}
\begin{split}
&\|u(t)\|_{L^\infty(\RR^2)}+\|\omega(t)\|_{{\rm L^{\alpha}bmo}(\RR^2)}\\\leq& C\begin{cases}\big(1+V_{\rm {Lip}}(t)\big)\|\omega_0\|_{{\rm L^{\alpha}bmo}(\RR^2)},\quad &\text{for}\quad \alpha\in[0,1[;\\
\big(1+V_{\rm {Lip}}(t)\big)\log\big(1+V_{\rm {Lip}}(t)\big)\|\omega_0\|_{{\rm L^{\alpha}bmo}(\RR^2)},\quad &\text{for}\quad \alpha=1;\\
\big(1+V_{\rm {Lip}}(t)\big)^\alpha\|\omega_0\|_{{\rm L^{\alpha}bmo}(\RR^2)},\quad &\text{for}\quad \alpha>1.
\end{cases}
\end{split}
\end{equation}
\end{theorem}

\begin{remark}
When  $\alpha=0$, Theorem \ref{prop.reg-pser} recovers  the
 result established in \cite{B-K-1}:
\begin{equation}\label{eq.bs-op}
\|\omega(t)\|_{{\rm bmo}(\RR^2)}\leq  C\big(1+V_{\rm {Lip}}(t)\big)\|\omega_0\|_{{\rm bmo}(\RR^2)}.
\end{equation}
It is more important that the authors in  \cite{B-K-1} showed that \eqref{eq.bs-op} is a sharp
estimate by  the property of $K$-quasi-conformal mapping and the
Whitney covering theorem. In this work, we  provide a simple proof for
\eqref{eq.bs-op}. Firstly, using the evolving property of
bi-Lipschitz flow, one can conclude for any $p\geq1$
\begin{equation*}
\|\omega(t)\|_{{\rm bmo}_p}\leq C\left(e^{(1+V_{\rm Lip}(t))}\right)^{\frac{2}{p}}\|\omega_0\|_{{\rm bmo}_p}.
\end{equation*}
In light of the H\"older inequality and Corollary \ref{Coro.JN} in
Appendix A, we have
\begin{align*}
\|\omega(t)\|_{{\rm bmo}}\leq\|\omega(t)\|_{{\rm bmo}_p}\leq &C\big(e^{(1+V_{\rm Lip}(t))}\big)^{\frac{2}{p}}\|\omega_0\|_{{\rm bmo}_p}\\
\leq &Cp\cdot\big(e^{(1+V_{\rm Lip}(t))}\big)^{\frac{2}{p}}\|\omega_0\|_{{\rm bmo}}.
\end{align*}
This estimate together with Lemma \ref{lem-take-log} yields
\begin{equation*}
\|\omega(t)\|_{{\rm bmo}(\RR^2)}\leq C\inf_{1\leq p<\infty}\Big(p\cdot e^{(1+V_{\rm Lip}(t))}\big)^{\frac{2}{p}}\Big)\|\omega_0\|_{{\rm bmo}(\RR^2)}\leq C\big(1+V_{\rm {Lip}}(t)\big)\|\omega_0\|_{{\rm bmo}(\RR^2)}.
\end{equation*}
\end{remark}
The paper is organized as follows. In Section \ref{Sec-PRE}, we review some useful statements
 on functional spaces and basic analysis tools, and introduce several technical lemmas.
 In the next section, we establish  some estimates with loss of regularity and logarithmic estimate of the solution for the transport equation
 with the vector field belonging to $\rm bmo$-type spaces. Section \ref{Sec-Proof}
 is devoted to the proof of our main theorems. Finally, we generalize the classical John-Nirenberg inequality
 and establish some product estimates and commutator estimates by using the Bony para-product decomposition.

\section{Preliminary}\label{Sec-PRE}
\subsection {Littlewood-Paley Theory and the functional spaces}

In this subsection, we first review the so-called Littlewood--Paley decomposition described, e.g., in \cite{BCD11,FS,MWZ2012}. Next, we introduce some useful functional spaces such as Morrey--Campanato space and its properties.

Let $(\chi,\varphi)$ be a couple of smooth functions with  values in $[0,1]$
such that $\chi$ is supported in the ball $\big\{\xi\in\mathbb{R}^{d}\big||\xi|\leq\frac{4}{3}\big\}$,
$\varphi$ is supported in the ring $\big\{\xi\in\mathbb{R}^{d}\,\big|\,\frac{3}{4}\leq|\xi|\leq\frac{8}{3}\big\}$ and
\begin{equation*}
    \chi(\xi)+\sum_{j\in \mathbb{N}}\varphi(2^{-j}\xi)=1\quad {\rm for \ each\ }\xi\in \mathbb{R}^{d}.
\end{equation*}
For any $u\in \mathcal{S}'(\mathbb{R}^{d})$, one  defines the
dyadic blocks as
\begin{equation*}
 \Delta_{-1}u=\chi(D)u\quad\text{and}\quad   {\Delta}_{j}u:=\varphi(2^{-j}D)u\quad {\rm for\ each\ }j\in\mathbb{N}.
\end{equation*}
We also define the following low-frequency cut-off:
\begin{equation*}
    {S}_{j}u:=\chi(2^{-j}D)u.
\end{equation*}
It is easy to verify that
\begin{equation*}
    u=\sum_{j\geq-1}{\Delta}_{j}u,\quad \text{in}\quad\mathcal{S}'(\mathbb{R}^{d})
\end{equation*}
which  is called the \emph{inhomogeneous Littlewood-Paley decomposition}.
It has nice properties of quasi-orthogonality:
\begin{equation*}
    {\Delta}_{j}{\Delta}_{j'}u\equiv 0\quad \text{if}\quad |j-j'|\geq 2.
\end{equation*}
\begin{equation*}
{\Delta}_{j}({S}_{j'-1}u{\Delta}_{j'}v)\equiv0\quad \text{if}\quad |j-j'|\geq5.
\end{equation*}
We shall also use the \emph{homogeneous Littlewood-Paley} operators as follows:
\begin{equation*}
    \dot{S}_{j}u:=\chi(2^{-j}D)u \quad\text{and}\quad\dot{\Delta}_{j}u:=\varphi(2^{-j}D)u\quad\text{for each}\,\,j\in\mathbb{Z}.
\end{equation*}
\begin{definition}
Let $\mathcal{S}_{h}'(\RR^d)$ be the space of tempered distributions $u$ such that
\begin{equation*}
    \lim_{j\rightarrow-\infty}\dot{S}_{j}u=0\quad \text{in} \quad \mathcal{S}'(\RR^d).
\end{equation*}
\end{definition}
\begin{remark} From this definition, we easily find that that for every $u\in\mathcal{S}_{h}'(\RR^d),$ we have
\[ \sum_{k<j}\dot{\Delta}_{k}u= \dot{S} _{j}u\quad\text{in the sense of distribution}.\]
One can show that for every nonconstant function  $u\in L^p(\mathbb R^d)$ ($1\le p\le\infty$),
  we have $u\in \mathcal{S}_{h}'(\RR^d)$. So, we do not  distinguish between $\sum_{k<j}\dot{\Delta}_{k}u$ and $ {\dot{S}}_{j}u$.
\end{remark}
\begin{definition} \label{def-2.5}
For every  $u, v\in\mathcal {S}_{h}^{\prime}(\RR^d)$,  the product $u\cdot
v$ has the Bony decomposition:
$$u\cdot v=\dot{T}_{u} v+\dot{T}_{v}u+\dot{R}(u,v),
$$
with the paraproduct term
$$\dot{T}_{u}v=\sum_{j\leq{k-2}}\dot{\Delta}_{j}u\dot{\Delta}_{k}v=\sum_{j}{\dot{S}_{j-1}u}{\dot{\Delta}_{j}v},$$
and with the remainder term
$$\dot{R}(u,v)=\sum_{j}\dot{\Delta}_{j}u\widetilde{\dot{\Delta}}_{j}v, \quad \widetilde{\dot{\Delta}}_{j}:=\sum_{k=-1}^1\dot{\Delta}_{j-k}.$$
\end{definition}
Now  we introduce the Bernstein lemma which will be useful
throughout this paper.
\begin{lemma}\label{bernstein}
Let $1\leq a\leq b\leq\infty$ and  $f\in L^a(\RR^d)$. Then there exists a positive constant $C$ such that for $q,\,k\in\NN$,
\begin{equation*}
\sup_{|\alpha|=k}\|\partial ^{\alpha}S_{q}f\|_{L^b(\RR^d)} \leq  C^k\,2^{q\big(k+d(\frac{1}{a}-\frac{1}{b})\big)}\|S_{q}f\|_{L^a(\RR^d)},
\end{equation*}
\begin{equation*}
C^{-k}2^{qk}\|{\Delta}_{q}f\|_{L^a(\RR^d)} \leq \sup_{|\alpha|=k}\|\partial ^{\alpha}{\Delta}_{q}f\|_{L^a(\RR^d)}\,\leq\,C^k2^{qk}\|{\Delta}_{q}f\|_{L^a(\RR^d)}.
\end{equation*}
\end{lemma}
\begin{definition}\label{def2.2}
Let $s\in \mathbb{R}$, $(p,q)\in [1,\infty]^{2}$ and $u\in \mathcal{S}'(\mathbb{R}^{d})$. Then we define the \emph{inhomogeneous Besov spaces} as
\begin{equation*}
   {B}^{s}_{p,q}(\mathbb{R}^{d}):=\big\{u\in\mathcal{S}'(\mathbb{R}^{d})\big|\norm{u}_{{B}^{s}_{p,q}(\mathbb{R}^{d})}<\infty\big\},
\end{equation*}
where
\begin{equation*}
    \norm{u}_{{B}^{s}_{p,q}(\mathbb{R}^{d})}:=\begin{cases}
    \Big(\sum_{j\geq-1}2^{jsq}\norm{{\Delta}_{j}u}_{L^{p}(\mathbb{R}^{d})}^{q}\Big)^{\frac{1}{q}}
    \quad&\text{if}\quad q<\infty,\\
   \sup_{j\geq-1}2^{js}\norm{{\Delta}_{j}u}_{L^{p}(\mathbb{R}^{d})}\quad&\text{if}\quad q=\infty.
    \end{cases}
\end{equation*}
\end{definition}
Next, we review  statements of the weighted Morrey-Campanato space and give its useful properties.
\begin{definition}\label{def2.24}
Let $\alpha\in [0,\infty[$, $p\in[1,\infty[$ and the scalar function $f$ is locally integrable. If
\begin{align*}
&\|f\|_{{\rm L^\alpha BMO}_p(\RR^d)}\\
:=&\sup_{x\in\RR^d,\,r\in]0,1[}\big(-\log r\big)^\alpha\Big(\frac{1}{m(B_r(x))}\int_{B_r(x)}\left|f(y)-\Avg_{B_r(x)}(f)\right|^p\intd y\Big)^{\frac1p}\\
&+\sup_{x\in\RR^d,\,r\geq1}\Big(\frac{1}{m(B_r(x))}\int_{B_r(x)}\left|f(y)-\Avg_{B_r(x)}(f)\right|^p\intd y\Big)^{\frac1p}
<\infty,
\end{align*}
then we say that $f\in {\rm L^\alpha BMO}_p(\RR^d)$. In the following, we denote ${\rm L^\alpha BMO}_1(\RR^d)$ by ${\rm L^\alpha BMO}(\RR^d)$.
\end{definition}
\begin{definition}\label{def2.25}
Let $\alpha\in [-1,\infty[$, $p\in[1,\infty[$ and the scalar function $f$ is a locally integrable. If
\begin{align*}
&\|f\|_{{\rm L^\alpha bmo}_p(\RR^d)}\\
:=&\sup_{x\in\RR^d,\,r\in]0,1[}\big(-\log r\big)^\alpha\Big(\frac{1}{m(B_r(x))}\int_{B_r(x)}\left|f(y)-\Avg_{B_r(x)}(f)\right|^p\intd y\Big)^{\frac1p}\\
&+ \sup_{x\in\RR^d}\int_{B_1(x)}\left|f(y)\right|\intd y<\infty,
\end{align*}
then we say that $f\in {\rm L^\alpha bmo}_p(\RR^d)$. Here, we denote ${\rm L^\alpha bmo}_1(\RR^d)$ by ${\rm L^\alpha bmo}(\RR^d)$.
\end{definition}
It is worthwhile to remark that ${\rm L^\alpha bmo}(\RR^d)$ is equivalent to ${\rm L^\alpha bmo}_p$ for all $\alpha\geq0$ and $p>1$ by Corollary \ref{Coro.JN}.
Next, we state basic properties of the space ${\rm L^\alpha bmo}(\RR^d)$  which will be used in the following sections.
\begin{proposition}\label{prop-lbmo-properties}
\begin{enumerate}
\item[\rm (1)] \label{eq.item-0} For $\alpha_1\geq\alpha_2$, we have ${\rm L^{\alpha_1} bmo}(\RR^d)\hookrightarrow{\rm L^{\alpha_2}bmo}(\RR^d)$.
  \item[\rm (2)] \label{eq.item-1} ${\rm L^\alpha bmo}(\RR^d)$  is a Banach space for any $\alpha\geq0$.
  \item[\rm (3)]\label{eq.item-2} If $\alpha\in]0,\infty[$, then, for all $q>\frac1\alpha$, we have that ${\rm L^\alpha bmo}(\RR^d)$ continuously embeds $B^0_{\infty,q}(\RR^d)$. In particular, ${\rm bmo}(\RR^d)\hookrightarrow B^{0}_{\infty,\infty}(\RR^d)$.
  \item[\rm(4)] \label{eq.item-4}For every  $f\in L^1(\RR^d)$ and $g\in {\rm L^\alpha bmo}(\RR^d)$ with $\alpha\geq0$, one has
\begin{equation*}
\|f\ast g\|_{{\rm L^\alpha bmo}(\RR^d)}\leq \|f\|_{L^1(\RR^d)}\|g\|_{{\rm L^\alpha bmo}(\RR^d)}.
\end{equation*}
\end{enumerate}
\end{proposition}
\begin{proof}
(1) is obvious.

(2).  It is well-known that ${\rm bmo}(\RR^d)$ is a Banach space (see for example \cite{LRbook}). So we just need to show  ${\rm L^\alpha bmo}(\RR^d)$  is a Banach space for any $\alpha>0$. Let the family $\{f_n\}_n$   be a Cauchy sequence in ${\rm L^\alpha bmo}(\RR^d)$. Since ${\rm bmo}(\RR^d)$
is complete, we know that the sequence $\{f_n\}_n$  converges in  ${\rm bmo}(\RR^d)$ and then in $L^1_{\text{loc}}(\RR^d)$.
According to the definition of space and the convergence in  $L^1_{\text{loc}}(\RR^d)$, we immediately get that the convergence holds in  ${\rm L^\alpha bmo}(\RR^d)$. This shows completeness of space ${\rm L^\alpha bmo}(\RR^d)$ for all $\alpha>0$.

(3).  For each $f\in {\rm L^\alpha bmo}(\RR^d)$, by
Lemma \ref{lemma-young} and \cite[Proposition 1]{B-H13}, we can
conclude that
\begin{align*}
\|f\|_{B^0_{\infty,q}(\RR^d)}=&\Big(\sum_{k\geq1}\|\Delta_kf\|^q_{L^\infty(\RR^d)}\Big)^{\frac1q}+\|\Delta_{0}f\|_{L^\infty(\RR^d)}+\|\Delta_{-1}f\|_{L^\infty(\RR^d)}\\
\leq&\Big(\sum_{k\geq1}k^{-\alpha q}\|f\|_{{\rm L^\alpha bmo}(\RR^d)}^q\Big)^{\frac1q}+C\|f\|_{1,\,1}\\
\leq&\|f\|_{{\rm L^\alpha bmo}(\RR^d)}\Big(\sum_{k\geq1}k^{-\alpha q}\Big)^{\frac1q}+C\|f\|_{{\rm L^\alpha bmo}(\RR^d)}.
\end{align*}
Since $q\alpha>1$,  the series $\sum_{k\geq1}k^{-\alpha q}$ is
convergent.


(4) Since the space ${\rm L^\alpha bmo}(\RR^d)$ is  a
shift-invariant space,  we obtain that ${\rm L^\alpha bmo}(\RR^d)$ is stable
through convolution with functions in $L^1(\RR^d)$ by   \cite[Proposition 4.1]{LRbook}.
\end{proof}
Next, we  introduce the space which is an important generalization  of Campanato spaces.
 This space was firstly studied by Spanne \cite{Spanne}, see also \cite{Func}  for more details.
\begin{definition}\label{def-spanne}
Let $\varphi$ be a positive non-increasing function.  We define  the Spanne space $\mathcal{M}_{\varphi}(\RR^d)$ of all integrable functions $f$ such that $\|f\|_{\mathcal{M}_{\varphi}(\RR^d)}<\infty$, where its norm $\|\cdot\|_{\mathcal{M}_{\varphi}(\RR^d)}$ is defined as follows
\begin{align*}
\|f\|_{\mathcal{M}_{\varphi}(\RR^d)}=&\sup_{x\in\RR^d,\,r\in]0,\frac12[}\frac{1}{\varphi(r)m\left(B_r(x)\right)}\int_{B_r(x)}|f(y)| \intd y\\
 &+\sup_{x\in\RR^d,\,r\geq\frac12 }\frac{1}{m\left(B_r(x)\right)}\int_{B_r(x)}|f(y)| \intd y.
\end{align*}
\end{definition}
Let us remark that the Spanne space $\mathcal{M}_{\varphi}(\RR^d)$ coincides with $L^\infty(\RR^d)$ when $\varphi$ is a constant function.
\subsection{Some useful lemmas and propositions} In this subsection, we give certain useful technical lemmas and propositions which are the cornerstones in our analysis.

\begin{lemma}[\cite{Tan04}]\label{lemma-young}
\begin{enumerate}
  \item[\rm(1)]\; Let $\varphi\in \mathcal{S}(\mathbb{R}^2)$, then there holds
\begin{equation*}
\big\|\varphi\ast f\|_{L^\infty(\RR^2)}\leq C\|f\|_{1,\,1},\quad \forall\, f\in L^1_{\rm ul}(\RR^2),
\end{equation*}
where $C$ is a positive constant independent of $f$.
  \item [\rm (2)]\;If $m\geq1$, then
\begin{equation*}
\|f\|_{q,\,m\lambda}\leq (2m^{2})^{\frac1q}\|f\|_{q,\,\lambda},\quad\forall\,f\in L^q_{\rm ul}(\RR^2)\quad \text{and}\quad \forall\,\lambda>0.
\end{equation*}
\end{enumerate}
\end{lemma}
\begin{lemma}\label{local-bern}
Let $1\leq q\leq\infty$, $j\in\mathbb{N}$ and $\|f\|_{q,\,2^{-j}}<\infty$. Then there holds
\begin{equation}
\|\Delta_jf\|_{L^\infty(\RR^2)}\leq C2^{\frac{2j}{q}}\|f\|_{q,\,2^{-j}},
\end{equation}
where $C$ is a positive constant  independent of $q,\,j$ and $f$.
\end{lemma}
\begin{proof}
Changing  variables, one concludes that
\begin{equation*}
|\dot\Delta_jf(x)|=\Big|2^{2j}\int_{\mathbb{R}^2}\varphi\big(2^{j}y\big)f(x-y)\intd y\Big|=\Big|\int_{\mathbb{R}^2}\varphi(y)f\Big(\frac{2^{j}x-y}{2^{j}}\Big)\intd y\Big|.
\end{equation*}
For $f_j(x):=f(x/2^{j})$,  we have
\begin{equation*}
|\dot\Delta_jf(x)|=\Big|\int_{\mathbb{R}^2}\varphi(y)f_j(2^{j}x-y)\intd y\Big|=|\dot\Delta_0f_j(2^{j}x)|.
\end{equation*}
Using Lemma \ref{lemma-young}, we see
\begin{equation*}
\|\dot\Delta_0f_j\|_{L^\infty(\RR^2)}\leq C\|f_j\|_{1,\,1}=C\|f(\cdot/2^j)\|_{1,\,1}\leq C\|f(\cdot/2^j)\|_{q,\,1}.
\end{equation*}
This implies
\begin{equation}\label{eq.-local-1}
\|\dot\Delta_jf\|_{L^\infty(\RR^2)}\leq C\|f(\cdot/2^j)\|_{q,\,1}.
\end{equation}
Clearly,
\begin{align*}
\Big(\int_{|x-y|\leq1}\Big|f\Big(\frac{y}{2^{j}}\Big)\Big|^{q}\intd y\Big)^{\frac{1}{q}}
=&\Big(2^{2j}\int_{|2^{-j}x-y|\leq2^{-j}}\big|f(y)\big|^{q}\intd y\Big)^{\frac{1}{q}}\nonumber\\
\leq&2^{\frac{2j}{q}}\|f\|_{q,\,2^{-j}}.
\end{align*}
Inserting this inequality into \eqref{eq.-local-1} yields the desired result.
\end{proof}
\begin{remark}
From Lemma \ref{lemma-young} and Lemma \ref{local-bern}, it follows that
\begin{equation}\label{eq.local-bern}
\|\Delta_jf\|_{L^\infty(\RR^2)}\leq \begin{cases}
C2^{\frac{2j}{q}}\|f\|_{q,\,1},\quad &\forall\,j\geq0;\\
C\|f\|_{q,\,1},\quad &\forall\,j<0.
\end{cases}
\end{equation}
\end{remark}
\begin{lemma}\label{lem-take-log}
For  $\alpha\in]0,1]$,  we  take $\phi(p)=p^\alpha$ in Definition \ref{def-take-log}. Then $\Phi(a)$ is given by
\begin{equation*}
\Phi(a)=(T\phi)(a)=\begin{cases}
\left(\frac{e}{\alpha}\right)^\alpha\cdot(\log a)^\alpha,\quad&\text{if}\quad a>e^{\alpha p_0};\\
p_0^{\alpha}a^{\frac{1}{p_0}},\quad&\text{if}\quad 1\leq a\leq e^{\alpha p_0};\\
p_0^{\alpha},\quad&\text{if}\quad 0\leq a<1.
\end{cases}
\end{equation*}
\end{lemma}
\begin{proof}
These are simple calculations, and we omit the details.
\end{proof}
\begin{proposition}\label{lem-log}
Let $f\in {\rm L^\alpha bmo}(\RR^d)$ with $\alpha\geq0$ and $B=B_r(x)$ be a ball in $\mathbb{R}^d$ with $r\in]0,\frac12]$.
Then for all $1<\lambda<\frac1r$ we have
\begin{equation*}
\begin{split}
\big|{\rm Avg}_B(f)&-{\rm Avg}_{\lambda B}(f)\big|
\leq2^{d}\left(-\log\lambda r\right)^{-\alpha}\|f\|_{{\rm L^\alpha bmo}(\RR^d)}\\&+C\begin{cases}
\left(\big(-\log r\big)^{1-\alpha}-\big(\log \lambda r\big)^{1-\alpha}\right)\|f\|_{{\rm L^\alpha bmo}(\RR^d)},&\text{for}\,\,\alpha\in[0,1[;\\
\big(\log\big(-\log r\big)-\log\big(-\log\lambda r\big)\big)\|f\|_{{\rm L^\alpha bmo}(\RR^d)},&\text{for}\,\,\alpha=1;\\
\left(\big(-\log\lambda r\big)^{1-\alpha}-\big(-\log r\big)^{1-\alpha}\right)\|f\|_{{\rm L^\alpha bmo}(\RR^d)},&\text{for}\,\,\alpha>1,
\end{cases}
\end{split}
\end{equation*}
where the positive constant $C$ depends on $\alpha$ and is independent of $f$.
\end{proposition}
\begin{proof}
Since $\lambda>1$,  there exists a nonnegative integer $k_0$ such
that $2^{k_0}\leq\lambda<2^{k_0+1}$. By the triangle inequality, we
have
\begin{align}\label{eq-avg-1}
\big|{\rm Avg}_B(f)-{\rm Avg}_{\lambda B}(f)\big|\leq&\sum_{k=0}^{k_0-1}\big|{\rm Avg}_{2^{k}B}(f)-{\rm Avg}_{2^{k+1}B}(f)\big|\\
&+\big|{\rm Avg}_{2^{k_0}B}(f)-{\rm Avg}_{\lambda B}(f)\big|.\nonumber
\end{align}
Using the doubling property of the Euclidean measure, one concludes that for $k\in[0,k_0-1]$
\begin{align}\label{eq-avg-2}
\big|{\rm Avg}_{2^{k}B}(f)-{\rm Avg}_{2^{k+1}B}(f)\big|\leq&{\rm Avg}_{2^{k}B}\big|f-{\rm Avg}_{2^{k+1}B}(f)\big|\nonumber\\
\leq&\frac{|2^{k+1}B|}{|2^kB|}\cdot\frac{1}{|2^{k+1}B|}\int_{2^{k+1}B}\big|f(y)-{\rm Avg}_{2^{k+1}B}(f)\big|\intd y\nonumber\\
\leq&2^{d}\left(-\log2^{k+1}r\right)^{-\alpha}\|f\|_{\rm L^\alpha bmo}\nonumber\\
=&2^{d}\big(-(k+1)-\log r\big)^{-\alpha}\|f\|_{\rm L^\alpha bmo}.
\end{align}
Similarly, we obtain
\begin{align}\label{eq-avg-3}
\big|{\rm Avg}_{2^{k_0}B}(f)-{\rm Avg}_{\lambda B}(f)\big|\leq &{\rm Avg}_{2^{k_0}B}\big|f-{\rm Avg}_{\lambda B}(f)\big|\nonumber\\
\leq&\frac{|\lambda B|}{|2^{k_0}B|}\cdot\frac{1}{|\lambda B|}\int_{\lambda B}\big|f(y)-{\rm Avg}_{\lambda B}(f)\big|\intd y\nonumber\\
\leq&2^{d}\left(-\log\lambda r\right)^{-\alpha}\|f\|_{\rm L^\alpha bmo}.
\end{align}
Inserting \eqref{eq-avg-2} and \eqref{eq-avg-3} into \eqref{eq-avg-1}, we eventually get
\begin{align}\label{eq-avg-4}
\big|{\rm Avg}_B(f)-{\rm Avg}_{\lambda B}(f)\big|\leq&\sum_{k=0}^{k_0-1}2^{d}\big(-(k+1)-\log r\big)^{-\alpha}\|f\|_{\rm L^\alpha bmo}\\
&+2^{d}\left(-\log\lambda r\right)^{-\alpha}\|f\|_{\rm L^\alpha bmo}.\nonumber
\end{align}
We observe that:  for $\alpha\in[0,1[,$ we have
\begin{align*}
&\sum_{k=0}^{k_0-1}2^{d}\big(-(k+1)-\log r\big)^{-\alpha}\|f\|_{\rm L^\alpha bmo}\\
=&\sum_{k=1}^{k_0}2^{d}\big(-k-\log r\big)^{-\alpha}\|f\|_{\rm L^\alpha bmo}\\
\leq&\frac{2^{d}}{1-\alpha}\left(\big(-1-\log r\big)^{1-\alpha}-\big(-k_0-\log r\big)^{1-\alpha}\right)\|f\|_{\rm L^\alpha bmo}\\
\leq&C\left(\big(-\log r\big)^{1-\alpha}-\big(-\log \lambda r\big)^{1-\alpha}\right)\|f\|_{\rm L^\alpha bmo}\\
\leq&C(\log\lambda)^{1-\alpha}\|f\|_{\rm L^\alpha bmo},
\end{align*}
for $\alpha=1,$ we obtain
\begin{align*}
&\sum_{k=1}^{k_0}2^{d}\big(-k-\log r\big)^{-1}\|f\|_{\rm L^\alpha bmo}\\
=&2^{d}\big(\log\big(-1-\log r\big)-\log\big(-k_0-\log r\big)\big)\|f\|_{\rm L^\alpha bmo}\\
\leq&C\big(\log\big(-\log r\big)-\log\big(-\log\lambda r\big)\big)\|f\|_{\rm L^\alpha bmo}\\
\leq&C\log\Big(1+\frac{\log\lambda}{-\log\lambda r}\Big)\|f\|_{\rm L^\alpha bmo},
\end{align*}
and for $\alpha>1,$ we have
\begin{align*}
&\sum_{k=1}^{k_0}2^{d}\big(-k-\log r\big)^{-\alpha}\|f\|_{\rm L^\alpha bmo}\\
\leq&\frac{2^{d}}{\alpha-1}\left(\big(-k_0-\log r\big)^{1-\alpha}-\big(-1-\log r\big)^{1-\alpha}\right)\|f\|_{\rm L^\alpha bmo}\\
\leq&C\left(\big(-\log\lambda r\big)^{1-\alpha}-\big(-\log r\big)^{1-\alpha}\right)\|f\|_{\rm L^\alpha bmo}\\
\leq&C(\log\lambda)^{\alpha-1}\big(-\log\lambda r\big)^{1-\alpha}\big(-\log r\big)^{1-\alpha}\|f\|_{\rm L^\alpha bmo}\\
\leq&C(\log\lambda)^{\alpha-1}\big(-\log\lambda r\big)^{1-\alpha}\big(-\log r\big)^{1-\alpha}\|f\|_{\rm L^\alpha bmo}.
\end{align*}
Plugging these estimates into \eqref{eq-avg-4}, we obtain the required result.
\end{proof}
\begin{remark}\label{rem-log}
From the above proof, we obviously see that
\begin{equation*}
\begin{split}
\big|{\rm Avg}_B(f)&-{\rm Avg}_{\lambda B}(f)\big|
\leq2^{d}\left(-\log\lambda r\right)^{-\alpha}\|f\|_{{\rm L^\alpha bmo}(\RR^d)}\\
&+C\|f\|_{{\rm L^\alpha bmo}(\RR^d)}\begin{cases}
(\log\lambda)^{1-\alpha},&\text{for}\,\,\alpha\in[0,1[;\\
\log\Big(1+\frac{\log\lambda}{-\log\lambda r}\Big),&\text{for}\,\,\alpha=1;\\
(\log\lambda)^{\alpha-1}\big(\log\lambda r\cdot\log r\big)^{1-\alpha},&\text{for}\,\,\alpha>1.
\end{cases}
\end{split}
\end{equation*}
\end{remark}
\begin{lemma}\label{lemma-p-alpha}
Let $f\in {\rm L^\alpha bmo}(\RR^d)$. There exists a positive constant $C$ independent of $f$ such that
\begin{equation}\label{eq.-p-alpha-1}
\sup_{1\leq p<\infty}\frac{\|f\|_{p,\,1}}{p^{1-\alpha}}\leq C\|f\|_{{\rm L^\alpha bmo}(\RR^d)},\quad \text{for}\quad \alpha\in[0,1[,
\end{equation}
and
\begin{equation}\label{eq.-p-alpha-2}
\sup_{1\leq p<\infty}\frac{\|f\|_{p,\,1}}{\log(1+p)}\leq C\|f\|_{{\rm Lbmo}(\RR^d)}.
\end{equation}
\end{lemma}
\begin{proof}
By the John-Nirenberg inequality, it follows from the result in
\cite{Tan04} that
\begin{equation*}
\sup_{1\leq p<\infty}\frac{\|f\|_{p,\,1}}{p}\leq C\|f\|_{{\rm bmo}(\RR^d)}.
 \end{equation*}
 So we just need to show Lemma \ref{lemma-p-alpha} in the case $\alpha\in]0,1]$.

For a fixed unit ball $B_1(x)\subset\RR^d$, performing the Vitali covering theorem, we conclude that there exists a collection $\{B_{2^{-p}}(x_k)\}_k$ such that
\begin{enumerate}
  \item[1.] $B_1(x)\subset\bigcup_{k}B_{2^{-p}}(x_k)$;
  \item[2.]The balls $\{B_{2^{-(p+1)}}(x_k)\}_k$ are mutually disjoint;
  \item[3.] For each $k$,  $B_{2^{-(p+1)}}(x_k)\subset B_1(x)$.
\end{enumerate}
Whence, we have
\begin{align*}
&\Big(\frac{1}{m\left(B_1(x)\right)}\int_{B_1(x)}\big|f(y)\big|^{p}\intd y\Big)^{\frac{1}{p}}\\
\leq &\Big(\frac{1}{m\left(B_1(x)\right)}\sum_{k}\int_{B_{2^{-p}}(x_k)}\big|f(y)\big|^{p}\intd y\Big)^{\frac{1}{p}}\nonumber\\
=&\Big(\sum_{k}\frac{m(B_{2^{-p}}(x_k))}{m\left(B_1(x)\right)}\cdot\frac{1}{m(B_{2^{-p}}(x_k))}\int_{B_{2^{-p}}(x_k)}\big|f(y)\big|^{p}\intd y\Big)^{\frac{1}{p}}\nonumber\\
\leq& 2^{\frac{d}{p}}\sup_{x\in\RR^d}\Big(\frac{1}{m\left(B_{2^{-p}}(x)\right)}\int_{B_{2^{-p}}(x)}\big|f(y)\big|^{p}\intd y\Big)^{\frac{1}{p}}.
\end{align*}
We observe that
\begin{align*}
&\Big(\frac{1}{m\left(B_{2^{-p}}(x)\right)}\int_{B_{2^{-p}}(x)}\big|f(y)\big|^{p}\intd y\Big)^{\frac{1}{p}}\\
\leq&\Big(\frac{1}{m\left(B_{2^{-p}}(x)\right)}\int_{B_{2^{-p}}(x)}\big|f(y)-\Avg_{B_{2^{-1}}(x)}(f)\big|^{p}\intd y\Big)^{\frac{1}{p}}\\
&+\frac{1}{m\left(B_{2^{-1}}(x)\right)}\int_{B_{2^{-1}}(x)}\big|f(y)\big|\intd y\\
:=&I_1+I_2.
\end{align*}
On one hand, it is clear that
\begin{equation*}\label{eq.-p-1}
I_2\leq 2^{d}\int_{B_{1}(x)}\big|f(y)\big|\intd y\leq 2^{d}\|f\|_{\rm bmo}.
\end{equation*}
Next, by the triangle inequality, we can conclude that
\begin{align*}
I_1\leq &\Big(\frac{1}{m\left(B_{2^{-p}}(x)\right)}\int_{B_{2^{-p}}(x)}\big|f(y)-\Avg_{B_{2^{-p}}(x)}(f)\big|^{p}\intd y\Big)^{\frac{1}{p}}\\
&+\frac{1}{m\left(B_{2^{-1}}(x)\right)}\int_{B_{2^{-1}}(x)}\big|f(y)-\Avg_{B_{2^{-p}}(x)}(f)\big|\intd y\\
:=&I_{11}+I_{12}.
\end{align*}
According to the definition of  $\rm L^\alpha bmo$ and Corollary
\ref{Coro.JN}, the term $I_{11}$ can be controlled by
\begin{equation*}
p^{-\alpha}\|f\|_{{\rm L^\alpha bmo}_p}\leq Cp^{1-\alpha}\|f\|_{\rm L^\alpha bmo}.
\end{equation*}
Next, performing Remark \ref{rem-log} with $\lambda=2^{p-1}$ and $r=2^{-p}$, we obtain
\begin{align*}
I_{12}\leq& \frac{1}{m\left(B_{2^{-1}}(x)\right)}\int_{B_{2^{-1}}(x)}\big|f(y)-\Avg_{B_{2^{-1}}(x)}(f)\big|\intd y\\
&+\left|\Avg_{B_{2^{-1}}(x)}(f)-\Avg_{B_{2^{-p}}(x)}(f)\right|\\
\leq&C\begin{cases}
p^{1-\alpha}\|f\|_{\rm L^\alpha bmo},&\alpha\in[0,1[;\\
\log(1+p)\|f\|_{\rm Lbmo},&\alpha=1.
\end{cases}
\end{align*}
Collecting all these estimates, we  obtain
\begin{align*}
\Big(\int_{B_1(x)}\big|f(y)\big|^{p}\intd y\Big)^{\frac{1}{p}}\leq&\big(m\left(B_1(x)\right)\big)^{\frac{1}{p}}\Big(\frac{1}{m\left(B_1(x)\right)}\int_{B_1(x)}\big|f(y)\big|^{p}\intd y\Big)^{\frac{1}{p}}\\
\leq&C\begin{cases}
p^{1-\alpha}\|f\|_{\rm L^\alpha bmo},&\alpha\in[0,1[;\\
\log(1+p)\|f\|_{\rm Lbmo},&\alpha=1.
\end{cases}
\end{align*}
This completes the proof of Lemma \ref{lemma-p-alpha}.
\end{proof}
\begin{proposition}\label{Prop-flowmap-log}
Let $u$ be a time-dependent vector field in $L^1_{\rm loc}(\RR^+; {\rm LogLog}^\alpha\cap L^\infty)$ with $\alpha\in[0,1]$. There exists a unique continuous map $\psi$ from
$\RR^+\times\RR^d$ to $\RR^d$ such that
\begin{equation}\label{eq.flowfun-2}
\psi(t,x)=x+\int_0^tu\big(\tau,\psi(\tau,x)\big)\intd\tau.
\end{equation}
Moreover, if $|x-y|<\min\{1,\,L_\alpha\}$, where
\begin{equation}\label{eq.flow-small-2}
 L_\alpha:=\begin{cases}e^{e-\exp{2^{\frac{\alpha}{1-\alpha}}}\big(1+\big((1-\alpha)V_{{\rm LogLog}^\alpha}(t)\big)^{\frac{1}{1-\alpha}}\big)},&\alpha\in[0,1[,\\
e^{\big(e-\exp(e^{V_{{\rm LogLog}}(t)})\big)},&\alpha=1,\end{cases}
\end{equation}
and $V_{{\rm LogLog}^\alpha}(t):=\int_0^t\|u(\tau)\|_{{\rm LogLog}^\alpha}\intd\tau$, then we have
\begin{align}\label{eq.flow-evolution-2}
&\big|\psi(t,x)-\psi(t,y)\big|\\
\leq&\left\{\begin{aligned}& e^{e-(e-\log \left(|x-y|\right))^{2^{-\frac{\alpha}{1-\alpha}}}\exp\big(-(1-\alpha)^{\frac{1}{1-\alpha}}\big(V_{{\rm LogLog}^\alpha}(t)\big)^{\frac{1}{1-\alpha}}\big)},\quad \alpha\in[0,1[,\\
&e^{e-\left(e-\log \left(|x-y|\right)\right)^{\exp{\big(-V_{{\rm LogLog}}}(t)\big)}},\qquad\alpha=1.
\end{aligned}\right.\nonumber
\end{align}
\end{proposition}
\begin{proof}
Since $u\in L^1_{\rm loc}(\RR^+; {\rm LogLog}^\alpha\cap L^\infty)$ with $\alpha\in[0,1]$ and
 $$\int_0^1\frac{1}{r(e-\log r)\log^\alpha(e-\log r)}\,{\rm d}r=\infty,$$
 we know that  $u$ is an Osgood modulus of continuity. By \cite[ Theorem 3.2]{BCD11},  we
 get existence and uniqueness of solution to equation \eqref{eq.flowfun-2}. Now, let $\delta(t):=\big|\psi(t,x)-\psi(t,y)\big|$. According to equations
\eqref{eq.flowfun-2} and the vector field $u\in L^1_{\rm loc}(\RR^+; {\rm LogLog}^\alpha\cap L^\infty)$, one concludes that
\begin{equation}\label{eq.flow-1-2}
\delta(t)\leq\delta(0)+\int_0^t\|u(\tau)\|_{{\rm LogLog}^\alpha}\delta(\tau)\big(e-\log\delta(\tau)\big)\big(\log(e-\log(\delta(\tau))\big)^\alpha\intd\tau
\end{equation}
as long as $\delta(\tau)<1$, for all $\tau\in[0,t[$.

Let $I(t):=\{\tau\in[0,t]\,|\, F(\tau)\leq1\}$ with
\[F(s):=\delta(0)+\int_0^s\|u(\tau)\|_{{\rm
LogLog}^\alpha}\delta(\tau)\big(e-\log\delta(\tau)\big)\big(\log(e-\log(\delta(\tau))\big)^\alpha\intd\tau.\]
Our
target is now to prove that $I(t)=[0,t]$ when $\delta(0)<
\min\{1,L_\alpha\}$. Thanks to the continuity in time of the flow
and the fact that $F(0)=\delta(0)<1$, we know that $I(t)$ is a
non-empty closed set. Thus, it remains for us to show that $t_*=t$,
where
\begin{equation*}
t_{*}=\max\big\{\tau\in[0,t],\,[0,\tau]\subset I(t)\big\}.
\end{equation*}
In a similar fashion as \eqref{eq.flow-1-2}, we infer that
\begin{equation}\label{eq.flow-2-2}
\delta(s)\leq\delta(0)+\int_0^s\|u(\tau)\|_{{\rm LogLog}^\alpha}\delta(\tau)\big(e-\log\delta(\tau)\big)\big(\log(e-\log(\delta(\tau))\big)^\alpha\intd\tau
\end{equation}
for each $s\in I_*(t):=[0,t_*].$

From definition of $F(s)$, a simple calculation yields
\begin{equation*}
\begin{split}
F'(s)=&\|u(s)\|_{{\rm LogLog}^\alpha}\delta(s)\big(e-\log\delta(s)\big)\big(\log(e-\log(\delta(s))\big)^\alpha\\
\leq &\|u(s)\|_{{\rm LogLog}^\alpha} F(s)\big(e-\log F(s)\big)\big(\log(e-\log(F(s))\big)^\alpha,
\end{split}
\end{equation*}
where we have used the facts that  $s(e-\log s)\big(\log(e-\log s)\big)^\alpha $ is a positive
increasing function on $]0,1]$ and $\delta(s)\leq F(s)$. This
implies
\begin{equation*}
-H'_{\alpha}\big(F(s)\big)\leq \|u(s)\|_{{\rm LogLog}^\alpha},
\end{equation*}
where
\begin{equation*}
\quad H_\alpha(\sigma)=\begin{cases}\frac{1}{1-\alpha}\big(\log(e-\log \sigma)\big)^{1-\alpha},\quad&\text{if}\quad\alpha\in[0,1[,\\
\log\big(\log(e-\log \sigma)\big),\quad&\text{if}\quad\alpha=1.
\end{cases}
\end{equation*}
Accordingly, we have
\begin{equation}\label{eq.flow-f-2}
H_{\alpha}\big(F(s)\big)\geq H_{\alpha}\big(F(0)\big)-\int_0^s\|u(\tau)\|_{{\rm LogLog}^\alpha}\intd\tau=H_{\alpha}\big(\delta(0)\big)-\int_0^s\|u(\tau)\|_{{\rm LogLog}^\alpha}\intd\tau.
\end{equation}
Thanks to the representation formula of $H_\alpha$ with $\alpha\in[0,1]$, we find that
$H_1$ is bijective from $]0,1]$ into $[0,\infty[$, and $H_\alpha$ also
is bijective from $]0,1]$ into $[\frac{1}{1-\alpha},\infty[$ for
$\alpha\in[0,1[$. Thus, there exists a unique inverse function
$H_\alpha^{-1}$ of $H_\alpha$ given by the following formula:
\begin{equation*}
\quad H_\alpha^{-1}(\sigma)=\begin{cases}e^{e-\exp{\left((1-\alpha)\sigma\right)^{\frac{1}{1-\alpha}}}},\quad&\text{if}\quad\alpha\in[0,1[,\\
e^{\left(e-\exp e^{\sigma}\right)},\quad&\text{if}\quad\alpha=1.
\end{cases}
\end{equation*}
Next, we see that \eqref{eq.flow-f-2} means that for all $t\in[0,t_*]$,
\begin{equation*}
\int_0^t\|u(\tau)\|_{{\rm LogLog}^\alpha}\intd\tau\leq H_{\alpha}\big(\delta(0)\big),\quad \text{for all}\quad\alpha\in[0,1].
\end{equation*}
For $\alpha\in[0,1[$ and $0<c<b<\infty$, we have
\begin{equation}\label{eq.flow-3-2}
\begin{split}
H_{\alpha}^{-1}(b-c)=&e^{e-\exp{\big((1-\alpha)^{\frac{1}{1-\alpha}}(b-c)^{\frac{1}{1-\alpha}}\big)}}\\
\leq&e^{e-\exp{(1-\alpha)^{\frac{1}{1-\alpha}}\big(2^{-\frac{\alpha}{1-\alpha}}b^{\frac{1}{1-\alpha}}-c^{\frac{1}{1-\alpha}}\big)}}\\
=&e^{e}\cdot e^{-\big(\exp{(1-\alpha)^{\frac{1}{1-\alpha}}b^{\frac{1}{1-\alpha}}}\big)^{2^{-\frac{\alpha}{1-\alpha}}}\exp\big(-(1-\alpha)^{\frac{1}{1-\alpha}}c^{\frac{1}{1-\alpha}}\big)}\\
=&e^{e-\big(e-\log H^{-1}_\alpha(b)\big)^{2^{-\frac{\alpha}{1-\alpha}}}\exp\big(-(1-\alpha)^{\frac{1}{1-\alpha}}c^{\frac{1}{1-\alpha}}\big)},
\end{split}
\end{equation}
where, in the second line of \eqref{eq.flow-3-2}, we have used the following inequality
\begin{equation*}
(b-c)^{\frac{1}{1-\alpha}}\geq2^{-\frac{\alpha}{1-\alpha}}b^{\frac{1}{1-\alpha}}-c^{\frac{1}{1-\alpha}}.
\end{equation*}
This together with \eqref{eq.flow-f-2} and  \eqref{eq.flow-small-2} allows us to conclude that for all $s\in[0,t_*]$
\begin{equation*}
\begin{split}
\delta(s)\leq F(s)\leq e^{e-\big(e-\log \left(\delta(0)\right)\big)^{2^{-\frac{\alpha}{1-\alpha}}}\exp\big(-(1-\alpha)^{\frac{1}{1-\alpha}}\big(\int_0^t\|u(\tau)\|_{{\rm LogLog}^\alpha}\intd\tau\big)^{\frac{1}{1-\alpha}}\big)} <1.
\end{split}
\end{equation*}
For $\alpha=1$ and $0<c<b<\infty$, we observe that
\begin{equation*}
H_{1}^{-1}(b-c)=e^{\left(e-\exp\left(e^{(b-c)}\right)\right)}=e^e\cdot e^{-\exp\big(e^{b}\cdot e^{-c}\big)}=e^{e-\left(e-\log H_{1}^{-1}(b)\right)^{e^{-c}}}.
\end{equation*}
Combining this with \eqref{eq.flow-f-2} and \eqref{eq.flow-small-2}, it follows that for all $s\in[0,t_*]$
\begin{equation*}
\delta(s)\leq F(s)\leq e^{e-\left(e-\log \left(\delta(0)\right)\right)^{e^{-\int_0^t\|u(\tau)\|_{{\rm LogLog}}\intd\tau}}}<1.
\end{equation*}
Therefore, we can conclude the proof by the continuity argument.
\end{proof}


\section{A priori estimates}

In this  section, we  gather useful \emph{a priori} estimates.
\subsection{\emph{A priori} estimates for the voriticity equation in $Y^\Theta_{\rm ul}(\RR^2)$} In this subsection, our target is to establish
\emph{a priori} estimates for the voriticity equation in $Y^\Theta_{\rm ul}(\RR^2)$. Let us begin by  the uniformly local $L^p$ estimate for
the transport equation.
\begin{proposition}\label{prop-loc-p}
Consider vector field $u\in L^1(\RR^+;L^\infty(\RR^d))$. Assume that  $f(t,x)$ is a smooth solution of the following equation
\begin{equation}\label{eq.TRAN}
\left\{\begin{array}{ll}
\partial_tf+(u\cdot\nabla) f=0,\quad (t,x)\in \mathbb{R}^+\times\mathbb{R}^d,\\
\Div u=0,\\
f|_{t=0}=f_0.
\end{array}\right.
\end{equation}
Then there exists a positive constant $C$, independent of $p$ and $r$, such that
\begin{equation*}
\|f(t)\|_{p,\,r}\leq C\left(\frac1r\right)^{\frac{d}{p}}\Big(r+\int_0^t\|u(\tau)\|_{L^\infty(\RR^d)}\intd\tau\Big)^{\frac{d}{p}}\|f_0\|_{p,\,r}.
\end{equation*}
\end{proposition}
\begin{proof}
For arbitrary $y\in\RR^2$, and  any positive number $\lambda$, let us denote  $\phi^{y}_{\lambda}(\cdot)=\phi^{y}(\frac{\cdot}{\lambda})$,
 where $\phi^{y}(x)$ is a non-negative smooth function satisfying
\begin{equation}\label{eq.def-cutoff}
\phi^{y}(x)=\begin{cases}1,\quad &x\in B_{1}(y)\\
0,\quad&x\in B^c_{2}(y).
\end{cases}
\end{equation}
In the following part, we will use $\phi(x)$  instead of $\phi^{y}(x)$ for convenience. Obviously, we get from \eqref{eq.TRAN} that
\begin{equation}\label{eq.TRAN-local}
\partial_t\left(\phi_{\lambda r} f\right)+(u\cdot\nabla)\left(\phi_{\lambda r} f\right)=((u\cdot\nabla)\phi_{\lambda r})f.
\end{equation}
Multiplying \eqref{eq.TRAN-local} by $|\phi_{\lambda r} f|^{p-2}\phi_{\lambda r} f$ and then integrating the resulting equation yield that
\begin{equation*}
\begin{split}
&\frac1p\dtd\left\|\phi_{\lambda r}(\cdot)f(t,\cdot)\right\|^p_{L^p(\RR^d)}\\
=&\int_{\RR^d}((u\cdot\nabla)\phi_{\lambda r})f|\phi_{\lambda r} f|^{p-2}\phi_{\lambda r} f\intd y\\
=&\int_{\RR^d}\phi_{2\lambda r}((u\cdot\nabla)\phi_{\lambda r})f|\phi_{\lambda r} f|^{p-2}\phi_{\lambda r} f\intd y\\
\leq &\|u\|_{L^\infty(\RR^d)}\|\nabla\phi_{\lambda r}\|_{L^\infty(\RR^d)}\left\|\phi_{\lambda r}(\cdot)f(t,\cdot)\right\|^{p-1}_{L^p(\RR^d)}\left\|\phi_{2\lambda r}(\cdot)f(t,\cdot)\right\|_{L^p(\RR^d)}\\
\leq&\frac{C}{\lambda r}\|u\|_{L^\infty(\RR^d)}\left\|\phi_{\lambda r}(\cdot)f(t,\cdot)\right\|^{p-1}_{L^p(\RR^d)}\left\|\phi_{2\lambda r}(\cdot)f(t,\cdot)\right\|_{L^p(\RR^d)}.
\end{split}
\end{equation*}
From this, it follows that
\begin{equation}
\dtd\left\|\phi_{\lambda r}(\cdot)f(t,\cdot)\right\|_{L^p(\RR^d)}\leq \frac{C}{\lambda r}\|u(t,\cdot)\|_{L^\infty(\RR^d)}\left\|\phi_{2\lambda r}(\cdot)f(t,\cdot)\right\|_{L^p(\RR^d)}.
\end{equation}
Integrating the above inequality with respect to time $t$, we immediately obtain that
\begin{align*}
\left\|\phi_{\lambda r}(\cdot)f(t,\cdot)\right\|_{L^p(\RR^d)}\leq& \left\|\phi_{\lambda r}(\cdot)f_0\right\|_{L^p(\RR^d)}\\
&+\int_0^t\frac{C}{\lambda r}\|u(\tau,\cdot)\|_{L^\infty(\RR^d)}\left\|\phi_{2\lambda r}(\cdot)f(\tau,\cdot)\right\|_{L^p(\RR^d)}\intd\tau.
\end{align*}
Taking the supremum of the above inequality over all $y\,\in\RR^d$ leads to
\begin{equation*}
\begin{split}
\|f(t)\|_{p,\,\lambda r}\leq&\|f_0\|_{p,\,2\lambda r}+\int_0^t\frac{C}{\lambda r}\|u(\tau,\cdot)\|_{L^\infty(\RR^d)}\left\|f(\tau)\right\|_{p,\,4\lambda r}\intd\tau\\
\leq&C\|f_0\|_{p,\,\lambda r}+\frac{C}{\lambda r}\int_0^t\|u(\tau,\cdot)\|_{L^\infty(\RR^d)}\left\|f(\tau)\right\|_{p,\,\lambda r}\intd\tau.
\end{split}
\end{equation*}
By the Gronwall inequality and $\lambda\geq1$, we have
\begin{equation}
\begin{split}
\|f(t)\|_{p,\,r}\leq\|f(t)\|_{p,\,\lambda r}\leq&\|f_0\|_{p,\,\lambda r}\cdot e^{\frac{C}{\lambda r}\int_0^t\|u(\tau,\cdot)\|_{L^\infty(\RR^d)}\intd\tau}\\
\leq&C\lambda^{\frac{d}{p}}\|f_0\|_{p,\,r}\cdot e^{\frac{C}{\lambda r}\int_0^t\|u(\tau,\cdot)\|_{L^\infty(\RR^d)}\intd\tau}.
\end{split}
\end{equation}
If, moreover, we choose suitable $\lambda$ satisfying
\begin{equation*}
\lambda=\max\left\{\frac{1}{r}\int_0^t\|u(\tau)\|_{L^\infty(\RR^d)}\intd\tau,\quad 1\right\},
\end{equation*}
we finally obtain
\begin{equation*}
\|f(t)\|_{p,\,r}\leq C\Big(1+\frac{1}{r}\int_0^t\|u(\tau)\|_{L^\infty(\RR^d)}\intd\tau\Big)^{\frac{d}{p}}\|f_0\|_{p,\,r},
\end{equation*}
which completes the proof.
\end{proof}
\begin{lemma}\label{Lem-pro-est}
Let $\alpha\in]0,1]$ and $\mathcal{R}$ be  the Riesz operator. For any positive integer $N$, there holds
\begin{equation}
\|\dot S_{-N}\mathcal{R}((u\cdot\nabla) u)\|_{L^\infty(\RR^2)}\leq C2^{-N\alpha}\|u\|^{1+\alpha}_{L^\infty(\RR^2)}\|\omega\|^{1-\alpha}_{p,\,1}+C2^{-N}\|\omega\|^2_{p,\,1},\quad\text{for}\quad p>2.
\end{equation}
\end{lemma}
\begin{proof}
According to Bony's paraproduct decomposition, one writes
\begin{equation*}
(u\cdot\nabla) u=T_{u_i}\partial_iu+T_{\partial_iu}u_i+R(u_i,\partial_iu).
\end{equation*}
Let us estimate  the paraproduct terms $T_{u_i}\partial_iu$ and $T_{\partial_iu}u_i$. By the H\"older inequality and the discrete Young inequality, we have
\begin{align}\label{eq.lemma-con-1}
&\|\dot S_{-N}\mathcal{R}(T_{u_i}\partial_iu)\|_{L^\infty(\RR^2)}\\
\leq&C\sum_{k\leq-N}2^{k}\|\dot\Delta_k(T_{u_i}u)\|_{L^\infty(\RR^2)}\nonumber\\
\leq&C\sum_{k\leq-N}2^{k}\sum_{|k-q|\leq5}\|\dot S_{q-1}u\|_{L^\infty(\RR^2)}\|\dot\Delta_qu\|_{L^\infty(\RR^2)}\nonumber\\
\leq&C\sum_{k\leq-N}2^{k\alpha}\sum_{|k-q|\leq5}2^{(k-q)(1-\alpha)}\|u\|_{L^\infty(\RR^2)}2^{-q\alpha}\|\dot\Delta_q\omega\|_{L^\infty(\RR^2)}\nonumber\\
\leq&C2^{-N\alpha}\|u\|_{L^\infty(\RR^2)}\|\dot S_{-N+5}\omega\|_{\dot B_{\infty,\infty}^{-\alpha}(\RR^2)}.\nonumber
\end{align}
Resorting to the interpolation theorem and \eqref{eq.local-bern}, we have
\begin{equation}\label{eq.sharp}
\begin{split}
\|\dot S_{-N+5}\omega\|_{\dot B_{\infty,\infty}^{-\alpha}(\RR^2)}\leq&C \|\omega\|^\alpha_{\dot B^{-1}_{\infty,\infty}(\RR^2)}\|\dot S_{-N+5}\omega\|^{1-\alpha}_{\dot B^0_{\infty,\infty}(\RR^2)}\\
\leq&C \|\omega\|^\alpha_{\dot B^{-1}_{\infty,\infty}(\RR^2)}\sup_{k\leq-N+6}\|\dot\Delta_k\omega\|^{1-\alpha}_{L^\infty(\RR^2)}\\
\leq&C \|u\|^\alpha_{L^\infty(\RR^2)}\|\omega\|^{1-\alpha}_{p,\,1}.
\end{split}
\end{equation}
Inserting \eqref{eq.sharp} into \eqref{eq.lemma-con-1}, we get
\begin{equation*}
\|\dot S_{-N}\mathcal{R}(T_{u_i}\partial_iu)\|_{L^\infty(\RR^2)}\leq C2^{-N\alpha}\|u\|^{1+\alpha}_{L^\infty(\RR^2)}\|\omega\|^{1-\alpha}_{p,1}.
\end{equation*}
Similarly, we obtain
\begin{equation*}
\|\dot S_{-N}\mathcal{R}(T_{\partial_iu}u_i)\|_{L^\infty(\RR^2)}\leq C2^{-N\alpha}\|u\|^{1+\alpha}_{L^\infty(\RR^2)}\|\omega\|^{1-\alpha}_{p,1}.
\end{equation*}
It remains to deal with the remainder term $R(u_i,\partial_iu)$. Thanks to the low-high decomposition technique, we decompose it into two parts as follows:
\begin{equation*}
R(u_i,\partial_iu)=\sum_{k<0}\dot\Delta_ku_i\widetilde{\dot\Delta_k}\partial_iu+
\sum_{k\geq0}\dot\Delta_ku_i\widetilde{\dot\Delta_k}\partial_iu:=R^\natural
+R^\sharp.
\end{equation*}
By using the support property of $R^{\natural}(u_i,\partial_iu)$, we see that  $\|\dot S_{-N}\mathcal{R}(R^\natural)\|_{L^\infty(\RR^2)}$ can be controlled by
\begin{align}\label{eq.lemma-con-2}
&C\sum_{j\leq -N}\|\dot\Delta_jR^\natural\|_{L^\infty(\RR^2)}\\
\leq&C\sum_{j\leq -N}2^{j}\sum_{j-5\leq k<0}\|\dot\Delta_ku_i\|_{L^\infty(\RR^2)}\|\widetilde{\dot\Delta_k}u\|_{L^\infty(\RR^2)}\nonumber\\
\leq&C\sum_{j\leq -N}2^{j\alpha}\sum_{j-5\leq k<0}2^{(j-k)(1-\alpha)}\|u\|_{L^\infty(\RR^2)}2^{-k\alpha}\|\dot\Delta_k\omega\|_{L^\infty(\RR^2)}\nonumber\\
\leq&C2^{-N\alpha}\|u\|_{L^\infty(\RR^2)}\|\Delta_{0}\omega\|_{B^{-\alpha}_{\infty,\infty}(\RR^2)}.\nonumber
\end{align}
By the same argument as in proof \eqref{eq.sharp}, we  infer that
\begin{equation}\label{eq.sharp-2}
\|\Delta_0\omega\|_{\dot B_{\infty,\infty}^{-\alpha}(\RR^2)}\leq C \|u\|^\alpha_{L^\infty(\RR^2)}\|\omega\|^{1-\alpha}_{p,\,1}.
\end{equation}
Plugging \eqref{eq.sharp-2} in \eqref{eq.lemma-con-2}, we obtain
\begin{equation*}
\|\dot S_{-N}\mathcal{R}(R^\natural)\|_{L^\infty(\RR^2)}\leq C2^{-N\alpha}\|u\|^{1+\alpha}_{L^\infty(\RR^2)}\|\omega\|^{1-\alpha}_{p,1}.
\end{equation*}
Finally, since $p>2$, the last term $\|\dot S_{-N}\mathcal{R}(R^\sharp)\|_{L^\infty(\RR^2)}$ can be bounded by
\begin{align*}
C\sum_{j\leq -N}\|\dot\Delta_jR^\sharp\|_{L^\infty(\RR^2)}
\leq&C\sum_{j\leq -N}2^{j}\sum_{k\geq 0}\|\dot\Delta_ku_i\|_{L^\infty(\RR^2)}\|\dot\Delta_ku\|_{L^\infty(\RR^2)}\\
\leq&C2^{-N}\sum_{k\geq 0}2^{-2k}\|\dot\Delta_k\omega\|_{L^\infty(\RR^2)}^2\\
\leq&C2^{-N}\sum_{k\geq 0}2^{-2k}2^{\frac{4k}{p}}\|\omega\|^2_{p,\,1}
\leq C2^{-N}\|\omega\|^2_{p,\,1}.
\end{align*}
Collecting all these estimates yields the desired result.
\end{proof}
\begin{proposition}\label{prop.local-infty}
Let $u_0\in L^\infty(\RR^2)$ and $\omega_0\in Y^\Theta_{\rm ul}(\RR^2)$ with $\Theta\in \mathcal{A}_1$. Assume that $u$ is a smooth solution of \eqref{eq.Euler}. Then we have
\begin{equation*}
\|u(t)\|_{L^\infty(\RR^2)}+\|\omega(t)\|_{Y^\Theta_{\rm ul}(\RR^2)}\leq C(t),
\end{equation*}
where the positive smooth function $C(t)$ depends on the initial data.
\end{proposition}
\begin{proof}
Thanks to the low-high decomposition technique, one can write
\begin{equation}\label{eq.low-cut-Esti}
\|u(t)\|_{L^\infty(\RR^2)}\leq \|\dot S_{-N}u\|_{L^\infty(\RR^2)}+\sum_{q\geq-N}\|\dot\Delta_qu\|_{L^\infty(\RR^2)},
\end{equation}
where $N$ is a positive integer to be specified later.

Let us recall that
\begin{equation*}
\partial_tu+\mathcal{P}((u\cdot\nabla) u)=0,
\end{equation*}
where the Leray projector is defined by
\[\mathcal{P}=Id-(\mathcal{R}_{i}\mathcal{R}_{j}).\]
Performing the low frequency cut-off operator $\dot S_{-N}$ to the above equality, we get
\begin{equation}\label{eq.low-cut}
\partial_t\dot S_{-N}u+\dot S_{-N}\mathcal{P}((u\cdot\nabla) u)=0.
\end{equation}
Integrating \eqref{eq.low-cut} in time $t$ and using Lemma \ref{Lem-pro-est}, one has  for $p>2$,
\begin{equation*}
\begin{split}
\|\dot{S}_{-N}u(t)\|_{L^\infty(\RR^2)}\leq&\|\dot{S}_{-N}u_0\|_{L^\infty(\RR^2)}+\int_0^t\left\|\dot S_{-N}\mathcal{P}((u\cdot\nabla) u)(\tau)\right\|_{L^\infty(\RR^2)}\intd\tau\\
\leq&C+C2^{-N\alpha}\int_0^t\left(\|u(\tau)\|^{1+\alpha}_{L^\infty(\RR^2)}\|\omega(\tau)\|^{1-\alpha}_{p,\,1}+\|\omega(\tau)\|^2_{p,\,1}\right)\intd\tau.
\end{split}
\end{equation*}
For the high frequency part, by resorting to \eqref{eq.local-bern}, the interpolation theorem and the H\"older inequality, we easily find that for $\alpha\in]0,1[$ and $p>2$
\begin{align*}
&\sum_{q\geq-N}\|\dot\Delta_qu\|_{L^\infty(\RR^2)}\\
\leq&\sum_{-N\leq q\leq-1}\|\dot\Delta_qu\|_{L^\infty(\RR^2)}+\sum_{q\geq0}\|\dot\Delta_qu\|_{L^\infty(\RR^2)}\\
\leq&C\sum_{-N\leq q\leq-1}2^{-q\alpha}\|\dot\Delta_qu\|^{1-\alpha}_{L^\infty(\RR^2)}\|\dot\Delta_q\omega\|^{\alpha}_{L^\infty(\RR^2)}
+C\sum_{q\geq0}2^{-q\big(1-\frac2p\big)}\|\omega\|_{p,\,1}\\
\leq&C2^{N\alpha}\|u(t)\|^{1-\alpha}_{L^\infty(\RR^2)}\|\omega(t)\|^{\alpha}_{p,\,1}+C\|\omega(t)\|_{p,\,1}\\
\leq&\frac12\|u(t)\|_{L^\infty(\RR^2)}+C2^N\|\omega(t)\|_{p,\,1}.
\end{align*}
Combining these estimates and then plugging the resulting estimate in \eqref{eq.low-cut-Esti}, we immediately obtain that
\begin{align*}
&\|u(t)\|_{L^\infty(\RR^2)}\\
\leq&C2^{-N\alpha}\int_0^t\left(\|u(\tau)\|^{1+\alpha}_{L^\infty(\RR^2)}
\|\omega(\tau)\|^{1-\alpha}_{p,\,1}+\|\omega(\tau)\|^2_{p,\,1}\right)\intd\tau+C2^N\|\omega\|_{p,\,1}+C\\
\leq&C2^{-N\alpha}\sup_{\tau\in[0,t]}\|\omega(\tau)\|^{1-\alpha}_{p,\,1}\int_0^t\|u(\tau)\|^{1+\alpha}_{L^\infty(\RR^2)}\intd\tau+Ct2^{-N\alpha}\sup_{\tau\in[0,t]}\|\omega(\tau)\|^2_{p,\,1}\\
&+C2^{N}\|\omega(t)\|_{p,\,1}+C.
\end{align*}
Take a suitable integer $N$ such that
\begin{equation*}
2^{N}\sim\Big(\int_0^t\|u(\tau)\|^{1+\alpha}_{L^\infty(\RR^2)}\intd\tau\Big)^{\frac{1}{1+\alpha}}
\Big(1+\sup_{\tau\in[0,t]}\|\omega(\tau)\|_{p,\,1}\Big)^{\frac{-\alpha}{1+\alpha}}+1.
\end{equation*}
From this, it follows that
\begin{align*}
\|u(t)\|_{L^\infty(\RR^2)}\leq&C\Big(\sup_{\tau\in[0,t]}\|\omega(\tau)\|_{p,\,1}\Big)^{\frac{1}{1+\alpha}}\Big(\int_0^t\|u(\tau)\|^{1+\alpha}_{L^\infty(\RR^2)}\intd\tau\Big)^{\frac{1}{1+\alpha}}
\\&+Ct2^{-N\alpha}\sup_{\tau\in[0,t]}\|\omega(\tau)\|^2_{p,\,1}+C.
\end{align*}
Furthermore, we have
\begin{align}\label{eq.w-loc-p-0}
\|u(t)\|^{1+\alpha}_{L^\infty(\RR^2)}\leq& C\sup_{\tau\in[0,t]}\|\omega(\tau)\|_{p,\,1}\int_0^t\|u(\tau)\|^{1+\alpha}_{L^\infty(\RR^2)}\intd\tau\\
&+Ct^{1+\alpha}\Big(\sup_{\tau\in[0,t]}\|\omega(\tau)\|_{p,\,1}\Big)^{2(1+\alpha)}+C.\nonumber
\end{align}
Next, applying Proposition \ref{prop-loc-p} to the vorticity equation, we can conclude that for any $p\geq1$
\begin{equation}\label{eq.w-loc-p}
\|\omega(t)\|_{p,\,1}\leq C\|\omega_0\|_{p,\,1}\left(1+\|u\|_{L^1_tL^\infty(\RR^2)}\right)^{\frac{2}{p}}.
\end{equation}
Inserting \eqref{eq.w-loc-p} into \eqref{eq.w-loc-p-0} leads to
\begin{align*}
\|u(t)\|^{1+\alpha}_{L^\infty(\RR^2)}\leq&C\|\omega_0\|_{p,\,1}\left(1+\|u\|_{L^1_tL^\infty(\RR^2)}\right)^{\frac{2}{p}}\int_0^t\|u(\tau)\|^{1+\alpha}_{L^\infty(\RR^2)}\intd\tau
\\&+Ct^{1+\alpha}\|\omega_0\|^{2(1+\alpha)}_{p,\,1}\left(1+\|u\|_{L^1_tL^\infty(\RR^2)}\right)^{\frac{4(1+\alpha)}{p}}+C\\
\leq&C\|\omega_0\|_{p,\,1}\Big(1+t^{\frac{2\alpha}{p(1+\alpha)}}\Big)\left(1+\|u\|_{L^{1+\alpha}_tL^\infty(\RR^2)}\right)^{\frac{2}{p}}\int_0^t\|u(\tau)\|^{1+\alpha}_{L^\infty(\RR^2)}\intd\tau
\\&+Ct^{1+\alpha}\|\omega_0\|^{2(1+\alpha)}_{p,\,1}\Big(1+t^{\frac{4\alpha}{p}}\Big)\left(1+\|u\|_{L^{1+\alpha}_tL^\infty(\RR^2)}\right)^{\frac{4(1+\alpha)}{p}}+C\\
\leq&C(1+t^2)\|\omega_0\|_{p,\,1}\left(1+\|u\|_{L^{1+\alpha}_tL^\infty(\RR^2)}\right)^{\frac{2}{p}}\int_0^t\|u(\tau)\|^{1+\alpha}_{L^\infty(\RR^2)}\intd\tau
\\&+C(1+t^4)\|\omega_0\|^{2(1+\alpha)}_{p,\,1}\left(1+\|u\|_{L^{1+\alpha}_tL^\infty(\RR^2)}\right)^{\frac{4(1+\alpha)}{p}}+C.
\end{align*}
Thus, the quantity $\|u(t)\|^{1+\alpha}_{L^\infty(\RR^2)}$ can be bounded by
\begin{equation}\label{eq.w-loc-1}
\begin{split}
&C(1+t^2)\sup_{2<q<\infty}\tfrac{\|\omega_0\|_{q,\,1}}{\Theta(q)}\cdot \Theta(p)\Big(1+\|u\|_{L^{1+\alpha}_tL^\infty(\RR^2)}\Big)^{\frac{2}{p}}\int_0^t\|u(\tau)\|^{1+\alpha}_{L^\infty(\RR^2)}\intd\tau\\
&+C(1+t^4)\Big(\sup_{2<q<\infty}\tfrac{\|\omega_0\|_{q,\,1}}{\Theta(q)}\Big)^{2(1+\alpha)}\cdot\big(\Theta(p)\big)^{2(1+\alpha)}
\Big(1+\|u\|_{L^{1+\alpha}_tL^\infty(\RR^2)}\Big)^{\frac{4(1+\alpha)}{p}}+C\\
\leq&C(1+t^2)\sup_{2<q<\infty}\tfrac{\|\omega_0\|_{q,\,1}}{\Theta(q)}\cdot \Theta\Big(\frac p2\Big)\Big(1+\|u\|_{L^{1+\alpha}_tL^\infty(\RR^2)}\Big)^{\frac{2}{p}}\int_0^t\|u(\tau)\|^{1+\alpha}_{L^\infty(\RR^2)}\intd\tau
\\&+C(1+t^4)\Big(\sup_{2<q<\infty}\tfrac{\|\omega_0\|_{q,\,1}}{\Theta(q)}\Big)^{2(1+\alpha)}\cdot\Big(\Theta\Big(\tfrac{p}{2}\Big)\Big)^{2(1+\alpha)}
\Big(1+\|u\|_{L^{1+\alpha}_tL^\infty(\RR^2)}\Big)^{\frac{4(1+\alpha)}{p}}+C,
\end{split}
\end{equation}
where we have used the relation that $\Theta(p)\leq C\Theta\left(\frac{p}{2}\right)$ for all $p>2$ because $\Theta(\cdot)$ satisfies the $\Delta_2$ condition.

 Since $\|\omega_0\|_{q,\,1}\leq C\|\omega_0\|_{4,\,1}$ for each $q\in[1,4]$, we just need to take the infimum of \eqref{eq.w-loc-1} over all $p\in]4,\infty[$ to obtain
\begin{align*}
&\|u(t)\|^{1+\alpha}_{L^\infty(\RR^2)}\\
\leq&C(1+t^2)\|\omega_0\|_{Y^\Theta_{\rm ul}(\RR^2)}\cdot\Phi\Big(1+\int_0^t\|u(\tau)\|^{1+\alpha}_{L^\infty(\RR^2)}\intd\tau\Big)\cdot\int_0^t\|u(\tau)\|^{1+\alpha}_{L^\infty(\RR^2)}\intd\tau
\\&+C(1+t^4)\|\omega_0\|^{2(1+\alpha)}_{Y^\Theta_{\rm ul}(\RR^2)}\cdot\Phi^{2(1+\alpha)}\Big(1+\int_0^t\|u(\tau)\|^{1+\alpha}_{L^\infty(\RR^2)}\intd\tau\Big)+C\\
\leq&C(1+t^2)\|\omega_0\|_{Y^\Theta_{\rm ul}(\RR^2)}\cdot\Theta\Big(\log\Big(e+\int_0^t\|u(\tau)\|^{1+\alpha}_{L^\infty(\RR^2)}\intd\tau\Big)\Big)\cdot\int_0^t\|u(\tau)\|^{1+\alpha}_{L^\infty(\RR^2)}\intd\tau
\\&+C(1+t^4)\|\omega_0\|^{2(1+\alpha)}_{Y^\Theta_{\rm ul}(\RR^2)}\cdot\Theta^{2(1+\alpha)}\Big(\log\Big(e+\int_0^t\|u(\tau)\|^{1+\alpha}_{L^\infty(\RR^2)}\intd\tau\Big)\Big)+C.
\end{align*}
Since $\Theta\in\mathcal{A}_1$, the admissible  condition guarantees that
$\int_e^\infty\frac{1}{s\Theta(\log s)}\intd s=\infty$. By the H\"older inequality, we have
\begin{equation*}
\begin{split}
\int_1^\infty\frac{1}{t\Theta(\log t)}\intd t\leq&\Big(\int_1^\infty t^{\frac{-2(1+\alpha)}{(1+2\alpha)}} \intd t\Big)^{\frac{1+2\alpha}{2(1+\alpha)}}\Big(\int_1^\infty\Theta^{-2(1+\alpha)}(\log t)\intd t\Big)^{\frac{1}{2(1+\alpha)}}\\
=&C(\alpha)\Big(\int_1^\infty\Theta^{-2(1+\alpha)}(\log t)\intd t\Big)^{\frac{1}{2(1+\alpha)}},
\end{split}
\end{equation*}
which implies that $\int_1^\infty\Phi^{-2(1+\alpha)}(\log t)\intd t=\infty$. This allows us to construct
\[\mathcal{M}(x)=\int_1^x\frac{\mathrm{d}\,t}{\Phi^{2(1+\alpha)}(\log t)}.\]
Following the proof of the Osgood theorem as in [\cite{BCD11}, Lemma 3.4], one can conclude
\begin{equation}\label{eq.u-uniform-0}
\int_0^t\|u(\tau)\|_{L^\infty(\RR^2)}\intd\tau\leq C(t).
\end{equation}
Plugging \eqref{eq.u-uniform-0} in \eqref{eq.w-loc-p} enables us to infer that
$
\|\omega(t)\|_{Y^\Theta_{\rm ul}(\RR^2)}\leq C(t).
$
Taking $\alpha=\frac12$ in \eqref{eq.w-loc-p-0}, we easily find that
\begin{equation*}
\|u\|^{\frac32}_{L^\infty_tL^\infty(\RR^2)}\leq C(t)\|u\|^{\frac12}_{L^\infty_tL^\infty(\RR^2)}\int_0^t\|u(\tau)\|_{L^\infty(\RR^2)}\intd\tau+C(t)\leq\frac12\|u\|_{L^\infty_tL^\infty(\RR^2)}^{\frac32}+C(t).
\end{equation*}
This implies that $\|u(t)\|_{L^\infty(\RR^2)}\leq C(t)$ and then the proof is complete.
\end{proof}
Now, we turn to the study of the regularity of the voriticity in the Spanne space $\mathcal{M}_{\varphi}$.
\begin{proposition}\label{prop.local-Spanne}
Let $\alpha\in]0,1/2]$, $u_0\in L^\infty(\RR^2)$, $\omega_0\in Y^\Theta_{\rm ul}(\RR^2)\cap \mathcal{M}_{\varphi}(\RR^2)$ with $\Theta\in \mathcal{A}_1$, and $\varphi(r)=\log^\alpha(e-\log r)$. Assume that $u$ is a smooth solution of \eqref{eq.Euler}. Then we have
\begin{equation}
\|u(t)\|_{L^\infty(\RR^2)}+\|\omega(t)\|_{Y^\Theta_{\rm ul}(\RR^2)\cap \mathcal{M}_{\varphi}(\RR^2)}\leq C(t),
\end{equation}
where the positive smooth function $C(t)$ depends on the initial data and $\alpha$.
\end{proposition}
Before proving this proposition, we first review some properties of flow maps established in \cite{B-K-1}. Assume that $\psi\in\mathcal{L}$ (the group of all bi-Lipschitz homeomorphism of $\mathbb{R}^d$) is measure preserving. We know that $\psi(B_r(x))$ is a bounded open set and $m(\psi(B_{r}(x)))=m(B_r(x))$. By using the Whitney covering theorem, one can conclude that there exists a bounded collection $\{O_k\}_k$ such that
\begin{enumerate}
  \item[\rm(A)]\label{eq.star-1} $\{2O_k\}_k$ is a bounded covering:
\begin{equation}
\psi(B_r(x))\subset\bigcup_k2O_k;
\end{equation}
   \item[\rm(B)]\label{eq.star-2}The balls $O_k$ are pairwise disjoint and  for each $k$, $O_k\subset\psi(B_r(x))$;

   \item[\rm(C)]\label{eq.star-3}The Whitney property is verified:
\begin{equation}
r_{O_k}\approx \intd\left(O_k,\,\psi(B_r(x))^c\right).
\end{equation}
\end{enumerate}
\noindent Clearly, the measure preserving property ensures that $m(O_k)\leq m(B_r(x))$ for all $k$, which implies that
$r_{O_k}\leq r_B$ for all $k$. Moreover, it entails  the following useful lemma.
\begin{lemma}\label{Lem-small-2}
Let $\alpha\in[0,1[$, $L_k:=\sum_{e^{-(k+1)}r< r_j\leq e^{-k}r}m\left(O_j\right)$ for any $k\geq1$ and
\begin{equation*}
N_\alpha:=\begin{cases}
\exp{2^{\frac{\alpha}{1-\alpha}}}\big(1+\big((1-\alpha)V_{{\rm LogLog}^\alpha}(t)\big)^{\frac{1}{1-\alpha}}\big)-e+\log r,&\alpha\in[0,1[,\\
\exp(e^{V_{{\rm LogLog}}(t)})-e+\log r,&\alpha=1,
\end{cases}
\end{equation*}
where $V_{{\rm LogLog}^\alpha}(t):=\int_0^t\|u(\tau)\|_{{\rm LogLog}^\alpha}\intd\tau$. Then there exists a universal constant $C>0$ such that for each $k\geq\max\{1, N_\alpha\}$
\begin{equation}\label{eq.theta-k}
L_k\leq
\begin{cases}Cr e^{e-(e+k-\log r)^{2^{-\frac{\alpha}{1-\alpha}}}\exp\big(-(1-\alpha)^{\frac{1}{1-\alpha}}\big(V_{{\rm LogLog}^\alpha}(t)\big)^{\frac{1}{1-\alpha}}\big)},&\alpha\in[0,1[,\\
Cre^{e-\left(e+k-\log r\right)^{\exp{\big(-V_{{\rm LogLog}}}(t)\big)}},&\alpha=1.
\end{cases}
\end{equation}
\end{lemma}
\begin{proof}
Here, we just give the proof of estimate \eqref{eq.theta-k} for $\alpha\in[0,1[$, because the proof for $\alpha=1$ is similar. Thanks to the preservation of Lebesgue measure by $\psi(x,t)$, we find
\begin{equation*}
 m\left(\big\{ y\in \psi(B): \intd\big(y, \psi(B)^c\big)\leq Ce^{-k}r\big\}\right)=m\left(\big\{ x\in B: \intd\big(\psi(x), \psi(B)^c\big)\leq Ce^{-k}r\big\}\right),
\end{equation*}
This together with the fact $\psi(B)^c=\psi(B^c)$ ensures
\begin{equation*}
L_k\leq m\left(\{ x\in B: \intd(\psi(x), \psi(B^c))\leq Ce^{-k}r\}\right):= D_k.
\end{equation*}
Since $\psi(\partial B)$ is the frontier of $\psi(B)$ and $\intd(\psi(x),
\psi(B^c))=\intd(\psi(x), \partial \psi(B))$, we have
\begin{equation*}
D_k\subset \big\{ x\in B: \exists\, y\in \partial B \,\,\text{with}\,\,
|\psi(x)- \psi(y)|\leq Ce^{-k}r\big\}.
\end{equation*}
The condition on $k$ allows us to use Proposition \ref{Prop-flowmap-log} to get
\begin{equation*}
D_k\subset\Big \{ x\in B: \intd(x,\partial B)\leq
C e^{e-(e+k-\log r)^{2^{-\frac{\alpha}{1-\alpha}}}\exp\big(-(1-\alpha)^{\frac{1}{1-\alpha}}(V_{{\rm LogLog}^\alpha}(t))^{\frac{1}{1-\alpha}}\big)}\Big\},
\end{equation*}
which implies the desired estimate \eqref{eq.theta-k} for $\alpha\in[0,1[$.
\end{proof}
\begin{proof}[Proof of Proposition \ref{prop.local-Spanne}]
From Proposition \ref{prop.local-infty}, we know the following estimate
\begin{equation}\label{eq.Local-infty}
\|u(t)\|_{L^\infty(\RR^2)}+\|\omega(t)\|_{Y^\Theta_{\rm ul}(\RR^2)}\leq C(t) \quad\text{for each}\; t>0.
\end{equation}
This together with Lemma \ref{lemma-young} allows us to conclude that for $r>1$
\begin{equation}\label{eq.small-big}
\frac{1}{m(B_r(x))}\int_{B_r(x)}|\omega(y,t)|\intd y\leq C\frac{1}{m(B_1(x))}\|\omega(t)\|_{1,\,1}\leq C\|\omega(t)\|_{Y^\Theta_{\rm ul}(\RR^2)} \leq C(t).
\end{equation}
So, we just have  to show the case where $r\in]0,1/e[$. Since $\Div u=0$ and $\omega(t,x)=\omega_0(\psi_t^{-1}(x))$, we have
\begin{equation*}
\frac{1}{m(B_r(x))}\int_{B_r(x)}|\omega_0(\psi^{-1}_{t}(y))|\intd y=\frac{1}{m(\psi(B_r(x)))}\int_{\psi(B_r(x))}|\omega_0|\intd y.
\end{equation*}
Moreover, applying the Whitney covering theorem, we find that
\begin{equation}
\begin{split}
\frac{1}{m(\psi(B_r(x)))}\int_{\psi(B_r(x))}|\omega_0|\intd y\leq &\frac{1}{m(\psi(B_r(x)))}\sum_{k} m(2O_k)\frac{1}{m(2O_k)}
\int_{2O_k}\left|\omega_0(y)\right|\intd y\\
\leq& C\frac{1}{m(\psi(B_r(x)))}\sum_{k} L_k\frac{1}{(e^{-k}r)^2}
\|\omega_0\|_{1,\,e^{-k}r}\\
\leq& C\frac{1}{m(\psi(B_r(x)))}\sum_{k} L_k\log^\alpha\big(e+k-\log r\big)\|\omega_0\|_{\mathcal{M}_{\varphi}}.
\end{split}
\end{equation}
We split the series into two parts as follows:
\begin{equation*}
\begin{split}
&\frac{1}{m\left(\psi(B_r(x))\right)}\sum_{k}L_k\log^\alpha\big(e+k-\log r\big)\\
=&\frac{1}{m\left(\psi(B_r(x))\right)}\sum_{k=0}^N L_k\log^\alpha\big(e+k-\log r\big)\\&+\frac{1}{m\left(\psi(B_r(x))\right)}\sum_{k>N} L_k\log^\alpha\big(e+k-\log r\big),
\end{split}
\end{equation*}
where the positive integer $N$ will be fixed later.

 Since
 $$\frac{1}{\log^\alpha(e-\log r)}\frac{1}{m(\psi(B_r(x)))}\int_{\psi(B_r(x))}|\omega_0|\intd y\leq C\|w(0\|_{1,\,1}$$
  for all $r>1$, we focus on the case where $0\leq r\leq1$.

\noindent\textit{Step 1:} We first consider the case where
$$r\leq r_\varphi:=e^{-\exp{2^{\frac{\alpha}{1-\alpha}}}\big(1+\big((1-\alpha)V_{{\rm LogLog}^\alpha}(t)\big)^{\frac{1}{1-\alpha}}\big)}$$
 which implies
\begin{equation}\label{eq.small-r}
\exp{2^{\frac{\alpha}{1-\alpha}}}\big(1+\big((1-\alpha)V_{{\rm LogLog}^\alpha}(t)\big)^{\frac{1}{1-\alpha}}\big)\leq-\log r.
\end{equation}
Denote by $N\ge N_{\alpha}$, a undetermined constant.  For $k\leq N$, a simple calculation yields
\begin{equation}\label{eq.<N}
\begin{split}
&\frac{1}{\log^\alpha(e-\log r)}\frac{1}{m\left(\psi(B_r(x))\right)}\sum_{k=0}^{N}L_k\log^\alpha\big(e+k-\log r\big)\\
\leq&\frac{\log^\alpha\left(e+N-\log r\right)}{\log^\alpha(e-\log r)}\frac{1}{m\left(\psi(B_r(x))\right)}\sum_{k=0}^{N} L_k
\leq \frac{\log^\alpha\left(e+N-\log r\right)}{\log^\alpha(e-\log r)}.
\end{split}
\end{equation}
As for $k> N$, Lemma \ref{Lem-small-2} and inequality \eqref{eq.small-r} allow us to obtain that
\begin{align*}
&\frac{1}{\log^\alpha(e-\log r)}\frac{1}{m\left(\psi(B_r(x))\right)}\sum_{k>N} L_k{\log^\alpha\big(e+k-\log r\big)}\\
\leq&\tfrac{C}r\sum_{k>N}e^{e-(e+k-\log r)^{2^{-\frac{\alpha}{1-\alpha}}}\exp\big(-(1-\alpha)^{\frac{1}{1-\alpha}}\big(\int_0^t\|u(\tau)\|_{{\rm LogLog}^\alpha}\intd\tau\big)^{\frac{1}{1-\alpha}}\big)}\tfrac{\log^\alpha\big(e+k-\log r\big)}{\log^\alpha(e-\log r)}\\
\leq&\tfrac{C}re^{e-(e+N-\log r)^{2^{-\frac{\alpha}{1-\alpha}}}\exp\big(-(1-\alpha)^{\frac{1}{1-\alpha}}\big(\int_0^t\|u(\tau)\|_{{\rm LogLog}^\alpha}\intd\tau\big)^{\frac{1}{1-\alpha}}\big)}\tfrac{\log^\alpha\big(e+N-\log r\big)}{\log^\alpha(e-\log r)}.
\end{align*}
This together with \eqref{eq.<N} enables us to conclude that
\begin{equation*}
\begin{split}
&\frac{1}{\log^\alpha(e-\log r)}\frac{1}{m(\psi(B_r(x)))}\int_{\psi(B_r(x))}|\omega_0|\intd y\\
\leq &C\Big(1+\frac1re^{e-(e+N-\log r)^{2^{-\frac{\alpha}{1-\alpha}}}\exp\big(-(1-\alpha)^{\frac{1}{1-\alpha}}\big(\int_0^t\|u(\tau)\|_{{\rm LogLog}^\alpha}\intd\tau\big)^{\frac{1}{1-\alpha}}\big)}\Big)\\
&\times\frac{\log^\alpha\big(e+N-\log r\big)}{\log^\alpha(e-\log r)}.
\end{split}
\end{equation*}
Taking $N=(e-\log r)^{2^{\frac{\alpha}{1-\alpha}}}\exp{\big(2^{\frac{\alpha}{1-\alpha}}}\big((1-\alpha)V_{{\rm LogLog}^\alpha}(t)\big)^{\frac{1}{1-\alpha}}\big)-e+\log r$, we easily find that
\begin{equation*}
\frac1re^{e-(e+N-\log r)^{2^{-\frac{\alpha}{1-\alpha}}}\exp\big(-(1-\alpha)^{\frac{1}{1-\alpha}}\big(\int_0^t\|u(\tau)\|_{{\rm LogLog}^\alpha}\intd\tau\big)^{\frac{1}{1-\alpha}}\big)}\leq C.
\end{equation*}
Consequently, we have by \eqref{eq.small-r}
\begin{equation}\label{est.<r}
\frac{1}{\log^\alpha(e-\log r)}\frac{1}{m(\psi(B_r(x)))}\int_{\psi(B_r(x))}|\omega_0|\intd y\leq C.
\end{equation}
\noindent\textit{Step 2:} We are now in a position to show the case where $r_\varphi \leq r<1$. By Lemma \ref{lemma-young} and estimate \eqref{est.<r}, we get
\begin{equation*}
\begin{split}
&\frac{1}{\log^\alpha(e-\log r)}\frac{1}{m(\psi(B_r(x)))}\int_{\psi(B_r(x))}|\omega_0|\intd y\\
\leq&C\frac{1}{\log^\alpha(e-\log r)}\frac{1}{m(B_{r_\varphi}(x))}\|\omega_0\circ\psi\|_{1,\,r_\varphi}\\
\leq &C\frac{\log^\alpha(e-\log r_\varphi)}{\log^\alpha(e-\log r)}\\
\leq&C\big(1+V_{{\rm LogLog}^\alpha}(t)\big)^{\frac{\alpha}{1-\alpha}}.
\end{split}
\end{equation*}
Since $\alpha\in]0,1/2]$, we finally get that for  $\varphi(r)=\log^\alpha(e-\log r)$
\begin{equation*}
\|\omega(t)\|_{\mathcal{M}_{\varphi}(\RR^2)}\leq C\big(1+V_{{\rm LogLog}^\alpha}(t)\big)^{\frac{\alpha}{1-\alpha}}\leq C\Big(1+\int_0^t\|u(\tau)\|_{{\rm LogLog}^\alpha}\intd\tau\Big).
\end{equation*}
Thus, our main task is now to show that
\begin{equation}\label{claim-embed}
\|u\|_{V_{{\rm LogLog}^\alpha}}\leq C\|u\|_{L^\infty(\RR^2)}+C\|\omega(t)\|_{\mathcal{M}_{\varphi}(\RR^2)}\quad \text{with}\quad \varphi(r)=\log^\alpha(e-\log r).
\end{equation}
Recall from \cite[Proposition 2.111.]{BCD11}  that
\begin{equation}\label{recall}
\frac1C\|u\|_{V_{{\rm LogLog}^\alpha}}\leq \sup_{j\geq 1}\frac{\|S_j\nabla u\|_{L^\infty(\RR^2)}}{j\log^\alpha(1+j)}\leq C\|u\|_{V_{{\rm LogLog}^\alpha}}.
\end{equation}
By the Bernstein inequality and Lemma \ref{local-bern}, we see that for  $\varphi(r)=\log^\alpha(e-\log r)$
\begin{align*}
\|S_j\nabla u\|_{L^\infty(\RR^2)}\leq& \|S_1\nabla u\|_{L^\infty(\RR^2)}+\sum_{1\leq k<j}\|\Delta_k\nabla u\|_{L^\infty(\RR^2)}\\
\leq& C\| u\|_{L^\infty(\RR^2)}+C\sum_{1\leq k<j}\|\Delta_k\omega\|_{L^\infty(\RR^2)}\\
\leq& C\| u\|_{L^\infty(\RR^2)}+C\sum_{1\leq k<j}2^{2k}\|\omega\|_{1,\,2^{-k}}\\
\leq& C\| u\|_{L^\infty(\RR^2)}+Cj\log^\alpha(e+j)\|\omega\|_{\mathcal{M}_{\varphi}(\RR^2)}.
\end{align*}
This implies claim \eqref{claim-embed} and compeltes the proof of Proposition \ref{prop.local-Spanne}.
\end{proof}

\subsection{A logarithmic loss of regularity in the borderline space $\rm L^\alpha bmo$}

The target of this subsection is to show an estimate  with a logarithmic loss of regularity in the borderline space $\rm L^\alpha bmo$ by developing the classical analysis tools such as the John-Nirenberg inequality.
\begin{proposition}\label{prop-Lbmo}
Let $u_0\in L^\infty(\RR^2)$ and its voriticity $\omega_0\in {\mathrm{L^\alpha bmo}(\RR^2)}$ with $\alpha\in[0,1]$. Assume that $u$ is a smooth solution of \eqref{eq.Euler}. Then we have
\begin{equation}\label{eq.lbm0-result-1}
\|u(t)\|_{L^\infty(\RR^2)}+\|\omega(t)\|_{{\rm L^{\alpha-1}bmo}(\RR^2)} \leq C(t)\left(1+\|\omega_0\|_{\rm L^{\alpha}bmo(\RR^2)}\right)\quad\text{for}\quad\alpha\in[0,1[,
\end{equation}
and
\begin{equation}\label{eq.lbm0-result-2}
\|u(t)\|_{L^\infty(\RR^2)}+\|\omega(t)\|_{{\rm L_{log}bmo}(\RR^2)} \leq C(t)\left(1+\|\omega_0\|_{\rm Lbmo(\RR^2)}\right).
\end{equation}
Here, $C(t)$ is a positive function depending on the initial data.
\end{proposition}
\begin{proof}
According to Lemma \ref{lemma-p-alpha}, we know that $\mathrm{L^\alpha bmo}(\RR^2)$ continuously embeds into $Y^{\Theta}_{\rm ul}(\RR^2)$ with
\begin{equation*}
\Theta(p)=\begin{cases}
p^{1-\alpha},&\alpha\in[0,1[;\\
\log(1+p),&\alpha=1.
\end{cases}
\end{equation*}
Thus, it is easy to verify that $\Theta(p)$ belongs to the class $\mathcal{A}_1$. Then, we immediately obtain, by using Proposition \ref{prop.local-infty}, that
\begin{equation}\label{eq.lbmo-0}
\|u(t)\|_{L^\infty(\RR^2)}+\|\omega(t)\|_{Y^\Theta_{\rm ul}(\RR^2)}\leq C(t).
\end{equation}
For every $r\in]0,1[$ and $x\in\RR^2$, using the H\"older inequality,  we deduce that for $\alpha\in[0,1[$
\begin{align*}
\frac{1}{m(B_r(x))}\int_{B_r(x)}|\omega(t,y)|\intd y\leq&\Big(\frac{1}{m(B_r(x))}\Big)^{\frac1p}\|\omega(t)\|_{p,\,1}\\
\leq&p^{1-\alpha}\Big(\frac{1}{m(B_r(x))}\Big)^{\frac1p}\sup_{1\leq p<\infty}\frac{\|\omega\|_{p,\,1}}{p^{1-\alpha}}.
\end{align*}
According to the arbitrariness of $p$, we  conclude from the estimate \eqref{eq.w-loc-p} and Lemma \ref{lem-take-log} that
\begin{equation*}
\left(-\log r\right)^{\alpha-1}\frac{1}{m(B_r(x))}\int_{B_r(x)}|\omega(t,y)|\intd y\leq C\sup_{1\leq p<\infty}\frac{\|\omega_0\|_{p,\,1}}{p^{1-\alpha}}\cdot\Big(1+\int_0^t\|u(\tau)\|_{L^\infty(\RR^2)}\intd\tau\Big)^2.
\end{equation*}
This together with \eqref{eq.lbmo-0} implies
\begin{equation}\label{eq.lbmo-1}
\left(-\log r\right)^{\alpha-1}\frac{1}{m(B_r(x))}\int_{B_r(x)}|\omega(t,y)|\intd y\leq C(t)\sup_{1\leq p<\infty}\frac{\|\omega_0\|_{p,\,1}}{p^{1-\alpha}}.
\end{equation}
On the other hand, using Lemma \ref{lemma-p-alpha} again, we observe that $\sup_{1\leq p<\infty}\frac{\|\omega_0\|_{p,\,1}}{p^{1-\alpha}}\leq C\|\omega_0\|_{\rm L^\alpha bmo}$ for $\alpha\in]0,1]$. Inserting this into \eqref{eq.lbmo-0} and then taking the supremum over all $r\in]0,1[$ entails
\begin{equation*}
\|\omega(t)\|_{{\rm L^{\alpha-1}bmo}(\RR^2)} \leq C(t)\|\omega_0\|_{{\rm L^{\alpha}bmo}(\RR^2)}.
\end{equation*}
Now, we are in a position to show \eqref{eq.lbm0-result-2}. Firstly, we see that (one may take $p=\log a$)
\begin{equation*}
\inf_{p\geq1}\log(1+p)\cdot a^{p}\leq e\log(1+\log a),\quad \text{for}\quad a>e.
\end{equation*}
Thus, using the same argument as above, we infer that
\begin{align*}
\big(\log(1-\log r)\big)^{-1}\frac{1}{m(B_r(x))}\int_{B_r(x)}|\omega(t,y)|\intd y\leq& C(t)\sup_{1\leq p<\infty}\frac{\|\omega_0\|_{p,\,1}}{\log(1+p)}\\
\leq& C(t) \|\omega_0\|_{\mathrm{Lbmo}(\RR^2)}.
\end{align*}
This completes the proof.
\end{proof}
\subsection{}
In the following part, we mainly focus on the proof of Theorem \ref{prop.reg-pser}.
\begin{proof}[Proof of Theorem \ref{prop.reg-pser}]
By Lemma \ref{lemma-p-alpha} and Proposition \ref{prop.local-infty}, one  concludes that
\begin{equation}\label{eq.final-per-1}
\|u(t)\|_{L^\infty(\RR^2)}+\sup_{1\leq p<\infty}\frac{\|\omega(t)\|_{p,\,1}}{p}\leq C(t).
\end{equation}
Let us fixed a ball $B_r(x)\subset\RR^2$. Our task is now to bound the following quantity
\begin{equation*}
\frac{1}{m\left(B_r(x)\right)}\int_{B_r(x)}\left|\omega(y)-\Avg_{B_r(x)}(\omega)\right|\intd y.
\end{equation*}
In order to do this, we split it into two cases.

\noindent\textbf{Case 1:} $r\leq e^{-4V_{\rm Lip}(t)}$.

Simple calculations lead to
\begin{align*}
&\Big(\frac{1}{m\left(B_r(x)\right)}\int_{B_r(x)}\left|\omega(y)-\Avg_{B_r(x)}(\omega)\right|^p\intd y\Big)^{\frac1p}\\
=&\Big(\frac{1}{m\left(B_r(x)\right)}\int_{\psi(B_r(x))}\left|\omega_0(y)-\Avg_{\psi(B_r(x))}(\omega_0)\right|^p\intd y\Big)^{\frac1p}\\
=&\Big(\frac{1}{m\left(B_r(x)\right)}\int_{\psi(B_r(x))}\Big|\big(\omega_0(y)-\Avg_{B_{r_\psi}(\psi(x))}(\omega_0)\big)\\
&+\big(\Avg_{B_{r_\psi}(\psi(x))}(\omega_0)-\Avg_{\psi(B_r(x))}(\omega_0)\big)\Big|^p\intd y\Big)^{\frac1p}\\
\leq&2\Big(\frac{1}{m\left(B_r(x)\right)}\int_{\psi(B_r(x))}\left|\omega_0(y)-\Avg_{B_{r_\psi}(\psi(x))}(\omega_0)\right|^p\intd y\Big)^{\frac1p}.
\end{align*}
If we take $r_\psi:=re^{V_{\rm Lip}(t)}$, it is easy to verify that
\begin{equation}\label{eq.r-rpsi}
\log r\sim\log r_\psi.
\end{equation}
This enables us to conclude that
\begin{align*}
&\left(-\log r\right)^{\alpha}\Big(\tfrac{1}{m(B_r(x))}\int_{B_r(x)}|\omega(y)-\Avg_{B_r(x)}(\omega)|^p\intd y\Big)^{\frac1p}\\
\leq&2\left(-\log r\right)^{\alpha}\Big(\tfrac{1}{m\left(B_r(x)\right)}\int_{B_{r_\psi}(\psi(x))}|\omega_0(y)-\Avg_{B_{r_\psi}(\psi(x))}(\omega_0)|^p\intd y\Big)^{\frac1p}\\
\leq&C(-\log r_\psi)^{\alpha}\Big(\tfrac{m(B_{r_\psi}(x))}{m(B_r(x))}\Big)^{\frac1p}\Big(\tfrac{1}{m(B_{r_\psi}(x))}\int_{B_{r_\psi}(\psi(x))}
|\omega_0(y)-\Avg_{B_{r_\psi}(\psi(x))}(\omega_0)|^p\intd y\Big)^{\frac1p}\\
\leq&C\big(e^{V_{\rm Lip}(t)}\big)^{\frac2p}\|\omega_0\|_{{\rm L^{\alpha}bmo}_{p}}.
\end{align*}
Moreover, by the H\"older inequality and Corollary \ref{Coro.JN}, we immediately obtain that
\begin{align*}
\|\omega(t)\|_{\rm L^{\alpha}bmo}\leq \|\omega(t)\|_{{\rm L^{\alpha}bmo}_p}\leq& C\big(e^{V_{\rm Lip}(t)}\big)^{\frac2p}\|\omega_0\|_{{\rm L^{\alpha}bmo}_{p}}\\
\leq& Cp\big(e^{V_{\rm Lip}(t)}\big)^{\frac2p}\|\omega_0\|_{{\rm L^{\alpha}bmo}}.
\end{align*}
Combining this with Lemma \ref{lem-take-log} leads to
\begin{align*}
\|\omega(t)\|_{\rm L^{\alpha}bmo}\leq& C\|\omega_0\|_{{\rm L^{\alpha}bmo}}\cdot\inf_{1\leq p<\infty}p\big(e^{V_{\rm Lip}(t)}\big)^{\frac2p}\\
\leq& C\big(1+V_{\rm Lip}(t)\big)\|\omega_0\|_{{\rm L^{\alpha}bmo}}.
\end{align*}

\noindent\textbf{Case 2:} $ e^{-4V_{\rm Lip}(t)}< r<1$.

First of all, we notice that the inequality $e^{-4V_{\rm Lip}(t)}\leq r\leq\frac12$ implies
$
-\log r\leq 4V_{\rm Lip}(t).
$
For $\alpha\in[0,1[$, by Proposition \ref{prop-Lbmo}, we have
\begin{align*}
&\left(-\log r\right)^\alpha\frac{1}{m\left(B_r(x)\right)}\int_{B_r(x)}\left|\omega(y)-\Avg_{B_r(x)}(\omega)\right|\intd y\\
\leq&\left(-\log r\right)\|\omega(t)\|_{{\rm L^{\alpha-1} bmo}(\RR^2)}\\
\leq&C\left(1+V_{\rm Lip}(t)\right)\|\omega_0\|_{{\rm L^\alpha bmo}(\RR^2)}.
\end{align*}
Similarly, one can infer that
\begin{align*}
&\left(-\log r\right)\frac{1}{m\left(B_r(x)\right)}\int_{B_r(x)}\left|\omega(y)-\Avg_{B_r(x)}(\omega)\right|\intd y\\
\leq&C\big(1+V_{\rm Lip}(t)\big)\log\big(1+V_{\rm Lip}(t)\big)\|\omega_0\|_{{\rm Lbmo}(\RR^2)}.
\end{align*}
For $\alpha=1$, we just need to modify the proof of the case $\alpha\in[0,1[$, slightly. In fact, we just need to use estimate \eqref{eq.flow-evolution-2} for $\alpha=1$   instead of $\alpha\in[0,1[$. So we omit it here.

When $\alpha>1$, using Proposition \ref{prop-lbmo-properties}, we know that ${\rm L^\alpha bmo}(\RR^2)$ continuously embeds into $B^0_{\infty,1}(\RR^2)$.
This together with the well-known fact $\|\omega(t)\|_{L^\infty(\RR^d)}\leq\|\omega_0\|_{L^\infty(\RR^2)}$ yields
\begin{align*}
&\left(-\log r\right)^\alpha\frac{1}{m\left(B_r(x)\right)}\int_{B_r(x)}\left|\omega(y)-\Avg_{B_r(x)}(\omega)\right|\intd y\\
\leq&2\left(-\log r\right)^\alpha\frac{1}{m\left(B_r(x)\right)}\int_{B_r(x)}\left|\omega(y)\right|\intd y\\
\leq&2\left(-\log r\right)^\alpha\|\omega(t)\|_{L^\infty(\RR^2)}\\
\leq&C\big(1+V_{\rm Lip}(t)\big)^\alpha\|\omega_0\|_{{\rm L^\alpha bmo}(\RR^2)}.
\end{align*}
Collecting all these estimates completes the proof.
\end{proof}
It follows from  Theorem \ref{prop.reg-pser}, and  from the inclusion relation ${\rm L^\alpha bmo}(\RR^2)\hookrightarrow B^0_{\infty,1}(\RR^2)$ that \eqref{eq.log-crease} is closed for $\alpha>1$.
More precisely:
\begin{corollary}\label{coro.reg-pser}
Let $u_0\in L^\infty(\RR^2)$ and $\omega_0\in {\rm L^{\alpha}bmo}(\RR^2)$ with $\alpha>1$. Assume that $u$ is a smooth solution of \eqref{eq.Euler}. Then there exist a positive smooth function $C(t)$, dependent of  initial data and $\alpha$ such that
\begin{equation*}
\|u(t)\|_{L^\infty(\RR^2)}+\|\omega(t)\|_{{\rm L^{\alpha}bmo}(\RR^2)}\leq C(t).
\end{equation*}
\end{corollary}
\begin{proof}
From \cite{Vi98}, we already know that
\begin{equation*}
\|\omega(t)\|_{B^0_{\infty,1}(\RR^2)}\leq C\big(1+V_{\rm Lip}(t)\big)\|\omega_0\|_{B^0_{\infty,1}(\RR^2)}.
\end{equation*}
Combining this estimate with the inclusion relation ${\rm L^\alpha bmo}(\RR^2)\hookrightarrow B^0_{\infty,1}(\RR^2)$ entails $V_{\rm Lip}(t)\leq C(t)$. Inserting this inequality into \eqref{eq.log-crease} with $\alpha>1$, we get the required result.
\end{proof}

\section{Proof of the main theorems} \label{Sec-Proof}

This section is devoted to the proof of Theorem \ref{Theorem-Exi}, Theorem \ref{Coro-Exi-bmo} and Theorem \ref{Theorem-Exi-UNI} which were formulated in Section~\ref{INTR}. We first restrict our attention to the existence statement. Here we just need to give the proof of Theorem \ref{Theorem-Exi-UNI} for the case $\alpha>1$ because the proof for other $\alpha\in[0,1]$ is very similar. Indeed, since ${\rm L^\alpha bmo}(\RR^2)\hookrightarrow B^0_{\infty,1}(\RR^2)$ for $\alpha>1$,
 it can be obtained by the result in \cite{Ser-1} and  estimate \eqref{eq.Vishik-log}.
 To do this, we shall adopt the following approximate scheme
\begin{equation}\label{eq.approx-Euler}
\left\{\begin{array}{ll}
\partial_tu^n+(u^n\cdot\nabla) u^n+\nabla\Pi^n=0,\quad (t,x)\in\RR^+\times\RR^2,\\
\Div u^n=0,\\
u^n|_{t=0}=S_{n+1}u_0.
\end{array}\right.
\end{equation}
Since $u_0\in L^\infty(\RR^2)$ and $\omega_0\in {\rm L^\alpha bmo}(\RR^2)$, we see that $u^n_0\in C_b^\infty(\RR^2)$  with $C_b^\infty(\RR^2):=\cap_{s>0}B^s_{\infty,\infty}(\RR^2)$. Performing a argument used in  \cite{Ser-1,Ser-2}, we know that there exists a unique global solution $u^n$ to problem \eqref{eq.approx-Euler} satisfying $u^n\in C(\RR^+;{B^s_{\infty,\infty}}(\RR^2))$ for any $s\geq0$. Corollary \ref{coro.reg-pser} enables us to conclude that the family $(u^n, \omega^n)$ is uniformly bounded in $L^\infty(\RR^2)\times {\rm L^{\alpha}bmo}(\RR^2)$ with $\alpha>1$. This regularity implies that $(u^n,\omega^n)$ has a limit $(u,\omega)$ such that
\begin{equation}\label{weak-con}
(u^n,\omega^n)\rightharpoonup (u,\omega)\quad\text{in}\quad L_{\rm loc}^\infty(\RR^2)\times {\rm L^{\alpha}bmo}(\RR^2)
\end{equation}
and that $(u,\omega)\in L^\infty_{\rm loc}(\mathbb{R}^+;L^\infty(\RR^2))\times L^\infty_{\rm loc}(\mathbb{R}^+;{\rm L^\alpha bmo}(\RR^2))$. On the other hand, we see that
\[\partial_tu^n=-(u^n\cdot\nabla) u^n-\nabla\Pi^n.\]
By the Bony paraproduct decomposition, one has  for $\epsilon>0$
\begin{equation*}
\begin{split}
 &\|(u^n\cdot\nabla) u^n\|_{L^\infty_{\rm loc}(\RR^+;B^{-\epsilon}_{\infty,\infty}(\RR^2))}\\
\leq&C\|u_n\|_{L^\infty_{\rm loc}(\RR^+;L^\infty(\RR^2))} \| \nabla  u^n\|_{L^\infty_{\rm loc}(\RR^+;B^{-\epsilon}_{\infty,\infty}(\RR^2))} \\
\leq&C\|u_n\|_{L^\infty_{\rm loc}(\RR^+;L^\infty(\RR^2))} \big(\|u_n\|_{L^\infty_{\rm loc}(\RR^+;L^\infty(\RR^2))} +\| \omega^n\|_{L^\infty_{\rm loc}(\RR^+;{\rm L^{\alpha}bmo}(\RR^2))}\big)<\infty.
\end{split}
\end{equation*}
For any $0<\epsilon<1$, we easily see that $\|\nabla\Pi^n\|_{L^\infty_{\rm loc}(\RR^+;B^{-\epsilon}_{\infty,\infty}(\RR^2))}<\infty$.  This together with the above
estimate shows that $\partial_tu^n\in L^\infty_{\rm loc}(\RR^+;B^{-\epsilon}_{\infty,\infty}(\RR^2))$ for $0<\epsilon<1$. Thus, by using the classical Aubin-Lions argument and performing the standard Cantor's diagonal process, we can deduce that, up to subsequence,
\begin{equation*}
u^n\rightarrow u \quad\text{in}\quad L^\infty_{\rm loc}(\RR^+;B^{-\epsilon}_{\infty,\infty}(\RR^2)).
\end{equation*}
Note that $ u\in L^{\infty}(\mathbb R^2)$  and $\omega\in L^{\alpha}{\rm bmo}(\mathbb R^2)$ imply that $u\in B^1_{\infty,\infty}(\mathbb R^2)$ by Proposition \ref{prop-lbmo-properties}. Thus, we immediately have by  the interpolation theorem
\begin{equation}\label{strong-con}
u^n\rightarrow u \quad\text{in}\quad L^\infty_{\rm loc}(\RR^+;B^{1-\epsilon}_{\infty,\infty}(\RR^2)).
\end{equation}
With the help of \eqref{weak-con} and \eqref{strong-con}, we easily find that the nonlinear term $u^n\cdot\nabla u^{n} $ tends to $u\cdot\nabla u$ in the sense of distribution. This means that the limit $u$ is a weak solution of \eqref{eq.Euler}.

Next,  we turn to show the time continuity  $u\in C(\RR^+;B^{1-\epsilon}_{\infty,1}(\RR^2))$ for any $\epsilon>0$. Since $\partial_tu\in L^\infty_{\rm loc}(\RR^+;B^{-\epsilon}_{\infty,\infty}(\RR^2))$, we have from the mean value formula that for any $t_1,\,t_2\in [0,\infty[$
\begin{equation*}
\|u(t_1)-u(t_2)\|_{B^{-\epsilon}_{\infty,\infty}(\RR^2)}\leq\int_{t_2}^{t_1}\|\partial_tu(\tau)\|_{B^{-\epsilon}_{\infty,\infty}(\RR^2)}\intd\tau\leq C\big|t_1-t_2\big|,
\end{equation*}
which implies $u\in C(\RR^+;B^{-\epsilon}_{\infty,\infty}(\RR^2))$. This together with the fact that $\omega\in L^\infty_{\rm loc}(\RR^+;{\rm L^\alpha bmo}(\RR^2))$ yields that $u\in C(\RR^+;B^{1-\epsilon}_{\infty,1}(\RR^2))$.  Mimicking  the above proof, we can show the existence of solution to Theorem \ref{Theorem-Exi} and  Corollary \ref{Coro-Exi-bmo}.

Next, we focus on the uniqueness statement.  Let $(u,\Pi)$ and $(\tilde u,\tilde\Pi)$ be two solutions of \eqref{eq.Euler} with the same initial data, then the differences $\big(\delta u,\delta\Pi\big):=\big(u-\tilde u,\Pi-\tilde\Pi\big)$ satisfies
\begin{equation}\label{eq.diff-Euler}
\left\{\begin{array}{ll}
       \partial_{t}\delta u+(u\cdot\nabla) \delta u+\nabla \delta\Pi=-\delta (u\cdot\nabla)\tilde u,\quad (t,x)\in\RR^+\times\RR^2,\\
       \Div \delta u=0,\\
      \delta u|_{t=0}=0.
\end{array}\right.
\end{equation}
Here and in what follows, we define
\begin{equation*}
\|\cdot\|_{\overline{Y^{\Theta}_{\rm Lip}}(\RR^d)}:=\|\cdot\|_{L^\infty(\RR^d)}+\|\cdot\|_{Y^\Theta_{\rm Lip}(\RR^d)}.
\end{equation*}
In order to prove the uniqueness of solution, it suffices to show the following proposition.
\begin{proposition}\label{Prop-uni}
Let $u$ and $\tilde u$ belong to $ L^\infty_T\big(\overline{Y_{\rm Lip}^{\Theta}}(\RR^2)\big)$ with $\Theta(x)$ satisfies
$$\int_1^\infty\frac{1}{\Theta(x)}\intd x=\infty.$$
 Assume that $(u,\Pi)$ and $(\tilde u,\tilde\Pi)$ are two solutions of \eqref{eq.Euler} with the same initial data. Then $u\equiv\tilde u$ on interval $[0,T]$.
\end{proposition}
\begin{proof}
The incompressibility condition implies that $\nabla\Pi=B(u,u)$ with $B(u,v):=\nabla\Div|D|^{-2}((u\cdot\nabla) v)$. Thus, we get from \eqref{eq.diff-Euler} that
\begin{equation}\label{eq-diff}
\partial_t\delta u+(u\cdot\nabla)\delta u=B(\delta u,\widetilde u)+B(u,\delta u)-\delta (u\cdot\nabla)\widetilde u.
\end{equation}
Applying the operator $\Delta_q$ to the above equality with $q\geq-1$ yields that
\begin{equation*}
\partial_t\Delta_q\delta u+\big((S_{q+1}u\big)\cdot\nabla)\delta u=\Delta_qB(\delta u,\widetilde u)+\Delta_qB(u,\delta u)-\Delta_q\bigl(\delta (u\cdot\nabla)\widetilde u\bigr)+\digamma_q(u,\delta u),
\end{equation*}
where $\digamma_q=((S_{q+1}u)\cdot\nabla)\delta u-\Delta_{q}((u\cdot\nabla)\delta u)$.

It follows that
\begin{equation*}
\|\Delta_q\delta u\|_{L^\infty}\leq\int_0^t\big\|\Delta_qB(\delta u,\widetilde u)+\Delta_qB(u,\delta u)-\Delta_q\bigl(\delta (u\cdot\nabla)\widetilde u\bigr)+\digamma_q(u,\delta u)\bigr)\big\|_{L^\infty}\intd\tau.
\end{equation*}
By using Lemma \ref{convection-est}, Lemma \ref{commutator-est} and Lemma \ref{pressure-estimate}, we readily get that for all $\varepsilon\in]0,1[$,
\begin{equation}\label{eq.low-esti}
\|\Delta_q\delta u(t)\|_{L^\infty}\leq C\Theta(q+2)2^{q\varepsilon}\int_0^t\Big(\|u(\tau)\|_{ \overline{Y^\Theta_{\rm Lip}}}+\|\widetilde u(\tau)\|_{ \overline{Y^\Theta_{\rm Lip}}}\Big)\|\delta u(\tau)\|_{B^{-\varepsilon}_{\infty,\infty}}\intd\tau.
\end{equation}
By resorting to the low-high frequency decomposition technique, we know
\begin{equation}\label{eq.diff-decom}
\|\delta u\|_{B^{-\varepsilon}_{\infty,\infty}}\leq \sup_{q\leq N}2^{-q\varepsilon}\|\Delta_q\delta u(t)\|_{L^\infty}+\sup_{q>N}2^{-q\varepsilon}\|\Delta_q\delta u(t)\|_{L^\infty},
\end{equation}
where $N$ is a positive integer to be specified later.

For the high frequency part by the Bernstein inequality, we obtain
\begin{align}\label{eq.diff-high}
\sup_{q>N}2^{-q\varepsilon}\|\Delta_q\delta u(t)\|_{L^\infty}\leq&C\sup_{q>N}2^{-q(1+\varepsilon)}\left(\|\Delta_q\nabla u(t)\|_{L^\infty}+\|\Delta_q\nabla \widetilde u(t)\|_{L^\infty}\right)\nonumber\\
\leq&C\sup_{q>N}\Theta(q)2^{-q(1+\varepsilon)}\left(\|u(t)\|_{Y^\Theta_{\rm Lip}}+\|\widetilde u(t)\|_{Y^\Theta_{\rm Lip}}\right)\nonumber\\
\leq&C\sup_{q>N}2^{-q\varepsilon}\left(\|u\|_{L^\infty_tY^\Theta_{\rm Lip}}+\|\widetilde u\|_{L^\infty_tY^\Theta_{\rm Lip}}\right)
\leq C2^{-N\varepsilon}.
\end{align}
Next, we deal with the low frequency part. We observe that \eqref{eq.low-esti}, the properties of $\Theta(\cdot)$,  and the H\"older inequality allow us to get
\begin{equation}\label{eq.diff-low}
\begin{split}
&\sup_{q\leq N}2^{-q\varepsilon}\|\Delta_q\delta u(t)\|_{L^\infty}\\
\leq& C\Theta(N)\int_0^t\Big(\|u(\tau)\|_{ \overline{Y^\Theta_{\rm Lip}}}+\|\widetilde u(\tau)\|_{ \overline{Y^\Theta_{\rm Lip}}}\Big)\|\delta u(\tau)\|_{B^{-\varepsilon}_{\infty,\infty}}\intd\tau\\
\leq&C\Big(\|u\|_{L^\infty_t \overline{Y^\Theta_{\rm Lip}}(\RR^2)}+\|\widetilde u\|_{ L^\infty_t\overline{Y^\Theta_{\rm Lip}}(\RR^2)}\Big)\Theta(N)\int_0^t\|\delta u(\tau)\|_{B^{-\varepsilon}_{\infty,\infty}(\RR^2)}\intd\tau.
\end{split}
\end{equation}
Plugging \eqref{eq.diff-high} and \eqref{eq.diff-low} in \eqref{eq.diff-decom}, we readily obtain
\begin{equation}\label{eq.diff-com}
\|\delta u(t)\|_{B^{-\varepsilon}_{\infty,\infty}}\leq C\Theta(N)\int_0^t\|\delta u(\tau)\|_{B^{-\varepsilon}_{\infty,\infty}}\intd\tau+C2^{-N\varepsilon}.
\end{equation}
According to the continuity of $\|(u,\widetilde u)(t)\|_{B^{-\varepsilon}_{\infty,\infty}}$, we conclude that there exists $T_0\in]0,\min\{T,1\}]$ such that $\sup_{t\in[0,T_0]}\|\delta u(t)\|_{B^{-\varepsilon}_{\infty,\infty}}\leq\frac12$.
Moreover, we may take $N$ satisfying
\begin{equation*}
2^{-N\varepsilon}\sim\int_0^t\|\delta u(\tau)\|_{B^{-\varepsilon}_{\infty,\infty}}\intd\tau,
\end{equation*}
that is,
\begin{equation*}
N\sim N_0=-\frac{1}{\varepsilon}\log\Big(\int_0^t\|\delta u(\tau)\|_{B^{-\varepsilon}_{\infty,\infty}}\intd\tau\Big).
\end{equation*}
Thus, \eqref{eq.diff-com} becomes  for $t\in[0,T_0]$
\begin{equation*}
\begin{split}
\|\delta u(t)\|_{B^{-\varepsilon}_{\infty,\infty}}\leq &C\Theta\left(-\frac{1}{\varepsilon}\log\Big(\int_0^t\|\delta u(\tau)\|_{B^{-\varepsilon}_{\infty,\infty}}\intd\tau\Big)\right)\int_0^t\|\delta u(\tau)\|_{B^{-\varepsilon}_{\infty,\infty}}\intd\tau\\
&+C\int_0^t\|\delta u(\tau)\|_{B^{-\varepsilon}_{\infty,\infty}}\intd\tau.
\end{split}
\end{equation*}
We observe that $\Theta(\cdot)$ fulfills
\begin{equation}\label{eq.claim}
\int_{0}^1\frac{1}{s\Theta(\log1/s)}\intd s=\int_{\infty}^1\frac{t}{ \Theta(\log t)}\intd \frac1t=\int_1^\infty\frac{1}{\Theta(s)}\intd s=\infty.
\end{equation}
By using Osgood's Theorem, we obtain that $\delta u(t)\equiv0$ on the interval $[0,\min\{T_0,1\}]$. Since $u$ and $\widetilde u$ are in $L^\infty\big([0,T];\,\overline{Y_{\rm Lip}^{\Theta}}(\RR^2)\big)$,  we eventually conclude that $u\equiv\widetilde u$ on the whole interval $[0,T]$ via a standard connectivity argument.
\end{proof}
Based on the above, we turn to prove the uniqueness of solutions.
\begin{itemize}
  \item Uniqueness in Theorem \ref{Theorem-Exi}.

  We see that
\begin{equation}
\sup_{2\leq j<\infty}\frac{\|S_{j+1}\nabla u(t)\|_{L^\infty(\RR^2)}}{j\Theta(j)}\leq C\|u(t)\|_{L^\infty(\RR^2)}+C\sup_{2\leq j<\infty}\frac{\|\omega(t)\|_{j,\,1}}{\Theta(j)}.
\end{equation}
Indeed, by using \eqref{eq.local-bern}, we can infer that
\begin{align*}
\|S_{j+1}\nabla u\|_{L^\infty(\RR^2)}\leq& \|\Delta_0u\|_{L^\infty(\RR^2)}+\sum_{1\leq k\leq j}\|\Delta_k\nabla u(t)\|_{L^\infty(\RR^2)}\\
\leq&C\|u(t)\|_{L^\infty(\RR^2)}+C\sum_{1\leq k\leq j}\|\Delta_k\omega(t)\|_{L^\infty(\RR^2)}\\
\leq&C\|u(t)\|_{L^\infty(\RR^2)}+C\sum_{1\leq k\leq j}2^{\frac{2k}{j}}\|\omega(t)\|_{j,\,1}\\
\leq&C\|u(t)\|_{L^\infty(\RR^2)}+Cj\|\omega(t)\|_{j,\,1}.
\end{align*}
Since $\omega\in Y^\Theta_{\rm ul}(\RR^2)$ with $\Theta\in\mathcal{A}_2$, then we have $u\in Y^\Theta_{\rm Lip}(\RR^2)$ with $\Theta$ satisfies $\int_1^\infty\frac{1}{\Theta(x)}\intd x=\infty$.
Moreover, applying Proposition \ref{Prop-uni}, we can conclude that the uniqueness of solution.

  \item Uniqueness in Theorem \ref{Coro-Exi-bmo}.

We  apply \eqref{recall} to Proposition \ref{Prop-uni} to get the uniqueness of solution.

  \item Uniqueness in Theorem \ref{Theorem-Exi-UNI}.

 By Proposition \ref{prop-lbmo-properties}, we know that
\begin{equation*}
\sup_{j\geq2}\frac{\|S_j\nabla u\|_{L^\infty(\RR^2)}}{j\log(1+j)}\leq C\|u\|_{L^\infty(\RR^2)}+\|\omega\|_{{\rm L_{\log}bmo}(\RR^2)}.
\end{equation*}
It is obvious that $\Theta(p)=p\log(1+p)$ satisfies $\int_1^\infty\frac{1}{\Theta(x)}\intd x=\infty$. Whence, it follow the uniqueness of solution from  Proposition \ref{prop-lbmo-properties}.
\end{itemize}
Now, the proof of our results is achieved completely.

\begin{center}{\bf Appendix }\end{center}

\appendix

 \setcounter{section}{5}\setcounter{theorem}{0}\setcounter{equation}{0}

In this section, we first show the generalized John-Nirenberg inequality and its
corollary. Next, we further generalize the estimates for convection
term which play an important role in proving the uniqueness, in the
spirit of \cite{BCD11,Vi00}.

\begin{theorem}[Generalized John-Nirenberg inequality]\label{Thm-John}
Let $f$ belong to ${\rm L^\alpha BMO}(\mathbb{R}^d)$ and $\alpha\in[0,\infty[$. There exist constants $B$ and $b$ dependent of $d$ such that for all cube $Q\subset\mathbb{R}^d$ with $r_Q\in]0,1[$ and $\beta>0$
\begin{equation}\label{eq.John-Nirenberg}
\mu_Q(\beta)\leq B\exp{\Big(-\frac{b\beta(-\log r_{Q})^{\alpha}}{\|f\|_{\rm L^\alpha BMO}}\Big)}\cdot m\left(Q\right),
\end{equation}
where $\mu_Q(\beta)$ be defined by
\begin{equation}
\mu_Q(\beta):=m\left(\big\{x\in Q:\,|f(x)-\Avg_Q(f)|>\beta\big\}\right).
\end{equation}
\end{theorem}
Let us remark that when $\alpha=0$, Theorem \ref{Thm-John} comes back to the classical John-Nirenberg inequality, see for instance \cite{Miao-book04}.
\begin{proof}
When $\beta(-\log r_{Q})^{\alpha}\leq\|f\|_{\rm L^\alpha BMO}$, one can conclude \eqref{eq.John-Nirenberg} by taking $B=e$ and $b=1$.

When $\beta(-\log r_{Q})^{\alpha}>\|f\|_{\rm L^\alpha BMO}$, for a fixed cube $Q_0$, we assume  that $r_{Q_0}<1$ and $\Avg_{Q_0}(f)=0$.
Otherwise, $f(x)$ may be replaced by $g(x)=f(x)-\Avg_{Q_0}(f)$ which fulfills $\Avg_{Q_0}(g)=0$ and $\|g\|_{\rm L^\alpha BMO}=\|f\|_{\rm L^{\alpha}BMO}$.
Also, we may assume that $\|f\|_{\rm L^\alpha BMO}=1$ without loss of generality. Now, Let us define
\begin{equation*}
\mu_{Q}(\beta)=m\left(\big\{x\in Q_0:\quad |f(x)|>\beta\big\}\right):=m\big(E_\beta\big).
\end{equation*}
Thus, for any $\lambda>\|f\|_{\rm L^\alpha BMO}(-\log r_{Q_{0}})^{-\alpha}=(-\log r_{Q_{0}})^{-\alpha}$, it is obvious that
\begin{equation*}
\frac{1}{m(Q_0)}\int_{Q_0}\big|f(x)\big|\intd x<\lambda.
\end{equation*}
Moreover, by the Calder\'on-Zygmund decomposition theorem, we get
\begin{equation*}
Q_0=F^\lambda\bigcup\left(\cup_{k=1}^\infty Q_k^\lambda\right),
\end{equation*}
where, the cubes $Q_k^\lambda$  are mutually disjoint and satisfies
\begin{equation*}
|f(x)|\leq \lambda,\quad \text{for}\quad \text{a.e}\quad x\in F^\lambda,
\end{equation*}
and
\begin{equation*}
\lambda<\frac{1}{m\left(Q_k^\lambda\right)}\int_{Q_k^\lambda}\big|f(x)\big|\intd x\leq 2^d\lambda.
\end{equation*}
According to construction of $Q_k^\lambda$, there exists a mother cube $\bar Q^\lambda_k$ such that $Q_k^\lambda$ is one of $2^d$   children cubes   $\bar Q^\lambda_k$ satisfying
\begin{equation*}
\frac{1}{m\left(\bar Q^\lambda_k\right)}\int_{\bar Q^\lambda_k}\big|f(x)\big|\intd x\leq\lambda.
\end{equation*}
Thus, it follows that
\begin{equation*}
\begin{split}
\frac{1}{m\left(Q^\lambda_k\right)}\int_{Q^\lambda_k}\big|f(x)\big|\intd x\leq&\frac{1}{m\left(Q^\lambda_k\right)}\int_{\bar Q^\lambda_k}\big|f(x)-\Avg_{\bar Q_k^\lambda}(f)\big|\intd x+\big|\Avg_{\bar Q_k^\lambda}(f)\big|\\
\leq&2^{d}(-\log r_{2Q_k})^{-\alpha}\|f\|_{\rm L^\alpha BMO}+\lambda\\
\leq&2^d(-\log r_{Q_k})^{-\alpha}+\lambda.
\end{split}
\end{equation*}
Suppose that $\zeta\geq\lambda$, in the same way as above, it is easy to construct Calder\'on-Zygmund decomposition of $Q_0$ as follows:
\begin{equation*}
Q_0=F^\zeta\bigcup\left(\cup_{j=1}^\infty Q_j^\zeta\right).
\end{equation*}
Clearly,
\begin{equation*}
\bigcup_{j=1}^\infty Q_j^\zeta\subset\bigcup_{k=1}^\infty Q_k^\lambda.
\end{equation*}
Whence, for each cube $Q_j^\zeta$, there exists a cube $Q_k^\lambda$ such that $Q_j^\zeta\subset Q_k^\lambda$. Now let us take $\zeta=\lambda+2^{d+1}(-\log r_{Q_k})^{-\alpha}$ and denote
\begin{equation*}
Q^{\zeta,\,\lambda}_{j,\,k}:=\bigcup_{j,\,\,Q_j^\zeta\subset Q_k^\lambda}Q_j^\zeta\subset Q_k^\lambda.
\end{equation*}
Thus, we easily find that
\begin{align*}
\zeta=&\lambda+2^{d+1}(-\log r_{Q_k})^{-\alpha}\\
<&\frac{1}{m\left(Q^{\zeta,\,\lambda}_{j,\,k}\right)}\int_{Q^{\zeta,\,\lambda}_{j,\,k}}\big|f(x)\big|\intd x\\
\leq&\frac{1}{m\left(Q^{\zeta,\,\lambda}_{j,\,k}\right)}\int_{Q^{\zeta,\,\lambda}_{j,\,k}}\big|f(x)-\Avg_{Q_k^\lambda}(f)\big|\intd x+\big|\Avg_{Q_k^\lambda}(f)\big|\\
\leq&\frac{m\left(Q_k^\lambda\right)}{m\left(Q^{\zeta,\,\lambda}_{j,\,k}\right)}\big(-\log r_{Q_{k}}\big)^{-\alpha}\|f\|_{\rm L^\alpha BMO}+2^d(-\log r_{Q_k})^{-\alpha}+\lambda,
\end{align*}
from which, it follow that
\begin{equation*}
 m\left(Q^{\zeta,\,\lambda}_{j,\,k}\right)\leq 2^{-d}m\left(Q_k^\lambda\right).
\end{equation*}
Consequently,
\begin{equation}\label{eq.JN-1}
 \sum_{j}m\left(Q^{\zeta}_{j}\right)\leq 2^{-d}\sum_{k}m\left(Q_k^\lambda\right),
 \end{equation}
 and
 \begin{equation*}
 \zeta-\lambda=2^{d+1}(-\log r_{Q_k})^{-\alpha}.
\end{equation*}
We observe that $\beta>\|f\|_{\rm L^\alpha BMO}(-\log r_{Q_{0}})^{-\alpha}=(-\log r_{Q_{0}})^{-\alpha}$. If, moreover, we take $r=\big[\frac{\beta-(-\log r_{Q_{0}})^{-\alpha}}{2^{d+1}(-\log r_{2Q_k})^{-\alpha}}\big]$ and $\zeta=(-\log r_{Q_{0}})^{-\alpha}+2^{d+1}r(-\log r_{2Q_k})^{-\alpha}$, then we have $(-\log r_{Q_{0}})^{-\alpha}\leq\zeta\leq\beta$ which implies $E_\beta\subset E_\zeta$. Since
\begin{equation*}
f(x)\leq\zeta,\quad\text{for a.e}\quad x\in F^\zeta,
\end{equation*}
we get $E_\zeta=\cup_jQ_j^\zeta$. This together with \eqref{eq.JN-1} enables us to infer that
\begin{equation*}
\sum_{j}m\left(Q_j^\zeta\right)\leq 2^{-rd}\sum_{k}m\left(Q_j^1\right).
\end{equation*}
Since $E_\beta=\big\{x\in Q_0:\quad |f(x)|>\beta\big\}$, we have
\begin{equation*}
m\left(E_\beta\right)\leq m\left(E_\zeta\right)\leq \sum_{j}m\left(Q_j^\zeta\right)\leq 2^{-rd}\sum_{k}m\left(Q_j^1\right).
\end{equation*}
This entails
\begin{equation}\label{eq.JN-2}
\mu_{Q_0}(\beta)\leq 2^{-rd}m\left(Q_0\right).
\end{equation}
Taking $B=2^{d(1+2^{-d-1})}$ and $b=\log(d2^{-d-1})$ and using \eqref{eq.JN-2}, we eventually obtain that
\begin{align*}
B m\left(Q_0\right)\exp{(-b(-\log r_{Q_{0}})^{\alpha}\beta)}=&2^{d} m\left(Q_0\right)\cdot 2^{d\big(1-(-\log r_{Q_{0}})^{\alpha}\beta\big)\cdot2^{-d-1}}\\
\geq&2^{d}m\left(Q_0\right)\cdot 2^{d(r+1)}\geq\mu_{Q_0}.
\end{align*}
This completes the proof.
\end{proof}
\begin{corollary}\label{Coro.JN}
Suppose that $1\leq q<\infty$ and $\alpha\in[0,\infty[$. Then  ${\rm L^\alpha BMO}_q={\rm L^\alpha BMO}$ and
there exists a number $C>0$ such that
$$\|f\|_{\rm L^\alpha BMO}\leq \|f\|_{{\rm L^\alpha BMO}_q}\leq Cq\|f\|_{\rm L^\alpha BMO}.$$
\end{corollary}
\begin{proof}
We just need to show that ${\rm L^\alpha BMO}_q={\rm L^\alpha BMO}$ for $1<q<\infty$. For an arbitrary cube $Q_r(x)\subset\mathbb{R}^d$ with $r\in]0,1[$, by the H\"older inequality, we have
\begin{align*}
&\frac{1}{m\left(Q_r(x)\right)}\int_{Q_r(x)}\big |f(x)-\Avg_{Q_r(x)}(f)\big|\intd x\\
\leq&\left(\frac{1}{m\left(Q_r(x)\right)}\int_{Q_r(x)}\big |f(x)-\Avg_{Q_r(x)}(f)\big|^{q}\intd x\right)^{\frac{1}{q}}\nonumber\\
\leq&(-\log r)^\alpha\|f\|_{{\rm L^\alpha BMO}_q}.
\end{align*}
On the other hand, the generalized John-Nirenberg inequality ensures that for $r\in]0,1[$
\begin{align}
&\frac{(-\log r)^{\alpha q}}{m\left(B_r(x)\right)}\int_{Q}\big| f(x)-\Avg_Q(f)\big|^{q}\intd x\nonumber\\
=&\frac{q(-\log r)^{\alpha q}}{m\left(B_r(x)\right)}\int_0^\infty\xi^{q-1}\mu_{Q}(\xi)\intd\xi\nonumber\\
\leq& Bq(-\log r)^{\alpha q}\int_0^\infty\xi^{q-1}\exp{\Big(-\frac{b(-\log r)^{\alpha}\xi}{\|f\|_{\rm L^\alpha BMO}}\Big)}\intd\xi\nonumber\\
=& Bq\int_0^\infty\xi^{q-1}\exp{\Big(-\frac{b\xi}{\|f\|_{\rm L^\alpha BMO}}\Big)}\intd\xi\nonumber\\
=&Bqb^{-q}\Gamma(q)\|f\|_{\rm L^\alpha BMO}^q.
\end{align}
This implies $\|f\|_{{\rm L^\alpha BMO}_q}\leq Cq\|f\|_{\rm L^\alpha BMO}$.
\end{proof}
In the following part, we always assume that $\Theta$ is a modulus of continuity  for the convenience of presentation.
\begin{lemma}\label{convection-est}
Let $1\leq p\leq\infty$ and $u$ satisfies $\mathrm{div}\, u=0$.
Then there exist a positive constant $C$ such that
\begin{equation}
\big\|\Delta_q\big((u\cdot\nabla) v)\big\|_{L^p}\leq C\Theta(q+2)2^{q\varepsilon}\big(\|v\|_{L^p}+\|v\|_{Y^\Theta_{\rm Lip}}\big)\|u\|_{B^{-\varepsilon}_{p,\infty}},
\quad \forall\,\varepsilon\in]0,1[.
\end{equation}
\end{lemma}
\begin{proof}
Thanks to the Bony para-product decomposition, one can write
\begin{equation*}
(u\cdot\nabla) v=T_{u_j}\partial_jv_i+T_{\partial_jv_i}u_j+\partial_j\mathcal{R}(u_j,v_i).
\end{equation*}
For the first term $T_{u_j}\partial_jv_i$,
\begin{align*}
\big\|\Delta_qT_{u_j}\partial_jv_i\big\|_{L^p}\leq& C\sum_{|k-q|\leq5}\big\|\Delta_q(S_{k-1}u_j\Delta_k\partial_jv_i)\big\|_{L^p}\nonumber\\
\leq&C\sum_{|k-q|\leq5}\|S_{k-1}u_j\|_{L^p}\|\Delta_k\partial_jv_i\|_{L^\infty}\nonumber\\
\leq&C\sum_{|k-q|\leq5}\|S_{k-1}u_j\|_{L^p}\|S_{q+5}\partial_jv_i\|_{L^\infty}\nonumber\\
\leq&C\Theta(q)2^{q\varepsilon}\|u_j\|_{B^{-\varepsilon}_{p,\infty}}\|v\|_{Y^\Theta_{\rm Lip}}.
\end{align*}
In a similar fashion as above, we have
\begin{align*}
\big\|\Delta_qT_{\partial_jv_i}u_j\big\|_{L^p}\leq& C\sum_{|k-q|\leq5}\big\|\Delta_q(S_{k-1}\partial_jv_i\Delta_ku_j)\big\|_{L^p}\nonumber\\
\leq&C\sum_{|k-q|\leq5}\|S_{k-1}\partial_jv_i\|_{L^\infty}\|\Delta_ku_j\|_{L^p}\nonumber\\
\leq&C\Theta(q)2^{q\varepsilon}\|u_j\|_{B^{-\varepsilon}_{p,\infty}}\|v\|_{Y^\Theta_{\rm Lip}}.
\end{align*}
Finally, the remainder term can be bounded as follows:
\begin{align}\label{eq.appen-1}
&\big\|\Delta_q\partial_j\mathcal{R}(u_j,v_i)\big\|_{L^p}\\
\leq&C2^{q}\|\Delta_q\mathcal{R}(u_j,v_i)\|_{L^p}\nonumber\\
\leq&C2^{q}\sum_{k\geq q-2}\|\widetilde\Delta_ku_j\Delta_kv_i\|_{L^p}\nonumber\\
\leq&C2^{q\epsilon}\sum_{k\geq q-2}2^{(q-k)(1-\epsilon)}2^{-k\epsilon}\|\widetilde\Delta_ku_j\|_{L^p}2^{k}\|\Delta_kv_i\|_{L^\infty}\nonumber\\
\leq&C2^{q\epsilon}\big(\|\Delta_{-1}v\|_{L^\infty}+\|v\|_{Y^\Theta_{\rm Lip}}\big)\sum_{k\geq q-2}2^{(q-k)(1-\epsilon)}2^{-k\epsilon}\Theta(k+2)\|\widetilde\Delta_ku_j\|_{L^p}.\nonumber
\end{align}
Since $\Theta$ is a modulus of continuity, we know that $\Theta(2h)\leq C_\Theta\cdot\Theta(h)$ for all $h\geq1$. Thus, we have
\begin{align}\label{eq.appen-2}
&\sum_{k\geq q-2}2^{(q-k)(1-\epsilon)}\Theta(k+2)2^{-k\epsilon}\|\widetilde\Delta_ku_j\|_{L^p}\nonumber\\
\leq &C\Theta(q+2)\sum_{k\geq q-2}2^{(q-k)(1-\epsilon)}\Big(1+\frac{k-q}{q}\Big)^{\log_2C_\Theta}2^{-k\epsilon}\|\widetilde\Delta_ku_j\|_{L^p}\nonumber\\
\leq&C\Theta(q+2)\|u\|_{B^{-\epsilon}_{p,\infty}}.
\end{align}
In the second line of \eqref{eq.appen-2}, we have used that for $k\geq q\leq1$,
\begin{align}\label{eq.delta-2}
\Theta(k)\leq C_\Theta\cdot\Theta\Big(q\frac{k}{2q}\Big)\leq C_\Theta^{}\Theta\Big(q\frac{k}{2^aq}\Big)\leq \Big(\frac{k}{q}\Big)^{\log_2C_\Theta}\Theta(q)
\end{align}
with $a=\log_2\frac{k}{q}$.

Inserting \eqref{eq.appen-2} into \eqref{eq.appen-1} leads to
\begin{equation*}
\big\|\Delta_q\partial_j\mathcal{R}(u_j,v_i)\big\|_{L^p}
\leq C\Theta(q+2)2^{q\epsilon}\big(\|v\|_{L^p}+\|v\|_{Y^\Theta_{\rm Lip}}\big)\|u\|_{B^{-\epsilon}_{p,\infty}}.
\end{equation*}
Collecting all these estimates yields the desired result.
\end{proof}
\begin{lemma}\label{commutator-est}
Let $\epsilon\in]-1,1[$ and $1\leq p\leq\infty$. Assume that $u$ be a divergence free vector field over $\mathbb{R}^d$. There exists a positive constant $C$ such that for all $q\geq-1$
\begin{equation}
\big\|R_q(u,v)\big\|_{L^p}\leq C\Theta(q+2)2^{-q\epsilon}\|u\|_{Y^\Theta_{\rm Lip}}\|v\|_{B^\epsilon_{p,\infty}},
\end{equation}
where $R_q(u,v):= S_{q+1}(u\cdot\nabla)\Delta_qv-\Delta_q((u\cdot\nabla) v)$.
\end{lemma}
\begin{proof}
We first decompose $R_q(u,v)$ as follows:
\begin{align}
R_q(u,v)=&S_{q+1}(u\cdot\nabla)\Delta_qv-\Delta_q(S_{q+1}(u\cdot\nabla) v)-\Delta_q\big(({\rm I_d}-S_{q+1})(u\cdot\nabla) v\big)\nonumber\\
=&-[\Delta_q,S_{q+1}\bar u]\cdot\nabla v-[\Delta_q,S_1u]\cdot\nabla v-\Delta_q\big(({\rm I_d}-S_{q+1})(u\cdot\nabla) v\big),
\end{align}
where $\bar u=({\rm I_d}-S_1)u$.

Note that
\begin{align*}
[\Delta_q,S_{q+1}\bar u]\cdot\nabla v=&[\Delta_q,S_{q+1}T_{\bar u_i}]\partial_i v+\Delta_q\big(T_{\partial_iv}S_{q+1}\bar u_i\big)+
\Delta_q\big(R(S_{q+1}\bar u_i,\partial_i v)\big)\nonumber\\
&-T_{\Delta_q\partial_iv}S_{q+1}\bar u_i-R(S_{q+1}\bar u_i,\Delta_q\partial_i v)\nonumber\\
:=&R^{1}_q(u,v)+R^{2}_q(u,v)+R^{3}_q(u,v)+R^{4}_q(u,v)+R^{5}_q(u,v)
\end{align*}
and
\begin{align*}
&\Delta_q\big(({\rm I_d}-S_{q+1})(u\cdot\nabla) v\big)\\
=&\Delta_q\big(T_{({\rm I_d}-S_{q+1})u_i}\partial_iv\big)+\Delta_q\big(T_{\partial_iv}({\rm I_d}-S_{q+1})u_i\big)+
\Delta_qR\big(({\rm I_d}-S_{q+1})u_i,\partial_iv\big)\nonumber\\
:=&R^{6}_q(u,v)+R^{7}_q(u,v)+R^{8}_q(u,v).
\end{align*}
First of all, we observe that
\begin{align*}
&[S_{q'-1}S_{q+1}\bar u_i,\Delta_q]\partial_i\Delta_{q'}v\nonumber\\
=&2^{q'd}\int_{\mathbb{R}^d}\big(S_{q'-1}S_{q+1}\bar u_i(x)-S_{q'-1}S_{q+1}\bar u_i(x-y)\big)\varphi\big(2^{q'}(x-y)\big)\partial_i\Delta_{q'}v(y)\intd y\nonumber\\
=&-2^{q'd}\int_{\mathbb{R}^d}\int_0^1\partial_kS_{q'-1}S_{q+1}\bar u_i\big(x+(1-\tau)(x-y)\big)\intd\tau\\
&\qquad\qquad\quad  \times (x_k-y_k)\varphi\big(2^{q'}(x-y)\big)\partial_i\Delta_{q'}v(y)\intd y\nonumber\\
=&-2^{q'(d-1)}\int_{\mathbb{R}^d}\int_0^1\partial_kS_{q'-1}S_{q+1}\bar u_i\big(x+(1-\tau)(x-y)\big)\intd\tau\\
&\qquad\qquad\quad \times 2^{q'}(x_k-y_k)\varphi\big(2^{q'}(x-y)\big)\partial_i\Delta_{q'}v(y)\intd y,
\end{align*}
where used the relation $\Delta_{q'}f=2^{q'd}\int_{\mathbb{R}^d}\varphi\big(2^{q'}(x-y)\big)f(y)\intd y$.

Therefore, we immediately get that
\begin{align*}
\|R^{1}_q(u,v)\|_{L^p}\leq&C \sum_{|q'-q|\leq4}2^{-q'}\|\partial_k S_{q'-1}\bar u_i\|_{L^\infty}
\|\partial_i\Delta_{q'}v\|_{L^p}\int_{\mathbb{R}^d}|x\varphi(x)|\intd x\nonumber\\
\leq&C \sum_{|q'-q|\leq4}\|\partial_k S_{q'-1}u_i\|_{L^\infty}\|\Delta_{q'}v\|_{L^p}\nonumber\\
\leq&C\Theta(q+2) \|u\|_{Y^\Theta_{\rm Lip}}\sum_{|q'-q|\leq4}\|\Delta_{q'}v\|_{L^p}.
\end{align*}
In a similar fashion as in proof of $R^1_q(u,v)$, we can bounded $[\Delta_q,S_1u]\cdot\nabla v$ as follows:
\begin{align*}
\big\|[\Delta_q,S_1u]\cdot\nabla v\big\|_{L^p}\leq& C\sum_{|q-q'|\leq2}\|\nabla S_1 u\|_{L^\infty}\|\Delta_{q'}v\|_{L^p}\nonumber\\
\leq&C \|u\|_{Y^{\Theta}_{\rm Lip}}\sum_{|q-q'|\leq2}\|\Delta_{q'}v\|_{L^p}.
\end{align*}
For the second term $R_q^2(u,v)$, by the same way as in proving Lemma \eqref{convection-est}, we infer that
\begin{align*}
\big\|R_q^2(u,v)\big\|_{L^p}\leq&C\sum_{|q'-q|\leq4}\|\Delta_{q'}\bar u_i\|_{L^\infty}\big\|S_{q'-1}\partial_iv\big\|_{L^p}\nonumber\\
\leq&C\sum_{|q'-q|\leq4}2^{q-q'}\|\Delta_{q'}\nabla u_i\|_{L^\infty}\sum_{-1\leq k\leq q'-2}2^{k-q}\big\|\Delta_{k}v\big\|_{L^p}\nonumber\\
\leq&C\Theta(q+2) \sum_{|q'-q|\leq4}\frac{\Theta(q'+2)}{\Theta(q+2)}2^{q-q'}\frac{1}{\Theta(q'+2)}\|\Delta_{q'}\nabla u_i\|_{L^\infty}\\
&\times\sum_{-1\leq k\leq q+2}2^{k-q}\big\|\Delta_{k}v\big\|_{L^p}\nonumber\\
\leq&C\Theta(q+2) \|u\|_{Y^\Theta_{\rm Lip}}\sum_{-1\leq k\leq q+2}2^{k-q}\big\|\Delta_{k}v\big\|_{L^p}.
\end{align*}
Similarly, we can conclude that
\begin{equation*}
\left\|R_q^4(u,v)\right\|_{L^p}\leq C\Theta(q+2) \|u\|_{Y^\Theta_{\rm Lip}}\sum_{-1\leq k\leq q+2}2^{k-q}\big\|\Delta_{k}v\big\|_{L^p}.
\end{equation*}
Since $\epsilon\in]0,1[$, the reminder term $R_q^3(u,v)$ can be bounded by
\begin{align*}
\|\partial_i\Delta_q\big(R(S_{q+1}\bar u_i, v)\big)\|_{L^p}
\leq& C\sum_{q'\geq q-3}2^{q}\|\Delta_{q'}v\|_{L^p}\|\widetilde\Delta_{q'}S_{q+1} \bar u_i\|_{L^\infty}\nonumber\\
\leq& C\sum_{q'\geq q-3}2^{q-q'}\|\Delta_{q'}v\|_{L^p}\|S_{q+1} \nabla \bar u_i\|_{L^\infty}\nonumber\\
\leq&C\Theta(q+2) 2^{-q\epsilon}\|u\|_{Y^\Theta_{\rm Lip}}\|v\|_{B^\epsilon_{p,\infty}}.
\end{align*}
Similarly, we can conclude that
\begin{equation*}
\|R_q^5(u,v)\|_{L^p}
\leq C\Theta(q+2) 2^{-q\epsilon}\|u\|_{Y^\Theta_{\rm Lip}}\|v\|_{B^\epsilon_{p,\infty}}.
\end{equation*}
It remains  to bound the last three terms $R_q^6(u,v),\,R_q^7(u,v)$ and $R_q^8(u,v)$. Thanks to the  support property
and the H\"older inequality, one has
\begin{align*}
\|R_q^6(u,v)\|_{L^p}
\leq &C\sum_{|q'-q|\leq5}\|S_{q'-1}({\rm I_d}-S_{q+1})u_i\|_{L^\infty}\|\Delta_{q'}\partial_iv\|_{L^p}\nonumber\\
\leq&C\sum_{|q'-q|\leq5}2^{-q'}\|S_{q'-1}({\rm I_d}-S_{q+1})\nabla u_i\|_{L^\infty}\|\Delta_{q'}\partial_iv\|_{L^p}\nonumber\\
\leq&C\sum_{|q'-q|\leq5}2^{q-q'}\|S_{q'-1}\nabla u_i\|_{L^\infty}\|\Delta_{q'}v\|_{L^p}\nonumber\\
\leq&C\Theta(q+2)\|u\|_{Y^\Theta_{\rm Lip}}\sum_{|q'-q|\leq5}2^{q-q'}\|\Delta_{q'}v\|_{L^p}.
\end{align*}
For the term $R_q^7(u,v)$, by the H\"older inequality, we obtain
\begin{align*}
\|R_q^7(u,v)\|_{L^p}
\leq &C\sum_{|q'-q|\leq5}\|S_{q'-1}\partial_iv\|_{L^p}\|\Delta_{q'}({\rm I_d}-S_{q+1})u_i\|_{L^\infty}\nonumber\\
\leq&C\sum_{|q'-q|\leq5}2^{q'-q}\,\sum_{-1\leq k\leq q'-2}2^{k-q'}\|\Delta_{k}v\|_{L^p}\|\Delta_{q'}\nabla u_i\|_{L^\infty}\nonumber\\
\leq&C\Theta(q+2) \|u\|_{Y^\Theta_{\rm Lip}}\sum_{-1\leq k\leq q+3}2^{k-q}\|\Delta_{k}v\|_{L^p}.
\end{align*}
As for the last term $R_q^8(u,v)$, by the H\"older inequality and \eqref{eq.delta-2}, we obtain
\begin{align*}
\|R_q^8(u,v)\|_{L^p}
\leq&C\sum_{q'\geq q-3}2^{q}\|\Delta_{q'}v\|_{L^p}\|\widetilde\Delta_{q'}({\rm I_d}-S_{q+1})u_i\|_{L^\infty}\nonumber\\
\leq&C\sum_{q'\geq q-3}2^{q}\|\Delta_{q'}v\|_{L^p}\|\widetilde{\dot{\Delta}}_{q'} u_i\|_{L^\infty}\nonumber\\
\leq&C\sum_{q'\geq q-3}2^{q-q'}\|\Delta_{q'}v\|_{L^p}\|\widetilde{\dot{\Delta}}_{q'}\nabla u_i\|_{L^\infty}\nonumber\\
\leq&C\Theta(q+2) 2^{-q\epsilon}\|u\|_{Y^\Theta_{\rm Lip}}\|v\|_{B^\epsilon_{p,\infty}}.
\end{align*}
Combining all these bounds yields the desired result.
\end{proof}
\begin{lemma}\label{pressure-estimate}
Let $B(u,v)=\nabla\Div|D|^{-2}((u\cdot\nabla) v)$ and $\varepsilon\in]0,1[$.
Then,  for $q\geq-1$,
\begin{equation}
\|\Delta_qB(u,v)\|_{L^\infty}\leq C\Theta(q+2)2^{q\varepsilon}\min\Big\{\|u\|_{B^{-\varepsilon}_{\infty,\infty}}
\|v\|_{\overline{Y^{\Theta}_{\rm Lip}}},\,\|v\|_{B^{-\varepsilon}_{\infty,\infty}}\|u\|_{ \overline{Y^{\Theta}_{\rm Lip}}}\Big\}.
\end{equation}
\end{lemma}
\begin{proof}
Thanks to the Bony-paraproduct decomposition, one  decompose  $\nabla\Pi=B(u,u)$ as follows:
$$
B(u,v):=B_1(u,v)+B_2(u,v)+B_3(u,v)+B_4(u,v)+B_5(u,v).
$$
Denoting by $\theta\in \mathcal{D}(B(0,2))$  a smooth function with value $1$ on the ball $B(0,1)$, we have
\begin{align*}
&B_1(u,v):=\nabla(-\Delta)^{-1}T_{\partial_iu_j}\partial_jv_i,\\
&B_2(u,v):=\nabla(-\Delta)^{-1}T_{\partial_jv_i}\partial_iu_j,\\
&B_3(u,v):=\partial_i\partial_j\nabla(-\Delta)^{-1}({\rm I_d}-\Delta_{-1})\mathcal{R}(u_j,v_i),\\
&B_4(u,v):=\theta E_d\ast\nabla\partial_i\partial_j\Delta_{-1}\mathcal{R}(u_j,v_i), \quad\text{with}\quad E_d=c_d\frac{1}{|x-y|^{d-2}},\\
&B_5(u,v):=\widetilde E_d\ast\nabla\partial_i\partial_j\Delta_{-1}\mathcal{R}(u_j,v_i)\quad\text{with}\quad \widetilde E_d:=(1-\theta)E_d.
\end{align*}
First, we tackle with the para-product terms. By using the H\"older inequality, we obtain
\begin{align}\label{eq-l-p}
\|\Delta_qB_1(u,v)\|_{L^\infty}\leq& \big\|\Delta_q\nabla(-\Delta)^{-1}\big(T_{\partial_iu_j}\partial_jv_i\big)\big\|_{L^\infty}\nonumber\\
\leq& C\sum_{|k-q|\leq 5}\big\|\Delta_q\nabla(-\Delta)^{-1}\big(S_{k-1}\partial_iu_j\Delta_k\partial_jv_i\big)\big\|_{L^\infty}\nonumber\\
\leq& C\sum_{|k-q|\leq 5}2^{-q}\big\|S_{k-1}\partial_iu_j\big\|_{L^\infty}\big\|\Delta_k\partial_jv_i\big\|_{L^\infty}.
\end{align}
Furthermore, the fact
\begin{equation*}
\big\|\Delta_k\partial_jv_i\big\|_{L^\infty}=\big\|S_{k+2}\Delta_k\partial_jv_i\big\|_{L^\infty}
\leq C\big\|S_{k+2}\partial_jv_i\big\|_{L^\infty}\leq C\Theta(k+2)\|v\|_{Y^\Theta_{\rm Lip}}
\end{equation*}
yields
\begin{align}\label{eq-l-p-1}
&\sum_{|k-q|\leq 5}2^{-q}\big\|S_{k-1}\partial_iu_j\big\|_{L^\infty}\big\|\Delta_k\partial_jv_i\big\|_{L^\infty}\\
\leq&C\Theta(q+2)2^{q\epsilon}\sum_{|k-q|\leq 5}2^{-(q-k)(1+\epsilon)}\tfrac{\Theta(k+2)}{\Theta(q+2)}2^{-k\epsilon}\big\|S_{k-1}u_j\big\|_{L^\infty}
\tfrac{1}{\Theta(k+2)}\big\|\Delta_k\partial_jv_i\big\|_{L^\infty}\nonumber\\
\leq&C\Theta(q+2)2^{q\varepsilon}\|v\|_{Y^\Theta_{\rm Lip}}\|u\|_{B^{-\varepsilon}_{\infty,\infty}}.\nonumber
\end{align}
On the other hand, we see that
\begin{align}\label{eq-l-p-2}
&\sum_{|k-q|\leq 5}2^{-q}\big\|S_{k-1}\partial_iu_j\big\|_{L^\infty}\big\|\Delta_k\partial_jv_i\big\|_{L^\infty}\\
\leq&C2^{q\epsilon}\sum_{|k-q|\leq 5}2^{-(q-k)(1+\epsilon)}\big\|S_{k-1}\partial_iu_j\big\|_{L^\infty}2^{-k\epsilon}\big\|\Delta_kv_i\big\|_{L^\infty}\nonumber\\
\leq&C\Theta(q+2)2^{q\varepsilon}\|u\|_{Y^\Theta_{\rm Lip}}\|v\|_{B^{-\varepsilon}_{\infty,\infty}}.\nonumber
\end{align}
Inserting \eqref{eq-l-p-1} and \eqref{eq-l-p-2} in \eqref{eq-l-p}, one obtains
\begin{align*}
\|\Delta_qB_1(u,v)\|_{L^\infty}\leq C\Theta(q+2)2^{q\varepsilon} \min\Big\{\|u\|_{B^{-\varepsilon}_{\infty,\infty}}
\|v\|_{Y^\Theta_{\rm Lip}},\,\|v\|_{B^{-\varepsilon}_{\infty,\infty}}\|u\|_{Y^\Theta_{\rm Lip}}\Big\}.
\end{align*}
Similarly, we have
\begin{align*}
\|\Delta_qB_2(u,v)\|_{L^\infty}\leq& \big\|\Delta_q\nabla(-\Delta)^{-1}\big(T_{\partial_jv_i}\partial_iu_j\big)\big\|_{L^\infty}\nonumber\\
\leq& C\sum_{|k-q|\leq5}\big\|\Delta_q\nabla(-\Delta)^{-1}\big(S_{k-1}\partial_jv_i\Delta_k\partial_iu_j\big)\big\|_{L^\infty}\nonumber\\
\leq& C\sum_{|k-q|\leq5}2^{-q}\big\|S_{k-1}\partial_jv_i\big\|_{L^\infty}\big\|\Delta_k\partial_iu_j\big\|_{L^\infty}\nonumber\\
\leq&C\Theta(q+2)2^{q\varepsilon} \min\Big\{\|u\|_{B^{-\varepsilon}_{\infty,\infty}}
\|v\|_{Y^\Theta_{\rm Lip}},\,\|v\|_{B^{-\varepsilon}_{\infty,\infty}}\|u\|_{Y^\Theta_{\rm Lip}}\Big\}.
\end{align*}
The remainder term $B_3(u,v)$ can be bounded by
\begin{align*}
&\|\Delta_q\partial_i\partial_j\nabla(-\Delta)^{-1}({\rm I_d}-\Delta_{-1})\mathcal{R}(u_j,v_i)\|_{L^\infty}\nonumber\\
\leq&\|\widetilde{\dot{\Delta}}_q\partial_i\partial_j\nabla(-\Delta)^{-1}\mathcal{R}(u_j,v_i)\|_{L^\infty}\nonumber\\
\leq&C\sum_{k\geq q-2}2^{q}\|\Delta_ku_j\|_{L^\infty}\|\widetilde\Delta_kv_i\|_{L^\infty}\nonumber\\
\leq&C2^{q\epsilon}\sum_{k\geq q-2}2^{(q-k)(1-\epsilon)}2^{k}2^{-k\epsilon}\|\Delta_ku_j\|_{L^\infty}\|\widetilde\Delta_kv_i\|_{L^\infty}\nonumber\\
\leq&C\Theta(q+2)2^{q\epsilon}\sum_{ k\geq q-2}\frac{\Theta(k+2)}{\Theta(q+2)}2^{(q-k)(1-\epsilon)} \min\Big\{\|u\|_{B^{-\varepsilon}_{\infty,\infty}}\|v\|_{ \overline{Y^\Theta_{\rm Lip}}},\,\|v\|_{B^{-\varepsilon}_{\infty,\infty}}\|u\|_{\overline{Y^\Theta_{\rm Lip}}}\Big\}.
\end{align*}
By using \eqref{eq.delta-2}, we can deduce that
\begin{align*}
\|\Delta_qB_3(u,v)\|_{L^\infty}\leq C\Theta(q+2)2^{q\varepsilon} \min\Big\{\|u\|_{B^{-\varepsilon}_{\infty,\infty}}
\|v\|_{Y^\Theta_{\rm Lip}},\,\|v\|_{B^{-\varepsilon}_{\infty,\infty}}\|u\|_{Y^\Theta_{\rm Lip}}\Big\}.
\end{align*}
Since $\theta E_d\in L^1(\RR^d)$,  we get
\begin{align*}
\|\Delta_qB_4(u,v)\|_{L^\infty}\leq&\big\|\Delta_q\big(\theta E_d\ast\nabla\partial_i\partial_j\Delta_{-1}\mathcal{R}(u_j,v_i)\big)\big\|_{L^\infty}\nonumber\\
\leq& C\|\theta E_d\|_{L^1}\|\nabla\partial_i\partial_j\Delta_{-1}\mathcal{R}(u_j,v_i)\|_{L^\infty}\nonumber\\
\leq& C\|\theta E_d\|_{L^1}\sum_{k\geq-1}\|\Delta_ku_j\|_{L^\infty}\|\Delta_k v_i\|_{L^\infty}\nonumber\\
\leq& C\|\theta E_d\|_{L^1}\sum_{k\geq-1}2^{-k(1-\epsilon)}2^{k}2^{-k\epsilon}\|\Delta_ku_j\|_{L^\infty}\|\Delta_k v_i\|_{L^\infty}\nonumber\\
\leq&C\Theta(q+2)2^{q\varepsilon} \min\Big\{\|u\|_{B^{-\varepsilon}_{\infty,\infty}}
\|v\|_{Y^\Theta_{\rm Lip}},\,\|v\|_{B^{-\varepsilon}_{\infty,\infty}}\|u\|_{Y^\Theta_{\rm Lip}}\Big\}.
\end{align*}
Finally, the fact that $\nabla\partial_i\partial_j\widetilde E_d\in L^1$ enables us to conclude that
\begin{align*}
\|\Delta_qB_5(u,v)\|_{L^\infty}\leq &\big\|\Delta_q\big(\nabla\partial_i\partial_j\widetilde E_d\ast\Delta_{-1}\mathcal{R}(u_j,v_i)\big)\big\|_{L^\infty}\nonumber\\
\leq& C\|\nabla\partial_i\partial_j\widetilde E_d\|_{L^1}\sum_{k\geq-1}\|\Delta_ku_j\|_{L^\infty}\|\Delta_k v_i\|_{L^\infty}\nonumber\\
\leq& C\|\nabla\partial_i\partial_j\widetilde E_d\|_{L^1}\sum_{k\geq-1}2^{-k(1-\varepsilon)}2^{k}2^{-k\varepsilon}\|\Delta_ku_j\|_{L^\infty}\|\Delta_k v_i\|_{L^\infty}\nonumber\\
\leq&C\Theta(q+2)2^{q\varepsilon} \min\Big\{\|u\|_{B^{-\varepsilon}_{\infty,\infty}}
\|v\|_{Y^\Theta_{\rm Lip}},\,\|v\|_{B^{-\varepsilon}_{\infty,\infty}}\|u\|_{Y^\Theta_{\rm Lip}}\Big\}.
\end{align*}
This ends the proof.
\end{proof}

\section*{Acknowledgments}

 The authors thank the  referees and the associated editor for their
invaluable comments  which helped improve the paper
greatly.  This work is supported in part by the National Natural Science Foundation of China
 under grant  No.11671045, No.11671047,  No.11501020 and No.11831004
 .



\begin{thebibliography}{9}

\bibitem{BCD11}H. \textsc{Bahouri}, J.-Y. \textsc{Chemin} and R. \textsc{Danchin}, {Fourier Analysis and Nonlinear Partial Differential Equations}, Grundlehren dermathematischen Wissenschaften \textbf{343}, Springer-Verlag, 2011.

\bibitem{BKM} J. T. \textsc{Beale}, T. \textsc{Kato} and A. \textsc{Majda}, \emph{Remarks on the Breakdown of Smooth Solutions for the 3-D Euler Equations}, Commun. Math. Phys., \textbf{94} (1984), 61--66.

\bibitem{B-K-1}F. \textsc{Bernicot} and S. \textsc{Keraani},
\emph{Sharp constants for composition with a bi-Lipschitz mesure-preserving map},   Math. Res. Lett. {\bf 21} (2014), 937-952.

\bibitem{B-K-2}F. \textsc{Bernicot} and S. \textsc{Keraani},
\emph{On the global wellposedness of the 2D Euler equations for a large class of Yudovich type data}, Ann. Sci. ¨¦c. Norm. Sup¨¦r.  {\bf 47}(2014), 559-576.

\bibitem{B-H13}F. \textsc{Bernicot} and T. \textsc{Hmidi}, \emph{On the global wellposedness for Euler equations with unbounded vorticity},  Dyn. Partial Differ. Equ.
{\bf 12}(2015), 127-155.


\bibitem{Ch} D. \textsc{Chae}, \emph{Weak solutions of 2D Euler equations with initial vorticity in $L\ln L$}, J. Differential Equation, {\bf 103} (1993), 323-337.


\bibitem{Chemin} J.Y. \textsc{Chemin},
\emph{Fluides parfaits  incompressibles},  Ast\'erisque, {\bf 230}, 1995.


\bibitem{Del91} J.-M. \textsc{Delort}, \emph{Existence de nappes de tourbillon en dimension deux}, J. Amer. Math. Sot., Vol. {\bf 4} (1991), 553-586.

\bibitem{Dip87} R. J. \textsc{Diperna} and A. Majda, \emph{Oscillations and concentrations in weak solutions of the incomprssible fluid equations}, Commun. Math. Phys. \textbf{108}, (1987), 667-689.


\bibitem{FS}C. \textsc{Fefferman} and E. M. \textsc{Stein}, \emph{${H}^p$ spaces of several variables},
 Acta Math.  \textbf{129} (1971), 137-193.

\bibitem{GMO88} Y. \textsc{Giga}, T. \textsc{Miyakawa} and H. \textsc{Osada}, \emph{Two-dimensional Navier-Stokes flow with measures as initial vorticity}. Arch. Ration. Mech. Anal. {\bf 104} (1988), 223-250.

\bibitem{HK08} T. \textsc{Hmidi} and S. \textsc{Keraani}, \emph{Incompressible viscous flows in borderline Besov spaces}, Arch. Ration. Mech. Anal. {\bf 189} (2008), 283-300.

\bibitem{KP86} T. \textsc{Kato} and G. \textsc{Ponce}, \emph{Well-posedness of the Euler and Navier-Stokes equations in the Lebesgue spaces $L^p_s(\RR^2)$},
Rev. Mat. Iberoamericana, {\bf 2} (1986), 73-88.

\bibitem{FLX} M. C. \textsc{Lopes Filho}, H. J. \textsc{Nussenzveig Lopes} and Z. \textsc{Xin}, \emph{Existence of
vortex sheets with reflection symmetry in two space dimensions}, Arch. Ration. Mech. Anal., {\bf 158} (2001), 235-257.

\bibitem{LRbook}
P. G. \textsc{Lemari\'e-Rieusset},  Recent Developments in the Navier--Stokes
Problem, Chapman \& Hall/CRC Press, Boca Raton, 2002.

\bibitem{MAj-book} A. \textsc{Majda} and A. L. \textsc{Bertozzi},  {Vorticity and incompressible flow}, Cambridge Texts
in Applied Mathematics, vol. {\bf 27}, Cambridge University Press, Cambridge, 2002.

\bibitem{Miao-book04}C. \textsc{Miao}, {Harmonic Analysis with Application to Partial Differential Equations}, 2nd edition,
    No.89. Science Press, Beijing, 2004.

\bibitem{MWZ2012} C. \textsc{Miao}, J. \textsc{Wu} and Z. \textsc{Zhang}, {Littlewood-Paley Theory and Applications to Fluid Dynamics Equations},
 Monographs on Modern pure mathematics, No.142. Science Press, Beijing, 2012.

\bibitem{Peano} G. \textsc{Peano}, \emph{Demonstration de l¡¯int¨¦grabilit¨¦ des ¨¦quations diff¨¦rentielles ordinaires, Mathematische Annalen}, {\bf 37} (1890) 182-228.


\bibitem{Ser-0} P. \textsc{Serfati},
\emph{Pertes de  r\'{e}gularit\'{e} le laplacien et l'\'{e}quation d'Euler sui $\mathbb R^n$},
priprint.15,pp., 1994.


\bibitem{Ser-1} P. \textsc{Serfati},
\emph{Solutions $C^\infty$ en temps, n. log lipschitz born\'ees en espace et \'equation d'Euler},
C. R. Acad. Sci. Paris, s\'er. I Math. \textbf{320} (1995), 555-558.

\bibitem{Ser-2} P. \textsc{Serfati}, \emph{Structures holomorphes \`a faible r\'egularit\'e spatiale en m\'ecanique des fluides}, J. Math. Pures Appl.
{\bf 74} (1995), 95-104.

\bibitem{Spanne} S. \textsc{Spanne}, \emph{Some function spaces defined using the mean oscillation over cubes}, Ann. Scuola Norm. Sup. Pisa,  \textbf{19} (1965), 593-608.



\bibitem{Tan04}Y. \textsc{Taniuchi}, \emph{Uniformly local $L^p$ estimate for 2D Vorticity Equation and its application to Euler equations with initial
Vorticity in {\rm BMO}}, Commun. Math Phys., {\bf 248} (2004), 169-186.

\bibitem{Tan10} Y. \textsc{Taniuchi}, T. \textsc{Tashiro} and T. \textsc{Yoneda}, \emph{On the two-dimensional Euler equations with spatially almost periodic initial data}, J. Math. Fluid Mech. {\bf 12} (2010),  594-612.


\bibitem{Taylor} M. \textsc{Taylor}, {Tools for PDE, Pseudodifferential Operators, Paradifferential Operators, and Layer Potentials}, Amerian Mathematical Socity, 2000.

\bibitem{Vi98}
M. \textsc{Vishik}, \emph{Hydrodynamics in Besov Spaces}, Arch. Ration. Mech. Anal.  \textbf{145} (1998), 197-214.

\bibitem{Vi00}M. \textsc{Vishik},
\emph{Incompressible flows of an ideal fluid with unbounded vorticity}, Commun. Math. Phys. \textbf{213}, (2000), 697-731.

\bibitem{Vi99}M. \textsc{Vishik}, \emph{Incompressible
flows of an ideal fluid with vorticity in borderline spaces of Besov type}, Ann. Sci. cole Norm. Sup. (4) {\bf 32} (1999),  769-812.

\bibitem{Wol} W. \textsc{Wolibner},\emph{ Un th¨¦or¨¨me sur l'existence du mouvement plan d'un fluide parfait, homog¨¨ne, incompressible, pendant un temps infiniment longue}, Math. Z. {\bf 37} (1933), 698-726.

\bibitem{Y1} V. I. \textsc{Yudovich}, \emph{Nonstationary flow of an ideal incompressible liquid}, Zh. Vych. Mat., {\bf 3} (1963), 1032-1066.

\bibitem{Y-1995}V. I. \textsc{Yudovich},
\emph{Uniqueness theorem for the basic nonstationary problem in the dynamics of an ideal incompressible fluid}, Mathematical Research Letters, {\bf 2}(1995), 27-38

\bibitem{Func}A. \textsc{Kufner}, O. \textsc{John} and S. \textsc{Fu\v c\'  ik}, Functional spaces, Prague, Academia, 1977.
\end{thebibliography}
\end{document}